\documentclass[twoside,11pt]{article}

\usepackage{blindtext}

\usepackage{jmlr2e_arxiv}
\usepackage{amsmath}
\newtheorem{prop}[theorem]{Proposition}

\newtheorem{assum}[theorem]{Assumption}
\newtheorem{alg}[theorem]{Algorithm}

\usepackage{numprint}
\npthousandsep{,}

\usepackage{pgfplots}
\usepackage{float}
\usepackage{alg}
\usepackage{enumitem}
\pgfplotsset{compat=1.15}

\newcommand{\sign}{\operatorname{sign}}
\newcommand{\argmax}{\operatorname{argmax}}

\newcommand{\diag}{\operatorname{diag}}

\newcommand{\PP}{\mathbb{P}}
\newcommand{\E}{\mathbb{E}}

\newcommand{\Cov}{\operatorname{Cov}}

\newcommand{\Var}{\mathrm{Var}}

\newcommand{\inprob}{\stackrel{p}{\rightarrow}}
\newcommand{\indist}{\stackrel{d}{\rightarrow}}

\newcommand{\avec}{\mathbf{a}}
\newcommand{\bvec}{\mathbf{b}}
\newcommand{\cvec}{\mathbf{c}}

\newcommand{\evec}{\mathbf{e}}

\newcommand{\hvec}{\mathbf{h}}

\newcommand{\rvec}{\mathbf{r}}
\newcommand{\svec}{\mathbf{s}}
\newcommand{\tvec}{\mathbf{t}}

\newcommand{\vvec}{\mathbf{v}}
\newcommand{\wvec}{\mathbf{w}}
\newcommand{\xvec}{\mathbf{x}}
\newcommand{\yvec}{\mathbf{y}}
\newcommand{\zvec}{\mathbf{z}}

\newcommand{\zerovec}{\mathbf{0}}
\newcommand{\onevec}{\mathbf{1}}

\newcommand{\Dmat}{\mathbf{D}}

\newcommand{\Imat}{\mathbf{I}}

\newcommand{\Mmat}{\mathbf{M}}

\newcommand{\Pmat}{\mathbf{P}}

\newcommand{\Rmat}{\mathbf{R}}

\newcommand{\Vmat}{\mathbf{V}}

\newcommand{\Xmat}{\mathbf{X}}

\newcommand{\calA}{\mathcal{A}}
\newcommand{\calB}{\mathcal{B}}

\newcommand{\calE}{\mathcal{E}}

\newcommand{\calL}{\mathcal{L}}

\newcommand{\calN}{\mathcal{N}}
\newcommand{\calO}{\mathcal{O}}

\newcommand{\bbeta}{\boldsymbol{\beta}}

\newcommand{\bpi}{\boldsymbol{\pi}}

\newcommand{\berror}{\boldsymbol{\varepsilon}}
\newcommand{\bmu}{\boldsymbol{\mu}}
\newcommand{\bSigma}{\boldsymbol{\Sigma}}

\newcommand{\LAR}{\operatorname{Lar}}

\usepackage{lastpage}
\firstpageno{1}

\begin{document}

\begin{center}
{\bf \Large New explanations and inference for least angle regression}
\end{center}

\vspace{6pt}

\begin{center}
\setlength{\tabcolsep}{.5in}
    \begin{tabular}{ll}
         {\bf Karl B. Gregory}& {\bf Daniel J. Nordman}  \\
         {\it Department of Statistics}& {\it Department of Statistics} \\
         {\it University of South Carolina} &{\it Iowa State University}\\
         {\it Columbia, SC 29225, USA}&{\it Ames, IA 50011, USA}\\
         gregorkb@stat.sc.edu & dnordman@iastate.edu
    \end{tabular}
\end{center}

\vspace{12pt}

\begin{abstract}%   <- trailing '%' for backward compatibility of .sty file
Efron et al.~(2004) introduced least angle regression (LAR)
as an algorithm for linear predictions, intended as an  alternative
to forward selection with connections to penalized regression.
However, LAR has remained somewhat of a ``black box,''    where some basic behavioral properties of LAR output are not well understood, including an appropriate termination point for the algorithm.
We   provide a novel framework for inference
with LAR, which also allows LAR to be understood from new perspectives with several newly developed mathematical properties.
The LAR algorithm at a data level can viewed as estimating a population counterpart ``path'' that organizes a response mean  along  regressor variables which are ordered according to
a decreasing series of population ``correlation'' parameters; such parameters are shown to have meaningful interpretations for explaining variable contributions whereby zero correlations denote  unimportant variables.
In the output of  LAR,  estimates of all non-zero population correlations turn out to have independent
 normal distributions for use in inference, while estimates of zero-valued population correlations
have a certain non-normal joint distribution.  These properties help to provide a formal rule for stopping the LAR algorithm. While
 the standard bootstrap for regression can fail for
   LAR, a modified bootstrap provides
a practical and formally justified tool for interpreting the entrance  of variables and   quantifying uncertainty in estimation.
The LAR inference method  is studied through simulation and illustrated with data examples. 
\end{abstract}

\vspace{12pt }


\section{Introduction}

 \cite{efron2004least} introduced the least angle regression (LAR) algorithm as an alternative approach to least squares regression with forward selection of variables.  The main idea is that, as the algorithm proceeds,
 new columns of a design matrix $\Xmat$  are admitted in a certain stepwise fashion, whereby linear predictions of a response vector $\yvec$ are then sequentially updated based on a set of so-called
   \textit{active} columns of $\Xmat$.    \cite{efron2004least} presented
   LAR  in connection to a suite of methods for obtaining a sparse estimator of  a vector of regression coefficients $\bbeta$ when the mean of $\yvec$ is equal to a linear combination $\Xmat \bbeta$.   Namely, because LAR admits columns of $\Xmat$ into an active set sequentially, this algorithm can
   be  stopped in stages, resulting in a type of sparse estimator of $\bbeta$ akin to the Lasso estimator or  the forward stagewise regression estimator \citep{tibshirani1996regression,tibshirani2015general}.  In fact, the LAR algorithm can be modified to return solutions of Lasso or
forward stagewise regression estimators \citep{hastie2007forward,tibshirani2013lasso,hesterberg2008least}, and this represents the context in which
LAR has received most attention. Perhaps due to these connections with sparse estimators of $\bbeta$, literature on inference with LAR has largely viewed it as a model selection algorithm, focusing either on post-selective inference conditionally on the current active set when LAR (or its Lasso modification) is stopped midway \citep{taylor2014post,lee2016exact} or on sequential inference, which aims to test, at each step, whether the current active set contains every important variable \citep{lockhart2014significance,su2018first,g2016sequential,azais2018power,azais2022multiple}.

In spite of these works,  however, our understanding of the LAR algorithm itself has not improved  much since the initial work of \cite{efron2004least},
and  texts have likewise noted that some main operating mechanics of LAR remain unclear (cf.~Ch.~4.1, \citealp{RV2025}). 
Furthermore, beyond the perspective of a data  algorithm, 
little has been known about inference on the quantities directly estimated by  LAR over the past 20+ years, relating to how to possibly understand the output of LAR toward estimation of some larger truth in regression.
For example, while LAR does indeed produce estimates of regression coefficients $\bbeta$,  the algorithm itself is not actually prescribed in terms of such coefficients and instead  outputs   variable orderings combined with  data-based ``correlations." 
  It turns out that new  
explanations and inference are possible in the LAR framework, but require a different perspective  and development for  better insight into what LAR  is mechanically doing and fundamentally estimating    at each step of the algorithm.   Namely, the output of the LAR algorithm applied to observed data $\yvec$ can be conceptualized as a noisy estimate  of the LAR algorithm
applied to the unobservable response mean $\bmu$. A population view of LAR is  introduced here and   
reveals that LAR aims simply to explain $\bmu$ as a sum of contributions after orthogonally re-writing the regressor variables (i.e., after a Gram-Schmidt orthogonalization of columns of $\Xmat$), where the order of regressor variables  is determined by a series of  {\it step correlations}
as the algorithm proceeds.  Such correlations determine both  the sequence and contribution size of variables   toward  explaining the true signal  $\bmu$, and
step correlations decrease so that
only those variables with non-zero correlations are important for $\bmu$.
The LAR algorithm applied to data $\yvec$ can then be shown to estimate the population path of step correlations in a way that has useful interpretation and discernible properties for inference.
We are not aware of any work that similarly targets   such interpretation  and inference with the direct output of LAR.

 In Section \ref{sec:lar}, we present the details of the LAR sample algorithm as described in \cite{efron2004least}, offer a revision to one key step to better clarify how LAR determines variable entrances, and then explain the notion of a prototypical  LAR {\it population path}  (i.e., $\LAR(\Xmat,\bmu)$) as a target for inference.     Section~\ref{sec:understanding} reviews some known facts about LAR and then provides new  results that advance insight about how LAR decomposes a response   into a series of projections onto the column space of $\Xmat$ and how LAR chooses the order of variable entry.  Section \ref{sec:Tk} then presents two main distributional results regarding LAR estimates, which establish the consistency and limit distributions from a LAR sample path  (i.e., $\LAR(\Xmat,\yvec)$)  involving a noisy response $\yvec$.  A formal way to estimate the termination point of LAR  in practice is also introduced. 
Section~\ref{sec:practicalinference} gives methods for making inferences on the LAR population path based on our distributional results.  Specifically, we introduce a bootstrap procedure for constructing confidence intervals for the step correlations in the LAR population path as well as for the coefficients giving the contributions of each column of $\Xmat$ to the linear prediction   at each   path step.
The bootstrap is useful as  outputted LAR correlations from a sample can have complicated distributions to approximate (e.g., involving non-normal  mixture distributions).  Additionally, a naive implementation of bootstrap  from standard regression can provably  fail, while the proposed bootstrap is simple and theoretically valid. As numerical support, Section~\ref{sec:sim} presents a simulation study, while Section~\ref{sec:illu} presents real data analyses to illustrate LAR inference.   Section~\ref{sec:midpathties} discusses some further technical details regarding LAR, in particular the notion of ``midpath ties.''     
Section~\ref{sec:discussion} then offers concluding discussion.  Proofs of the main results appear in the Appendix.

For clarity in studying LAR and establishing estimation results, we use the same   theoretical framework  as in
\cite{efron2004least}'s introduction of LAR, 
focusing primarily on the $p < n$ case, where $p$ is the number of columns in $\Xmat$ and $n$ the number of rows.  
This  setting provides a non-trivial starting point for explaining mechanics and for developing statistical inference with LAR.  Additionally, we formally establish  bootstrap   to help quantify uncertainty with LAR estimates, which is a challenging task.  Even for the more commonly studied Lasso method in regression, which has the benefit of not involving a purely algorithmic description like LAR,  bootstrap theory  currently exists only for the $p<n$ case \citep{chatterjee2011bootstrapping,giurcanu2019bootstrapping}. Our work, though, can provide   groundwork for  developments in the $p > n$ setting and for  future investigations with LAR, including   modifications.

% Throughout the paper, we will assume that we observe $\yvec = \bmu + \berror$, where $\berror$ represents some noise, as well as an $n \times p$ design matrix $\Xmat = [\xvec_1,\dots,\xvec_p]$ with $\|\xvec_j\|_2 = 1$ for all $j =1,\dots,p$.  We assume $p < n$ throughout.

%\subsection{Notation} For a $d\times d$ symmetric matrix $\Vmat$ and sets of indices $\calA,\calB \subset\{1,\dots,d\}$, we denote by $\Vmat_{\calA,\calB}$ the matrix %constructed by keeping rows and columns of $\Vmat$ having indices in $\calA$ and $\calB$, respectively. We set $\Vmat_{\calA|\calB} = \Vmat_{\calA,\calA} - %\Vmat_{\calA,\calB}(\Vmat_{\calB,\calB})^{-1}\Vmat_{\calB,\calA}$.
%If $\vvec$ is a column vector, we denote by $(\vvec)_j$ entry $j$ of $\vvec$ and by $(\vvec)_\calA$ the vector formed by keeping the entries in $\vvec$ with indices in %$\calA$. We denote by $\|\vvec\|^2$ the sum of the squared entries in $\vvec$. We denote by $\Pmat_{\Vmat}$ the orthogonal projection onto the column space of the %matrix $\Vmat$. We write the expectation and variance operators as $\E$ and $\V$, respectively. We will denote by $\calN(\vvec,\Vmat)$ the multivariate Normal %distribution with mean vector $\vvec$ and covariance matrix $\Vmat$ and by $\chi^2_\nu$ the chi-squared distribution with degrees of freedom $\nu$. We use $\Imat$, %$\zerovec$, and $\onevec$ to denote an identity matrix, a column vector of zeroes, and a column vector of ones, respectively, of appropriate dimension.

\section{The least angle regression (LAR) algorithm}
\label{sec:lar}
For context, we first give an overview of the original LAR algorithm of \cite{efron2004least} as described for
an $n\times 1$ response vector $\yvec$ and an $n \times p$ matrix $\Xmat$ with columns denoting $p<n$ linearly independent regressor variables.
Typically, the columns of $\Xmat$ are scaled to have unit norm, which we assume here too.
The first steps of the LAR algorithm are roughly as follows.  With an initialization  $\hat \yvec_0 \equiv 0$,
compute a vector of so-called ``correlations'' (technically dot-products here)  as $\hat \cvec_1 \equiv \Xmat^T(\yvec - \hat \yvec_0)$ and find the column (say $\xvec_1$) of $\Xmat$ with the greatest absolute correlation, say $\hat{C}_1$.  Define the 1st  \textit{active set} $\hat \calA_1$ to be the index of this $\Xmat$-column
and let $\hat \avec_1$
denote $\xvec_1$ upon signing this to have positive correlation  with the residual $\yvec - \hat \yvec_0$.  Next, create an updated prediction
$\hat \yvec_1 \equiv \hat \yvec_0 + \hat{\gamma}_1 \hat \avec_1$ by moving from $\hat \yvec_0$ in the direction of $\hat \avec_1$ by the smallest positive amount, denoted by $\hat{\gamma}_1>0$, whereby another column of $\Xmat$ (say $ \xvec_2$) has a correlation with the updated residual $\yvec - \hat \yvec_1$ which matches that of $\xvec_1$ with $\yvec - \hat \yvec_1$ in absolute value, say $\hat{C}_2$.
Then, the 2nd step active set $\hat{\mathcal{A}}_2$ corresponds to column indices of $\Xmat$ prescribing
$\xvec_1,\xvec_2$ and we define   $\hat \avec_2$ to be a (unit) vector making equal angles with $\xvec_1,\xvec_2$ (after signing these to have positive correlation with $\yvec - \hat \yvec_1$).    An updated prediction
$\hat \yvec_2 \equiv \hat \yvec_1 + \hat{\gamma}_2 \hat \avec_2$ follows by shifting $\hat \avec_2$ by the smallest positive $\hat{\gamma}_2>0$, whereby a further column of $\Xmat$ (say $  \xvec_3$) has a correlation with $\yvec - \hat \yvec_2$ equaling that of $\xvec_1$ (or $\xvec_2$) with $\yvec - \hat \yvec_1$ in absolute value, say $\hat{C}_3$.  The 3rd step active set $\hat{\mathcal{A}}_3$  corresponds to the column indices prescribing $\xvec_1,\xvec_2,\xvec_3$.  In this fashion, the algorithm proceeds by updating predictions stepwise in a direction equi-angular to all currently ``active'' regression variables until another (i.e., non-active) variable becomes active through a matching residual correlation.  The sequence of correlations $\hat C_k$ also decreases to zero over increasing steps.

Algorithm \ref{alg:lar_efron} gives the complete details of the LAR algorithm as presented in \cite{efron2004least}.
For a generic vector $\zvec$, let $(\zvec)_j$ denote its $j$th component in the following.

\begin{alg}[Data-level Least Angle Regression, $\LAR(\Xmat,\yvec)$, \cite{efron2004least}]\label{alg:lar_efron}
Given an $n \times p$ matrix $\Xmat$ with linearly independent columns having unit norm and an $n \times 1$ response vector $\yvec$, initialize $\hat \yvec_0 \equiv \zerovec$. Then update predictions from $\hat \yvec_{k-1}$ to $\hat \yvec_{k}$ as follows:
\begin{enumerate}
    \item Define $\hat C_k \equiv \max\{|(\hat \cvec_k)_j|\}$, where $\hat \cvec_k \equiv \Xmat ^T(\yvec - \hat \yvec_{k-1})$.
    \item If $\hat C_k =0$ then set $\hat \yvec_k  = \hat \yvec_{k-1}$ and stop; otherwise continue.
    \item Define the active set $\hat \calA_k \equiv \{j : |(\hat \cvec_k)_j| = \hat C_k\}$.
    % \item $\Xmat_k = \Xmat_{\calA_k}\diag(\sign((\cvec_k)_{\calA_k}))$
    \item Define $\hat \Xmat_k \equiv [\sign((\hat \cvec_k)_j) \xvec_j , ~ j \in \hat \calA_k]$ from signed active columns of $\Xmat$.
    % \Xmat_{\calA_k}\diag((\svec_k)_{\calA_k})$, where $\svec_k = \sign(\cvec_k)$.
    \item Define equi-angular vector $\hat \avec_k \equiv \hat A_k \hat \Xmat_k (\hat \Xmat_k^T\hat \Xmat_k)^{-1}\onevec$ with angle $\hat A_k^{-2} \equiv \onevec^T (\hat \Xmat_k^T\hat \Xmat_k)^{-1}\onevec$, where $\onevec$ denotes a vector of ones.
    \item If $|\hat \calA_k| < p$, compute $\hat \wvec_k \equiv \Xmat^T\hat \avec_k$ and set
    \begin{equation}
    \label{eqn:efrongamma}
    \hat \gamma_k \equiv \underset{j \notin \hat \calA_k}{\operatorname{min}^{+}} \Big\{\frac{\hat C_k - (\hat \cvec_k)_j}{\hat A_k - (\hat \wvec_k)_j},\frac{\hat C_k + (\hat \cvec_k)_j}{\hat A_k + (\hat \wvec_k)_j}\Big\},
    \end{equation}
    where  $\min^{+}\{\cdot,\cdot\}$ in  \eqref{eqn:efrongamma} denotes a minimum taken only over positive arguments.\\
    % \begin{equation}
    % \label{eqn:gammak}
    % \gamma_k = \min_{j \in \calA_k^C}\Big\{\frac{C_k - (\cvec_k)_jr_{kj} }{A_k - (\wvec_k)_jr_{kj} }\Big\}, \quad r_{kj} = \sign \Big((\cvec_k)_j - \frac{C_k}{A_k} (\wvec_k)_j \Big ).
    % \end{equation}
    Otherwise, if $|\hat \calA_k| = p$, set $\hat \gamma_k \equiv \hat C_k / \hat A_k$.
    \item Update $\hat \yvec_k \equiv \hat \yvec_{k-1} + \hat \gamma_k \hat \avec_k$.
\end{enumerate}
\end{alg}

The values $\hat C_k$ above, which we will call \textit{step correlations}, represent the maximal absolute dot-products between the columns of $\Xmat$ and the residual vectors $\hat \yvec - \hat \yvec_{k-1}$ at stage $k \leq p$ in the algorithm.  Strictly speaking,
the  $\hat C_k$ are not quite correlations, as these are not scaled to lie in the interval $[-1,1]$, but we still refer to these as correlations in the tradition of \cite{efron2004least}.  Such  correlations determine which columns of $\Xmat$ enter the \textit{active set} $\hat \calA_k$ on each step of the LAR algorithm, where the active set again denotes the particular (signed) columns $\hat \Xmat_k$  of $\Xmat$ contributing to an updated predictor $\hat \yvec_k$.   This updating involves computation of  a \textit{weight factor} $ \hat \gamma_k>0$ as well as
a so-called \textit{equi-angular vector} $\hat \avec_k$, where  $\hat \avec_k$ has unit-norm and makes equal angles $\hat \Xmat_k^T\hat \avec_k = \hat A_k \onevec$ with every column of $\hat \Xmat_k$     at step $k$.    We will refer to the terms $\hat A_k$ as \textit{angles} in the fashion of \cite{efron2004least}, though each $\hat A_k$ is technically the cosine of the angle between $\hat \avec_k$ and any column of $\hat \Xmat_k$.   If all variables become active $|\hat \calA_k| = p$, LAR is designed to return
the standard least squares estimator of $\yvec$ (cf.~\cite{efron2004least}).  Note that we assume $p < n$, though
one can replace  $p$ with $\min\{p,n-1\}$ in line 6 of Algorithm \ref{alg:lar_efron} to accommodate the possibility of $p > n$, where ``$n-1$'' comes from the assumption that $\yvec$ as well as the columns of $\Xmat$ have been centered to have mean zero (whereupon only $n -1$ columns of $\Xmat$ are needed to reconstruct $\yvec$). Such centering, though conventional, is not necessary to our results.
% (we assume $p < n$, though \cite{efron2004least} replaces $p$ with $n-1$ in line 6 of Algorithm \ref{alg:lar_efron} in the $p>n$ case, where the $n - 1$ ). We remark that one often centers $\yvec$ and the columns of $\Xmat$ to have mean zero, but this is not necessary to our results.  
Other related descriptions of the LAR algorithm   appear in \cite{khan2007robust} and \cite{taylor2014post}.

As a first result, we make a revision to the standard computation of weight $\hat \gamma_k$ from \eqref{eqn:efrongamma} in the algorithm which will later be useful in proving our results. In particular, the updated expression for $\hat \gamma_k$   clarifies exactly how variables must be signed to enter LAR active sets.

% Note that $\hat \gamma_k$ is the distance we can move $\hat \yvec_{k-1}$ in the direction of $\hat \avec_k$ until a column of $\Xmat$ not in the active set has an absolute correlation with the residual equal to or greater than that of any column of $\Xmat$ in the active set.

\begin{lemma}
\label{lem:gammak}
A definition of $\hat \gamma_k$ equivalent to \eqref{eqn:efrongamma} is $\hat \gamma_k = \min_{j \notin \calA_k} \{\hat \gamma_{k,j}\}$, where
    \begin{equation}
    \label{eqn:hatgammak}
    \hat  \gamma_{k,j} \equiv \frac{\hat C_k - (\hat \cvec_k)_j \hat r_{k,j} }{ \hat A_k - (\hat \wvec_k)_j \hat r_{k,j} }>0, \quad \hat r_{k,j} \equiv \sign \Big((\hat \cvec_k)_j - \frac{\hat C_k}{\hat A_k} (\hat \wvec_k)_j \Big).
    \end{equation}
\end{lemma}
% \rd{We note that an expression similar to \eqref{eqn:hatgammak} appears in \cite{tibshirani2013lasso} in the context of the modification to the LAR algorithm which returns the Lasso solution path.}

We will be interested in studying how the data-level version of $\LAR(\Xmat,\yvec)$, as outlined  in Algorithm~\ref{alg:lar_efron} based on a
 design matrix $\Xmat$ and a sample response vector $\yvec$, has a {\it population-level} counterpart of LAR  as applied to a design matrix $\Xmat$ and a mean response vector $\bmu$.
  Algorithm \ref{alg:lar} describes this population rendition of LAR (say $\LAR(\Xmat,\bmu)$), which aims to re-write
  the population mean response $\bmu$ through a series of step-wise linear updates (say $\bmu_k$) from regressors in the design matrix $\Xmat$;
  this algorithm likewise implements
     an analog of \eqref{eqn:hatgammak}
   for simplifying population-level expressions of weighting factors.

\begin{alg}[Population-Level Least Angle Regression, $\LAR(\Xmat,\bmu)$]\label{alg:lar}
Given an $n \times p$ matrix $\Xmat$ with columns having unit norm and an $n \times 1$ response vector $\bmu$, initialize $\bmu_0 \equiv \zerovec$. Then update approximations from $\bmu_{k-1}$ to $\bmu_{k}$ as follows:
\begin{enumerate}
    \item Define $C_k \equiv \max\{|(\cvec_k)_j|\}$, where $\cvec_k  \equiv \Xmat ^T(\bmu - \bmu_{k-1})$.
    \item If $ C_k =0$ then set $\bmu_k  = \bmu_{k-1}$ and stop; otherwise continue.

    \item Define the active set  $\calA_k  \equiv \{j : |(\cvec_k)_j| = C_k\}$.
    \item Define $\Xmat_k \equiv [\sign((\cvec_k)_j) \xvec_j , ~ j \in  \calA_k]$ from signed active columns of $\Xmat$.

    \item Define equi-angular vector $\avec_k \equiv A_k \Xmat_k (\Xmat_k^T\Xmat_k)^{-1}\onevec$ with angle $A_k^{-2} \equiv \onevec^T (\Xmat_k^T\Xmat_k)^{-1}\onevec$, where $\onevec$ denotes a vector of ones.
    \item If $|\calA_k| < p$, compute $\wvec_k  \equiv \Xmat^T\avec_k$ and set $\gamma_k \equiv \min_{j \notin \calA_k} \{\gamma_{k,j}\}$, where
    \begin{equation}
    \label{eqn:gammak}
    \gamma_{k,j} = \frac{C_k - (\cvec_k)_jr_{k,j} }{A_k - (\wvec_k)_jr_{k,j} }>0, \quad r_{k,j} \equiv \sign \Big((\cvec_k)_j - \frac{C_k}{A_k} (\wvec_k)_j \Big ).
    \end{equation}
    Otherwise, if $|\calA_k| = p$, set $\gamma_k \equiv C_k / A_k$.
    \item Update $\bmu_k = \bmu_{k-1} + \gamma_k \avec_k$.
\end{enumerate}
\end{alg}

 Algorithm \ref{alg:lar}  matches the original formulation of LAR described in Algorithm \ref{alg:lar_efron} when applied to a response mean $\bmu$ in place of the response vector $\yvec$.  For simplicity, we will refer to  the output of  $\LAR(\Xmat,\bmu)$ as the \textit{population path} and to that of $\LAR(\Xmat,\yvec)$ as the \textit{sample path}.  To  develop some properties regarding the LAR population path in Section~\ref{sec:understanding},  it will be useful to first define
how one might expect  output from  $\LAR(\Xmat,\bmu)$ to  look.

\noindent
\begin{definition}
\label{def:1}  A so-called {\it prototypical population path}   involves a sequence of columns of $\Xmat$, say $j_1,\dots,j_m$, that  enter the active set
of $\LAR(\Xmat,\bmu)$ in a one-by-one consecutive fashion on  steps $1,\dots,m$, respectively, for some $m \leq p$, with  corresponding step correlations such that $C_1>\ldots>C_m> C_{m+1}=0$ holds in Algorithm~\ref{alg:lar}.\end{definition}
Definition~\ref{def:1} outlines an outcome for population path $\LAR(\Xmat,\bmu)$, where only one variable enters on each step of Algorithm~\ref{alg:lar}, and this feature
matches how the data-based LAR sample path typically proceeds in Algorithm~\ref{alg:lar_efron} (i.e., with probability 1 for a continuous response).  
However,  while the LAR sample path $\LAR(\Xmat,\yvec)$ standardly ends only after including all $p$ variables, this aspect may {\it not} hold in the LAR population path, as reflected in that
   Definition~\ref{def:1} allows    $\LAR(\Xmat,\bmu)$ to terminate at $m$ steps which may involve fewer variables than the number $p$ of columns in $\Xmat$.
When $m<p$ holds in Definition~\ref{def:1}, it follows necessarily that $C_{m+1}=0$ at step $(m+1)$ and  this feature  implies that only $m$ (and not all $p$)
variables are needed to explain  the response mean $\bmu$; see also Lemma~\ref{lem:muk} to follow.

Additionally, it is important to note that the population level LAR algorithm in Algorithm~\ref{alg:lar}, like its data counterpart in Algorithm~\ref{alg:lar_efron}, allows for the entrance of multiple regressor variables on a single step. Although one would ideally like to generalize 
Definition~\ref{def:1} to permit a population path with such possible ``tied'' variable entries before the algorithm terminates at a step $m \leq p$, this situation presents a difficult starting point for study. While it may be reasonable to assume that such ties (or near ties) occur at the population level, the original LAR data algorithm 
can produce problematic and unclear output in the presence of such population ties, suggesting that a modification to the data-level algorithm may be necessary in this setting.  We give further discussion of this in Section~\ref{sec:midpathties}. Arguably, in the original proposal of \cite{efron2004least},  an intent exists for the steps 
of the LAR data algorithm to have meaning, which motivates our disallowing tied entrances in the population path per Definition~\ref{def:1}. 

%   \textcolor{blue}{
%   Because the statement of the LAR population Algorithm~\ref{alg:lar}, like its data counterpart in  Algorithm~\ref{alg:lar_efron},   allows for the entrance of multiple regressor variables on a single step,   one would ideally like to generalize 
% Definition~\ref{def:1} to permit a population path with such possible ``tied" variable entries before the algorithm terminates at a step $m \leq p$.  This situation, though, 
% presents a difficult starting point for study.
% While such ties (or near so) at the population level may appear reasonable, the original LAR data algorithm 
% can produce problematic and unclear output  in the presence of such population ties, so as to require modification.  We explain this aspect further in Section~\ref{sec:midpathties}. Arguably, in the original proposal of \cite{efron2004least},  an intent exists for the steps 
% of the LAR data algorithm to have meaning, which motivates our Definition~\ref{def:1}  of no variables ties in the population path.    } 

\section{Understanding least angle regression}
\label{sec:understanding}

This section presents some new findings in order to provide deeper insights about how LAR operates
 as well as to set the stage for  inference with LAR, which will be the focus of Sections~\ref{sec:Tk} and \ref{sec:practicalinference}.  We present results here for a prototypical population path $\LAR(\Xmat,\bmu)$  (i.e.,~Definition~\ref{def:1}),
though analogous findings may be expressed for the sample path $\LAR(\Xmat,\yvec)$ (cf.~Remark~\ref{rem:1}).
For context in framing the new results to follow, Proposition~\ref{prop:LARknown} first reviews some known properties of LAR  as established by \cite{efron2004least} (though the latter technically
considers sample paths, not population paths).   In the following, let $\Pmat_k$ denote the orthogonal projection onto the column space of $\Xmat_{k}$, where the latter denotes the matrix of (signed) regressor variables in the active set at step $k$ of $\LAR(\Xmat,\bmu)$.

\begin{prop}
\label{prop:LARknown}
 For each step $k$  in Algorithm~\ref{alg:lar}, it holds  that
 (i)    $j \in \calA_{k}$ implies $j \in \calA_{k+1}$;
 (ii)  $\Xmat_k^T(\bmu - \bmu_{k-1})  = C_k \onevec$, (iii) $C_{k+1}  = C_k - \gamma_k A_k$; and (iv)
\begin{equation}
\label{eqn:Pkmu}
\Pmat_k\bmu = \bmu_{k-1} + \frac{C_k}{A_k}\avec_k.
\end{equation}
 \end{prop}
 Proposition~\ref{prop:LARknown}(i) says that a column of $\Xmat$ in the active set at step $k$ will stay in   active sets for remaining LAR steps,
  which comports  with the structure of population paths in Definition~\ref{def:1}.    Proposition~\ref{prop:LARknown}(ii) re-states that all signed columns of $\Xmat$  in the $k$th active set (given by $\Xmat_k$) must have  a common positive correlation with the $k-1$ step residual $\bmu-\bmu_{k-1}$; the signs of active variables will also not change over steps (cf.~Lemma~\ref{lem:aAupdate}).  Proposition~\ref{prop:LARknown}(iii)  provides a relation  between the absolute correlations $C_{k}$ and $C_{k+1}$ across sequential LAR steps.     As discussed in \cite{efron2004least}, Proposition~\ref{prop:LARknown}(iv)  says
   that $\bmu_{k-1}$ would be updated to the projection of $\bmu$ onto the column space of   $\Xmat_k$
if employing a weight $\gamma_k = C_k / A_k$    in Step 6 of  Algorithm~\ref{alg:lar}, though typically $\gamma_k< C_k/A_k$ will hold instead. Moreover, replacing $\bmu_{k-1}$ in $\eqref{eqn:Pkmu}$ with the sum of the updates $\gamma_j\avec_j$ over steps $j=1,\dots,k-1$, where $\gamma_j = (C_j - C_{j+1})/A_j$ by Proposition \ref{prop:LARknown}(iii), we may rewrite $\eqref{eqn:Pkmu}$ as the sum
\begin{equation}
\label{eqn:Pkdecomp}
\Pmat_k \bmu = \sum_{j=1}^k C_j \Big(\frac{1}{A_j}\avec_j - \frac{1}{A_{j-1}}\avec_{j-1}\Big),
\end{equation}
after defining $A_0 \equiv \infty$ and $\avec_0 = \zerovec$.

 Note that Proposition~\ref{prop:LARknown}   does not provide explicit forms  for  important quantities like population correlations $C_{k}$ or   step approximations $\bmu_{k}$. To establish these, we first state a new result in Lemma~\ref{lem:aAupdate} concerning the computation of  equi-angular vectors and angles in $\LAR(\Xmat,\bmu)$.
For clarity, in a prototypical population path, we note that a single  column  $j_k$ of $\Xmat$ (denoted as $\xvec_{j_k}$, say) enters the active set
 on step $k$ with some sign, say $s_k \equiv  \mathrm{sign}\big(  \xvec_{j_k}^T (\bmu - \bmu_{k-1})   \big)$.  Now, for each step $k=1,\ldots,m$, define $\Imat\equiv \Imat_n$ as an $n\times n$ identity matrix
 and
 let
   \begin{equation}
   \label{eqn:in}
   \evec_k \equiv  (\Imat  - \Pmat_{k-1})\xvec_{j_k}
   \end{equation}
denote the residual/innovation of the new variable $\xvec_{j_k}$ at step $k$ after removing the linear effect $\Pmat_{k-1}\xvec_{j_k}$  of   variables active at the previous step $k-1$,
where $\Pmat_{k-1}$   is
the projection matrix for the column space of $\Xmat_{k-1}$;  above, we set $\Pmat_{0}$ to be a zero matrix and note that innovations $\evec_k$ are orthogonal ($\evec_k^T \evec_j=0$ for $k\neq j$).
We then have the following result.

% spanned by $\xvec_{j_1},\ldots, \xvec_{j_{k-1}}$ (or, equivalently, the columns of $\Xmat_{k-1}$ given by $[s_1 \xvec_{j_1},\ldots, s_{k-1}\xvec_{j_{k-1}} ]$)  for $k=1,\ldots,m$; above we set $\Pmat_{0}$ to be a zero matrix.

\begin{lemma}
\label{lem:aAupdate}   
At each step $k=1,\ldots,m$ in a prototypical population path $\LAR(\Xmat,\bmu)$,
\\
(i)   the sign $s_{k}$ of the column (i.e., $j_k$) of $\Xmat$ entering on step $k$  satisfies
\[
s_{k}  = \mathrm{sign}\Big(  \xvec_{j_k}^T (\bmu - \bmu_{k^\prime})   \Big) \in \{-1,1\}   \quad \mbox{for each $k^\prime = k-1,\ldots,m-1$}
\]
while $ \xvec_{j_k}^T (\bmu - \bmu_m)=0$.  \\
(ii)  the equi-angular vector $\avec_k$ at step $k$ satisfies
\begin{equation}
\label{eqn:aupdate}
    \frac{1}{A_k}\avec_k - \frac{1}{A_{k-1}}\avec_{k-1} = u_k s_k \evec_k
\end{equation}
with an innovation vector $\evec_k$ from (\ref{eqn:in}) and scaling $u_k \equiv (\evec_k^T\evec_k)^{-1}(1 - s_k \xvec_{j_k}^T\avec_{k-1} / A_{k-1})>0$.\\
(iii)  the  angle $A_k$  at step $k$ satisfies
\begin{equation}
\label{eqn:Aupdate}
    \frac{1}{A_k^2} - \frac{1}{A_{k-1}^2} =  u_k^2 \evec_k^T\evec_k =    (1 - s_k \xvec_{j_k}^T\avec_{k-1} / A_{k-1})^2/\evec_k^T\evec_k,
\end{equation}
where we define $A_0  \equiv \infty$ and $\avec_0 \equiv \zerovec$ above.
\end{lemma}
Up until the LAR algorithm concludes, Lemma~\ref{lem:aAupdate}(i) says that signs of active variables cannot change, while
Lemma~\ref{lem:aAupdate}(iii) implies angles $A_k$ must decrease over steps; both notions appear in \cite{efron2004least}, though without proof.
Most importantly, though, Lemma~\ref{lem:aAupdate}(ii)-(iii) states exactly how the addition of a new variable $\xvec_{j_k}$ on a step $k$ impacts the computation of the step-$k$ equi-angular vector $\avec_k$ relative to the previous equi-angular vector $\avec_{k-1}$ involving variables prior to the inclusion of $\xvec_{j_k}$: essentially,
$\avec_k$  updates $\avec_{k-1}$ by a vector $\evec_k\equiv  (\Imat - \Pmat_{k-1})\xvec_{j_k} $    that specifically accounts for the linear contribution  of a newly active variable $\xvec_{j_k}$ after removing the linear effects  $(\Imat - \Pmat_{k-1})$ of earlier active variables.  This aspect is useful towards better explaining LAR because the algorithmic output  $\bmu_{k} \equiv \bmu_{k-1} +  \gamma_k \avec_k  $
on step $k$ can   be understood as incorporating new variable information only through an innovation $\evec_k$ in $\avec_k$ that is orthogonal to  all variables used in  previous equi-angular vectors $\avec_{k-1}$ and step approximations   $\bmu_{k-1}$  (i.e., $\evec_k^T(\bmu_{k} - \gamma_k\avec_k)=0$).  This feature also leads to new and meaningful closed forms for the step correlations $C_k$ given next.
\begin{lemma}
\label{lem:Ck} At each step $k=1,\ldots,m$ in a prototypical population path $\LAR(\Xmat,\bmu)$,  the step correlation   $C_k$  has form
\begin{equation}
\label{eqn:Ckangles}
    C_k = \Big(\frac{1}{A_k^2} - \frac{1}{A_{k-1}^2}\Big)^{-1}\Big(\frac{1}{A_k}\avec_k - \frac{1}{A_{k-1}}\avec_{k-1}\Big)^T\bmu = \frac{s_k \evec_{k}^T \bmu}{ (1 - s_k \xvec_{j_k}^T\avec_{k-1} / A_{k-1})},
\end{equation}
with  $  (1 - s_k \xvec_{j_k}^T\avec_{k-1} / A_{k-1})>0$, and also satisfies
\begin{equation}
\label{eqn:Pek}
  C_k \Big(\frac{1}{A_k}\avec_k - \frac{1}{A_{k-1}}\avec_{k-1}\Big) =    \Pmat_{\evec_k} \bmu,
\end{equation}
with $ \Pmat_{\evec_k} \equiv  \evec_k(\evec_k^T \evec_k)^{-1}\evec_k^T =\Pmat_{k}-\Pmat_{k-1}$ as the projection matrix defined by $\evec_k$ in (\ref{eqn:in}).
\end{lemma}
Lemma~\ref{lem:Ck} shows that  the correlation $C_k$  at step $k$ involves a   dot product  between the response mean $\bmu$  and the linear innovation   $\evec_k\equiv  (\Imat - \Pmat_{k-1})\xvec_{j_k} $ with certain scaling.  This scaling is such that, when connected to the step-$k$ update
of the equi-angular vector $\avec_k$ based on
$ (\avec_k/A_k -  \avec_{k-1}/A_{k-1}) \propto \evec_k$ from Lemma~\ref{lem:aAupdate}, the combination $C_k  (\avec_k/A_k -  \avec_{k-1}/A_{k-1})$  is exactly
an orthogonal contribution $\Pmat_{\evec_k} \bmu = (\Pmat_{k}-\Pmat_{k-1})\bmu$,   representing  how much     a newly active variable $\xvec_{j_k}$ on step $k$
can additionally explain of the response mean $\bmu$ after removing the linear effects of all previous variables in the active set. Putting  Lemmas \ref{lem:aAupdate} and \ref{lem:Ck} together, the LAR population path is now seen to build step approximations purely through such correlations $C_k$ and  linear innovations.
 Further, we can  view  the LAR population algorithm as attempting to 
summarize the  response mean $\bmu$ through a type of Gram-Schmidt orthogonalization of columns in $\Xmat$, as stated in the following result.
\begin{lemma} 
\label{lem:muk}
At each step $k=1,\ldots,m$ in a prototypical population path $\LAR(\Xmat,\bmu)$,  the step approximation $\bmu_k \equiv \sum_{j=1}^k \gamma_j \avec_j$ of the response mean $\bmu$ satisfies
\begin{equation}
\label{eqn:sumproj}
\bmu_k   =
\sum_{j=1}^k (C_j-C_{k+1})  \Big(\frac{1}{A_j}\avec_j - \frac{1}{A_{j-1}}\avec_{j-1}\Big) =     
\sum_{j=1}^k \Pmat_{\evec_j}\bmu
-  \frac{C_{k+1} }{A_{k}}\avec_{k},
\end{equation}
where $\sum_{j=1}^k \Pmat_{\evec_j} \bmu =\Pmat_k \bmu $ is the orthogonal projection of $\bmu$ onto the column space of $\Xmat_k$ and where
$\avec_{k}/A_k=\sum_{j=1}^k s_j u_j \evec_j$ may also be re-expressed through innovations $\evec_j$ using (\ref{eqn:aupdate}).
Further, $C_{m+1}=0$ and
$\bmu_m$ is the orthogonal projection of $\bmu$ onto the column space of $\Xmat$.

% \bl{ Further,   if $m<p$ in $\LAR(\Xmat,\bmu)$ then $\Pmat_\Xmat\bmu = \bmu_m$ holds as well.} \rd{[But we could have $\bmu$ which is not in the column space of $\Xmat$.  That is, what if $\bmu = \bmu_m + \vvec$, where $\Xmat^T\vvec = 0$.  Then $\LAR(\Xmat,\bmu)$ would finish after $m$ steps, but the resulting predictor would not be equal to $\bmu$. So I think we need to remove the last statement---since do not want to assume (at this point) that $\bmu$ lies in the column space of $\Xmat$]}

\end{lemma}
In other words, Lemma~\ref{lem:muk} shows that the LAR  population path $\LAR(\Xmat,\bmu)$
 tries to explain $\bmu$ through a sequence of linearly orthogonal contributions $\Pmat_{\evec_1}\bmu,\Pmat_{\evec_2}\bmu,\ldots,\Pmat_{\evec_m}\bmu$
defined by a Gram-Schmidt orthogonalization $\evec_k \equiv (\Imat - \Pmat_{k-1})\xvec_{j_k} $, $k=1,\ldots,m$, of the variable sequence $\xvec_{j_1},\ldots, \xvec_{j_m}$
given in the   population path $\LAR(\Xmat,\bmu)$.   
Because correlation parameters $C_k \downarrow$ decrease as $k\uparrow$, the step approximations $\bmu_k$ more closely resemble  linear projections $\Pmat_k \bmu$ as steps $k$ increase.
If the LAR population path $\LAR(\Xmat,\bmu)$ stops after $m$ steps, the resulting step approximation $\bmu_m = \Pmat_m \bmu$ is  simply the linear projection of $\bmu$
onto the space spanned by the $m$ active variables $\xvec_{j_1},\ldots, \xvec_{j_m}$; when $m$ is less than the number $p$ of columns in   $\Xmat$, then a stoppage of  LAR   entails that 
only the $m$ active variables are needed to 
write the projection of the response mean $\bmu$ onto all columns of $\Xmat$.
Lemma~\ref{lem:muk}
is also a completely different result from $\bmu_k  +  (C_{k+1}/A_{k+1})\avec_{k+1} =\Pmat_{k+1} \bmu$
given in Proposition~\ref{prop:LARknown}(iv), where the latter explains a step approximation $\bmu_k$ through future variables using $\avec_{k+1},A_{k+1},\Pmat_{k+1}$
rather than through current variables using $\avec_{k},A_{k},\Pmat_{k}(= \sum_{j=1}^k \Pmat_{\evec_j})$ as  in the former result.

Lastly in this section,   Lemma~\ref{lem:whichtoenter}  plays a further role in  advancing the  interpretation of step correlations $C_k$.  Below, $\avec_{0}$ and $\Pmat_{0}$ denote a vector and matrix of zeros, with $A_0\equiv \infty$.

\begin{lemma}
\label{lem:whichtoenter} At each step $k=1,\ldots,m$ in a prototypical population path $\LAR(\Xmat,\bmu)$, \\
(i) for any  column   $j\notin \calA_k$ of $\Xmat$, it holds  in Algorithm~\ref{alg:lar} that
\begin{equation}
\label{eqn:projkj}
(\cvec_k)_j  - \frac{C_k}{A_k}(\wvec_k)_j = \xvec_j^T(\Imat - \Pmat_k)\bmu
\end{equation}
and that  $ A_k > (\wvec_k)_j r_{k,j} \equiv \xvec_j^T\avec_{k}r_{k,j}$  with $r_{k,j} =\mathrm{sign}(\xvec_j^T(\Imat - \Pmat_k)\bmu)$ as in (\ref{eqn:gammak}).\\
(ii)  in (6.)~of Algorithm~\ref{alg:lar}, a column  $j_k$   enters the active set on step $k$  if $j_k$ maximizes
\begin{equation}
\label{eqn:Ckplus1j}
C_{k,j} \equiv   \frac{|\xvec_j^T(\Imat - \Pmat_{k-1})\bmu |}{1 - r_{k-1,j} \xvec_j^T\avec_{k-1}/A_{k-1}}  \geq 0
\end{equation}
over $j \in \{1,\ldots,p\}\setminus \calA_{k-1}$, noting that $r_{k-1,j}= \sign ( \xvec_j^T(\Imat - \Pmat_{k-1})\bmu  )$ holds  as well as $1 > r_{k-1,j} \xvec_j^T\avec_{k-1}/A_{k-1}$. \\
(iii)
the correlation $C_k$ at step $k$  further satisfies
\[C_k =C_{k,j_k} = \max_{j \not\in \calA_{k-1}}C_{k,j} >0,\]
while the sign $s_k$ satisfies $s_k = r_{k-1,j_k}$.
\end{lemma}

Lemma~\ref{lem:whichtoenter} offers a new a  perspective about how the {\it ordering} of variables/correlations is determined  across  LAR steps, which can be explained as follows.  Lemmas~\ref{lem:whichtoenter}(i)-(ii) are clarifying relations about the part (\ref{eqn:gammak}) of the LAR population Algorithm~\ref{alg:lar}, which
determines how one update step ends and the next begins. Essentially, how the algorithm ends at a step $k-1$ (i.e., based on a current active set $\calA_{k-1}$, projection matrix $\Pmat_{k-1}$, and equi-angular vector $\avec_{k-1}$ and angle $A_{k-1}$) determines which variable $j_k$ enters the active set $\calA_k$ on step $k$. In particular,
 Lemma~\ref{lem:whichtoenter}(ii) re-writes the Algorithm~\ref{alg:lar}   criterion (\ref{eqn:gammak}) to show that a non-active variable $j \in \{1,\ldots,p\}\setminus \calA_{k-1}$ at step $k-1$ competing to enter the active set $\calA_k$ on step $k$ must maximize  the nonnegative criterion  $C_{k,j}$ in \eqref{eqn:Ckplus1j}, or equivalently, its square $C_{k,j}^2$.  Using Lemma~\ref{lem:aAupdate} with Lemma~\ref{lem:whichtoenter}(ii), we can view this square  as
\begin{equation}
\label{eqn:Cksquared}
C^2_{k,j} = \frac{\mathrm{SS}_{k,j}}{A_{j,k}^{-2} - A_{k-1}^{-2}},
\end{equation}
where $\mathrm{SS}_{k,j} \equiv   [\xvec_j^T(\Imat - \Pmat_{k-1})\bmu]^2/ \xvec_j^T (\Imat - \Pmat_{k-1}) \xvec_j$  is the sequential sum of squares (SS) for how much of the response mean $\bmu$ can be explained by a linear relationship with variable $\xvec_j$ after accounting for  those variables active at step $k-1$, and
$A_{j,k}$ denotes what  angle would arise at step $k$ if variable $\xvec_j$ were to be admitted into the active set $\mathcal{A}_k$ and used to produce the step-$k$ equi-angular vector $\avec_k$ and angle $A_k$ via the recursions in Lemma~\ref{lem:aAupdate}(ii)-(iii); refer to Remark~\ref{rem:2} for details on how computation of $A_{j,k}$  requires signing variable $\xvec_j$ with a sign $r_{k-1,j}$.
From \eqref{eqn:Cksquared}, we can see that, when choosing the next variable to enter, LAR does not simply consider the variable with the largest  SS as greedy  selection would prescribe, but rather LAR seeks a variable that maximizes  SS (i.e., the numerator of (\ref{eqn:Cksquared}))  relative to how an equi-angle would update (i.e., the denominator of (\ref{eqn:Cksquared})).
Namely, variable $\xvec_j$ is a strong candidate for entry on step $k$ if the correlation between $\bmu$ and $(\Imat - \Pmat_{k-1})\xvec_j$ is large and if the entrance of $\xvec_j$ with sign $r_{k-1,j}$ would produce a small change in the angle $A_{k-1}$ to a next angle $A_k$.

\begin{remark}
\label{rem:2}
  Lemma~\ref{lem:whichtoenter} and (\ref{eqn:Cksquared}) entail that a variable $\xvec_j$ must be signed $r_{k-1,j}$
such that $1 > r_{k-1,j} \xvec_j^T\avec_{k-1}/A_{k-1}$ holds in (\ref{eqn:Ckplus1j}).  This ensures that if a signed variable $\xvec_j$ is included in the active set $\calA_{k}$ at step $k$ to update the equi-angular vector $\avec_{k}$ and angle $A_k$, then the vector $\avec_{k}/A_k$ at step $k$
  will   differ  from the previous vector $\avec_{k-1}/A_{k-1}$.  That is, the LAR algorithm forces equi-angular vectors to {\it change} across steps,
   as also entailed by Lemma~\ref{lem:aAupdate}  and  the correlation expressions of Lemma~\ref{lem:Ck} (i.e., where the entering variable $j_k$  on step $k$ satisfies  $1 > s_k \xvec_{j_k}^T \avec_{k-1}/A_{k-1}$ in those results and in Lemma~\ref{lem:whichtoenter}(iii)).  To appreciate how LAR can change   equi-angular vectors/angles across steps, consider a case where some variable  index $j\notin \mathcal{A}_{k-1}$  is not active at step $k-1$
   (so that $|(\cvec_{k-1})_j| < C_{k-1}$)
  but is already equiangular to $\avec_{k-1}$ at step $k-1$ (so that $(\wvec_{k-1})_j \equiv \xvec_{j}^T \avec_{k-1} = A_{k-1}$);  this implies that $r_{k-1,j} = \mathrm{sign}((\cvec_{k-1})_j - C_{k-1}(\wvec_{k-1})_j /A_{k-1})=-1$  in (\ref{eqn:gammak}) of  Algorithm~\ref{alg:lar} or in Lemma~\ref{lem:whichtoenter},
  so that the LAR algorithm would flip the sign of   $\xvec_j$ before   considering this variable for possible entry at  step $k$ in order to   force a  subsequent  equi-angular vector $\avec_k$ and angle $A_k$ to change.
\end{remark} 
\begin{remark}
 \label{rem:3}
 \cite{efron2004least} developed a notion of degrees of freedom (df), or model complexity, for LAR sample paths based on a covariance study, suggesting df $k$ in LAR sample paths should equate to the number $k$ of steps.    Lemma~\ref{lem:muk}  suggests a related possible notion of df for the LAR population path, when df is standardly measured by the size/trace of a linear combination. That is, the step-$k$ approximation $\bmu_k$ emerges as a linear combination  $\bmu_k = \Mmat_k \bmu $ of the response mean $\bmu$, where the matrix $\Mmat_k \equiv \Pmat_k  +    \avec_{k}\evec_{k+1}^T c$  has a form by  
(\ref{eqn:Ckangles})-(\ref{eqn:sumproj}) (for some constant $c$) such that
\[
\mathrm{trace}(\Mmat_k) = \mathrm{trace}(\Pmat_k) +    \evec_{k+1}^T\avec_{k}c =k
\]
holds, using that $\Pmat_k$ is a projection matrix based on the first $k$ variables in the LAR population path  so that $\mathrm{trace}(\Pmat_k)=k$, and that the vector $\avec_k$ is a linear combination of these variables whereby the $k+1$ step innovation vector $\evec_{k+1} \equiv  (\Imat - \Pmat_{k})\xvec_{j_{k+1}} $ from (\ref{eqn:in}) is orthogonal  to   $ \avec_k=\Pmat_k\avec_k$, i.e., so $\mathrm{trace}( \avec_{k}\evec_{k+1}^T )=\evec_{k+1}^T\avec_{k}=0$ follows in $\mathrm{trace}(\Mmat_k)$.  
\end{remark}
\begin{remark}
 \label{rem:1}
  While   results in this section were prescribed for the LAR population path   $\LAR(\Xmat,\bmu)$ as the output of Algorithm~\ref{alg:lar}, the same results also hold for the LAR sample path
 $\LAR(\Xmat,\yvec)$ as the output of Algorithm~\ref{alg:lar_efron} upon a notational change of adding ``hats'' ($~\hat{\cdot}\;$) to quantities and replacing ``$\bmu$" with ``$\yvec$."
 As explained at the end of Section~\ref{sec:lar}, however, the LAR sample path will typically terminate only after $p$ steps with $p<n$ columns of $\Xmat$.
 \end{remark}

We end this section with a description of LAR on a small illustrative data set.

\begin{example}
Here we give the details at each step of $\LAR(\Xmat,\bmu)$ (Algorithm \ref{alg:lar}) for  
$$
\Xmat = \frac{1}{\sqrt{2}}\left[\begin{array}{ccc}
     1&1&1  \\
     1&0&0 \\ 
     0&1&0 \\
     0&0&1
\end{array}\right] \quad \text{ and } \quad \bmu = \left[\begin{array}{r}
     1\\
     2\\
     -1\\
     0
\end{array}\right]
$$
and illustrate how LAR decomposes the mean by orthogonalizing the columns of $\Xmat$. We moreover illustrate the role of the entrance criteria defined in \eqref{eqn:Ckplus1j} in Lemma \ref{lem:whichtoenter}.

In the first step of Algorithm \ref{alg:lar} we obtain $\cvec_1 = (1/\sqrt{2})(3,0,1)^T$, giving $C_1 = 3/\sqrt{2}$, so $\xvec_1$ enters the active set, whereby $\calA_1 = \{1\}$, $\avec_1 = \xvec_1$, and $A_1 = 1$. We next compute $\wvec_1 = \Xmat^T\avec_1 = (1,1/2,1/2)^T$ and obtain $\gamma_{1,2} = \sqrt{2}$ and $\gamma_{1,3} = \sqrt{2}(4/3)$, leading to $\gamma_1 = \sqrt{2}$ with the update as $\bmu_1 = \gamma_1 \avec_1 = (1,1,0,0)^T$. On the second step, we obtain $\cvec_2 = \Xmat^T(\bmu - \bmu_1) = (1/\sqrt{2})(1,-1,0)$, implying that $C_2 = 1/\sqrt{2}$  and that $\xvec_2$ enters the active set (note that $\xvec_1$ and $\xvec_2$ have equal absolute correlation with the residual $\bmu - \bmu_1$) so that $\calA_2 = \{1,2\}$, $\avec_2 = (1/\sqrt{2})(0,1,-1,0)^T$, and $A_2 = 1/2$.  From here we compute $\wvec_2 = \Xmat^T \avec_2= (1/2,-1/2,0)^T$ and obtain $\gamma_{2,3} = \sqrt{2}$, noting that  $\xvec_3$ is the only column of $\Xmat$ outside the active set, so that $\gamma_2 = \gamma_{2,3} = \sqrt{2}$. Now, because $C_2/A_2 = \sqrt{2}$ here, the update $\bmu_2 = \bmu_1 + \gamma_2\avec_2$ must equal the projection of $\bmu$ onto the span of $\{\xvec_1,\xvec_2\}$ (see Proposition \ref{prop:LARknown}(iv)).  Since $\bmu$ lies in the span of $\{\xvec_1,\xvec_2\}$, we have $\bmu_2 = \bmu$. In consequence $\cvec_3 = \Xmat^T(\bmu - \bmu_2)=\zerovec$ holds and $\LAR(\Xmat,\bmu)$ terminates at step $2$ with $C_3 = 0$ ($\xvec_3$ never enters).

Next consider the relation of the path quantities $\bmu_1$ and $\bmu_2$ to the innovation vectors $\evec_1$ and $\evec_2$, as given in Lemma $\ref{lem:muk}$. In this example, LAR constructs innovation vectors $\evec_1 = \xvec_1 = (1/\sqrt{2})(1,1,0,0)^T$ and $\evec_2 = \xvec_2 - \xvec_1 \xvec_1^T\xvec_2 = (1/\sqrt{2})(1/2,-1/2,1,0)^T$ with projections $\Pmat_{\evec_1}\bmu = (3/2,3/2,0,0)^T$ and $\Pmat_{\evec_2}\bmu = (-1/2,1/2,-1,0)^T$ and then constructs $\bmu_1$ and $\bmu_2$ as
\begin{align*}
\bmu_1 & = \Pmat_{\evec_1}\bmu - \frac{C_2}{A_1}\avec_1 = \left[\begin{array}{r}
       3/2\\
       3/2\\
       0\\
       0
\end{array}\right]  - \left[\begin{array}{r}
       1/2\\
       1/2\\
       0\\
       0
\end{array}\right] \\
\bmu_2 &= C_1 \frac{1}{A_1} \avec_1 + C_2\Big(\frac{1}{A_2}\avec_2 - \frac{1}{A_1}\avec_1\Big) = \Pmat_{\evec_1}\bmu + \Pmat_{\evec_2}\bmu = \left[\begin{array}{r}
       3/2\\
       3/2\\
       0\\
       0
\end{array}\right] + \left[\begin{array}{r}
       -1/2\\
       1/2\\
       -1\\
       0
\end{array}\right].
\end{align*}
Note that $\bmu_2$ is decomposed as the sum of two vectors which are orthogonal to each other.

% The expression for $\bmu_1$ illustrates the identity in \eqref{eqn:sumproj}, and the expression for $\bmu_2$ illustrates the decomposition in \eqref{eqn:Pkdecomp} as well as the identity in \eqref{eqn:Pek} in Lemma \ref{lem:Ck}.

To see how the criterion defined in \eqref{eqn:Ckplus1j} in Lemma \ref{lem:whichtoenter} determines variable entrances, note that the magnitudes of the entries of $\cvec_1$ give the criterion values $C_{1,1} = 3/\sqrt{2}$, $C_{1,2}=0$, and $C_{1,3}=1/\sqrt{2}$. Since $C_{1,1} > \max\{C_{1,2},C_{1,3}\}$, column $\xvec_1$ is admitted into the active set on step 1 and we have $C_1 = C_{1,1} = 3/\sqrt{2}$. Next, as soon as $\xvec_1$ has entered the active set, one may compute the criteria $C_{2,2}$ and $C_{2,3}$ to determine which among the remaining variables ($\xvec_2$ and $\xvec_3$) will enter on step 2. One obtains $C_{2,2} = 1/\sqrt{2} > C_{2,3} = 1/(3\sqrt{2})$, indicating the entrance of $\xvec_2$ on step~2 with $C_2 = C_{2,2} = 1/\sqrt{2}$. Lastly, we have $C_3 = C_{3,3} = 0$.
\end{example}

\section{Estimation properties of the LAR sample path}
\label{sec:Tk}

The previous section showed that the LAR population path $\LAR(\Xmat,\bmu)$ produces sequences of variable indices $j_1,\ldots,j_m$, variable signs $s_1,\ldots,s_k$, and positive step correlations $C_1>\cdots>C_m$ for explaining a response mean $\bmu$
(cf.~Lemmas~\ref{lem:aAupdate}-\ref{lem:whichtoenter}).   We next consider inference about these  LAR population quantities based on the LAR sample version $\LAR(\Xmat,\yvec)$.  In particular, several large-sample properties for LAR estimation  can be established under some typical linear regression assumptions:

\begin{assum}
\label{assum:regassumptions}
Suppose:
\begin{enumerate}[label=(R\arabic*)]
    \item For each $n \geq 1$, $\yvec_n = \Xmat_n \bbeta_n + \berror_n$, where $\Xmat_n = [\xvec_{n1}~\dots~\xvec_{np}]$ is an $n \times p$ design matrix with rank $p$, $\bbeta_n$ is a $p \times 1$ vector with norm $\|\bbeta_n\|\leq C$ for some $C < \infty$, and $\berror_n$ is a vector of error terms such that $\E (\berror_n) = \zerovec$ and $\Var (\berror_n) = \sigma^2\Imat_n$ for $\sigma^2>0$.
    \item  as $n\to\infty$, $n^{-1}\Xmat_n^T\Xmat_n \to \bSigma$ as $n \to \infty$, where $\bSigma$ is positive definite.
    \item $n^{-1/2}\Xmat_n^T\berror_n \indist \calN(\zerovec,\bSigma \sigma^2)$   (i.e., scaled error vectors have a normal limit) and $ n^{-1}\berror_n^T\berror_n \inprob \sigma^2$ (i.e., the average of squared error terms converges in probability to $\sigma^2$) as $n \to \infty$.
    \item The normalized design matrix $\Xmat$, the noisy response $\yvec$, and the noiseless response $\bmu$ are obtained as $\Xmat =  \Xmat_n \diag( \|\xvec_{n1}\|^{-1} ,\dots, \|\xvec_{np}\|^{-1})$, $\yvec = n^{-1/2}\yvec_n$, and $\bmu = n^{-1/2}\Xmat_n \bbeta_n$.
    %     \begin{align*}
    % \Xmat &=  \Xmat_n \diag( \|\xvec_{n1}\|^{-1} ,\dots, \|\xvec_{np}\|^{-1}) \\
    % \yvec &= n^{-1/2}\yvec_n\\
    % \bmu &= n^{-1/2}\Xmat_n \bbeta
    % \end{align*}
\end{enumerate}
% In this setup define a normalized version of the design matrix as
    % as well as rescaled noisy and noiseless responses $\yvec = n^{-1/2}\yvec_n$ and $\bmu = n^{-1/2}\Xmat_n \bbeta$, respectively.
\end{assum}

The conditions in  Assumption~\ref{assum:regassumptions} are common to 
asymptotic developments for  standard linear regression (cf.~\cite{freedman1981bootstrapping,lai1979strong}) and do not technically require  model errors  to be iid (independent and identically distributed); in (R1) the errors are uncorrelated with mean zero and a common variance (i.e., white noise). In the above, $\Xmat_n$ and $\yvec_n$ represent the observed or ``raw'' design matrix and response vector, whereas $\Xmat$ and $\yvec$ represent rescaled versions of these to be used as inputs in the LAR algorithm. In (R1), we allow $\bbeta_n$ to depend on $n$, so that the coefficients giving the contributions of the columns of $\Xmat_n$ to the mean of $\yvec_n$ need not converge as $n \to \infty$. The rescalings in (R4) ensure that  the population step correlations $C_k$  from $\LAR(\Xmat,\bmu)$ (i.e., Algorithm~\ref{alg:lar})  remain finite  as $n \to \infty$ and that their sample counterparts $\hat C_k$ from $\LAR(\Xmat,\yvec)$ (i.e., Algorithm~\ref{alg:lar_efron}) have standard  $1/\sqrt{n}$ estimation rates; in other words, under (R4), we may write $\yvec = \bmu + \berror$, where $\berror \equiv n^{-1/2}\berror_n$ and $\bmu$ have bounded norms as $n \to \infty$.  
In the following, we  assume that population correlations $C_k$ admit  an eventual ordering as $n\to \infty$, which is further discussed at the end of Section~4.1.

\subsection{Consistency of the LAR sample path}
Our first main result gives conditions under which the LAR sample path   is consistent  
for the LAR population path.  Recall that, in a prototypical population path (Definition~\ref{def:1}),  
 a sequence of columns $j_1,\dots,j_m$ of $\Xmat$ enter the active set  one-by-one on steps $1,\dots,m$ (i.e., $\calA_k\setminus \calA_{k-1} =\{j_k\}$ over each $k=1,\ldots,m$ for $m\leq p $) with corresponding step correlations $C_{1}>\cdots >C_m> C_{m+1}=0$ and corresponding signs $s_1,\ldots,s_m \in\{\pm 1\}$;  those variables which do not enter before the population path concludes on step $m$ make no contribution toward explaining the response $\bmu$.   Note that, while  the population path may terminate  after $m$ steps for some given $m\leq p$, the sample path will standardly  complete exactly $p$ steps before terminating, yielding $p$ estimated  correlations $\hat C_1 > \dots >  \hat C_p > \hat C_{p+1} = 0$.
 Despite such differences, the LAR sample path   can capture  the population version, as seen next.
 In particular, as the sample size $n$ grows, Theorem~\ref{thm:larconsistency}  states that the LAR sample path   will produce the following: (i)    a sequence of estimated   active sets $\hat{\calA}_k$ as well as the variable signings that correctly match the population versions $\calA_k$ for $k \leq m$; (ii)   estimated step correlations $\hat C_k$  that converge in probability to their non-zero population counterparts $C_k$ for $k \leq m$, (iii) vectors of estimated correlations $\hat \cvec_k \equiv \Xmat^T (\yvec - \hat \yvec_{k-1}) $ that also converge more generally  to counterparts $\cvec_k \equiv \Xmat^T (\bmu-\bmu_{k-1})$ on each step $k\leq m$, and (iv) estimated step correlations $\hat C_k$ that  converge in probability to zero for any $m < k \leq p$, indicating that such sample correlations are correctly estimating population counterparts in such cases (i.e., $C_{k}=0$ for $k>m$).

\begin{theorem}[Consistency of $\LAR(\Xmat,\yvec)$]
\label{thm:larconsistency}
For   $n \geq n_0$ given some $n_0$, suppose the   prototypical population path $\LAR(\Xmat,\bmu)$ holds for columns $j_1,\ldots,j_m$ of $\Xmat$ with $\calA_k\setminus \calA_{k-1} = \{j_k\}$ for   $k=1,\ldots,m$ (some $m\leq p$) with   signs   $s_{1},\dots,s_{m}$
and  correlations $C_1 >\cdots>C_m > C_{m+1}=0$. Along with Assumption~\ref{assum:regassumptions}, suppose also in Algorithm~\ref{alg:lar} that the following conditions hold:

\begin{enumerate}[label=(M\arabic*)] \itemsep0cm
    \item For some $\tau$ and $\xi$ such that $0 < \xi \leq \tau < 1$, it holds that $|(\cvec_k)_j|/C_k \leq 1 - \tau$  for any $k \in \{1,\dots,m\}$ and $j \notin \calA_k$ satisfying 
    \[
    \frac{|\xvec_j^T\avec_k|}{A_k} < \frac{|(\cvec_k)_j|}{C_k} + \xi \quad \text{ and }  \quad \sign(\xvec_j^T\avec_k) = \sign((\cvec_k)_j).
    \]
    \item For some $\delta > 0$, $C_1 - |(\cvec_1)_j| \geq \delta$ holds for all $j \neq j_1$ and  $A_k(\gamma_{k,j} - \gamma_k) \geq \delta$ holds for all $j \not\in \calA_k \cup\{j_{k+1}\}$ and $k \in \{1,\dots,m-1\}$.
\end{enumerate}

Then, in the sample path $\LAR(\Xmat,\yvec)$, it holds as $n\to \infty$ that\\
(i) $P(\hat \calA_k \setminus \hat \calA_{k-1} = \{j_k\} \text{ and } \hat s_k = s_k) \to 1$ for each $k=1,\dots,m$, where $\hat \calA_0 \equiv \emptyset \equiv \calA_{0}$; \\
(ii) $\hat C_k - C_k \inprob 0$ for each $k=1,\dots,m$;\\  
(iii)     $\max_{j \not\in \calA_k} |(\hat \cvec_k)_j - (\cvec_k)_j | \inprob 0$
  for    each $k=1,\dots,m$; and\\
(iv) $\hat C_k \inprob 0$   for each $k > m$. 
\end{theorem}

To motivate conditions (M1) and (M2) of Theorem \ref{thm:larconsistency}, recall that the LAR population path decides what variable to admit at each step by maximizing the correlation criterion (\ref{eqn:Ckplus1j})  described in  Lemma~\ref{lem:whichtoenter}(ii). Conditions (M1) and (M2) are designed to be minimal conditions under which one can ensure that the sample path is able to identify at each step the correct variable to admit to the active set. In particular, the conditions ensure that if $C_{k+1,j_{k+1}} > C_{k+1,j}$ holds for all competing indices $j$ at a step $k+1$ in the population path (entailing the entrance of $\xvec_{j_{k+1}}$), then the data counterpart $\hat C_{k+1,j_{k+1}} > \hat C_{k+1,j}$ can also be shown to hold for all competing indices $j$ in the sample path.  This is ensured by i) bounding the denominator of $\eqref{eqn:Ckplus1j}$ away from zero in particular cases where the LAR population path may allow this to be small, which is the role of (M1), and ii) by requiring that $C_{k+1,j_{k+1}}$ exceed the  criterion values $C_{k+1,j}$ for all competing indices $j$ by some positive margin, which is the role of (M2). These features of a LAR population path give a LAR sample path a chance to resolve from data which variable should enter on each step of the algorithm in practice, at least as the sample size  $n$ increases.

Of the two conditions, (M1) appears more technical.  In place of (M1),  denominators in \eqref{eqn:Ckplus1j} could also be bounded away from zero in a more direct way by replacing (M1) with:
\begin{enumerate}[label=\textit{(M\arabic*$'$)}] \itemsep0cm
    \item \textit{For some $0 < \tau < 1$,  $|\xvec_j^T\avec_k/A_k| \leq 1 - \tau$ holds for all $k\in\{1,\dots,m\}$ and $j \notin \calA_k$.}
\end{enumerate}
Condition (M1$'$) closely resembles the irrepresentable condition found in results on the consistency of Lasso \citep{zhao2006model}. To see this, note that for each $k=1,\dots,m$ we may write $|\xvec_j^T \avec_k/A_k|=|\xvec_j^T \Xmat_{\calA_k}(\Xmat_{\calA_k}^T\Xmat_{\calA_k})^{-1}\svec_k|$ where $\svec_k\equiv(s_1,\dots,s_k)^T$ and $\Xmat_{\calA_k}$ denotes the matrix constructed from the columns of $\Xmat$ with indices in $\calA_k$. However, we believe (M1$'$) to be considerably stronger than necessary, due to the fact that  LAR  often bounds the denominator of $\eqref{eqn:Ckplus1j}$ away from zero as $1 - r_{k,j}\xvec_j^T\avec_k/A_k \geq 1$ in its mechanics, except in particular cases where one  then generally has $1 - r_{k,j}\xvec_j^T\avec_k/A_k \geq 1 - |(\cvec_k)_j|/C_k$ (see Lemma \ref{lem:denombound} in the Appendix). It is only in these particular cases that a condition is needed to bound the denominator in \eqref{eqn:Ckplus1j}, and (M1) is minimal in the sense that it is designed to address only these cases.

\begin{figure}
\begin{center}
\begin{tikzpicture}[scale=0.7]
  \draw[-,dotted] (-1,0) -- (8,0); \node[below] at (8,0) {$\gamma$};
  % \node[left] at (-.1,0) {$0$};\node[below] at (0,-.1) {$0$};
  \draw[-,dotted] (0,-4) -- (0,4); \node[below] at (7,-.1) {$C_k / A_k$};
  % \draw[-,dotted] (7,-4) -- (7,4);
  \draw[-] (0,3) -- (7,0); \node[left] at (0,3) {$C_k$};
  \draw[-] (0,-3) -- (7,0);\node[left] at (0,-3) {$ - C_k$};
  \fill[gray, opacity = 0.1] (0,-3) -- (7,0) --  (0,3) --  (0,-3);

  % \fill[red, opacity = 0.2] (35/17,-3.5) -- (35/17,3.5) --  (35/17+.2*7/3,3.5) -- (35/17+.2*7/3,-3.5); \node[below] at (35/17+.1*7/3,-3.5) {$\delta A_k^{-1}$};

  \fill[red, opacity = 0.2] (0,36/17) -- (8,36/17) --  (8,36/17-.2) -- (0,36/17-.2); \node[right] at (8,36/17-.1) {$\delta$};

  \fill[red, opacity = 0.2] (35/17,0) -- (35/17,4) --  (35/17+.2*7/3,4) -- (35/17+.2*7/3,0); \node[above] at (35/17+.1*7/3,4) {$\delta A_k^{-1}$};

  % \fill[red, opacity = 0.2] (-.5,3) -- (.5,3) -- (.5,2.6) -- (-.5,2.6);
  % \fill[red, opacity = 0.2] (-.5,-3) -- (.5,-3) -- (.5,-2.6) -- (-.5,-2.6);

  \draw[-,dashed] (0,1/2) -- (7,3); \node[left] at (0,1/2) {$(\cvec_k)_j$};
  % \node[right] at (5.8,2.4) {$(\cvec_k)_{j} - \gamma (\wvec_k)_{j}$};
  \draw[-,dotted] (0,1.636364) -- (35/11,1.636364); \node[left] at (0,1.636364) {$C_{k+1,j}$};
  \draw[-,dotted] (35/11,0) -- (35/11,1.636364); \node[below] at (35/11,0) {$\gamma_{kj}$};

  \draw[-,dashed] (0,-2) -- (3,4); \node[left] at (0,-2) {$(\cvec_k)_{j_{k+1}}$};
  % \node[right] at (2.5,3) {$(\cvec_k)_{j_{k+1}} - \gamma (\wvec_k)_{j_{k+1}}$};
  \draw[-,dotted] (0,36/17) -- (35/17,36/17); \node[left] at (0,36/17) {$C_{k+1,j_{k+1}}$};
  \draw[-,dotted] (35/17,0) -- (35/17,36/17); \node[below] at (35/17,0) {$\gamma_{k,j_{k+1}}$};

\end{tikzpicture}
\end{center}
\caption{Column $\xvec_{j_{k+1}}$ enters on step $k+1$, while $\xvec_j$ does not enter. The dashed lines denote either  $ (\cvec_k)_{j_{k+1}} - \gamma (\wvec_k)_{j_{k+1}}$ or $  (\cvec_k)_j - \gamma (\wvec_k)_j$, while the solid lines denote $  C_k - \gamma A_k$ and $ -C_k + \gamma A_k$.  The horizontal and vertical bands illustrate condition (M2) of Theorem \ref{thm:larconsistency}.}
\label{fig:margincond2}
\end{figure}
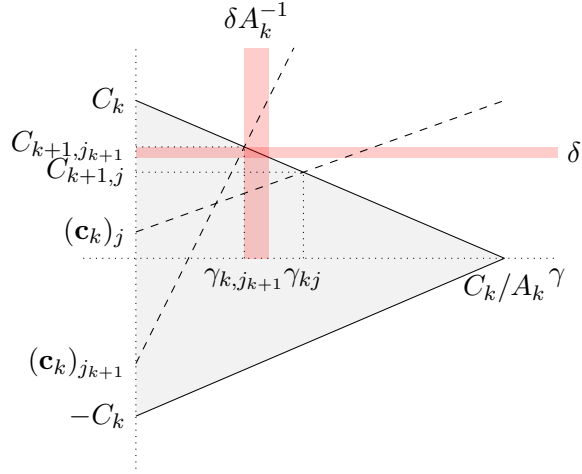

Condition~(M2) is illustrated  in Figure~\ref{fig:margincond2}, which depicts a population entrance criterion $C_{k+1,j}$ in Lemma~\ref{lem:whichtoenter}(ii)  in terms of weights $\gamma_{k,j}$  from \eqref{eqn:gammak}   as the value of $\gamma$ at which the line $ (\cvec_k)_j + \gamma (\wvec_k)_j$ intersects with either the line $C_k - \gamma A_k$ or the line $C_k + \gamma A_k$.  If an index $j_{k+1}$ should enter the population active set on step $k+1$, condition (M2) requires that $\gamma_{k,j_{k+1}}$ be smaller than $\gamma_{k,j}$ for all competing indices $j$ by the margin $\delta A_k^{-1}$, or, equivalently, that $C_{k+1,j_{k+1}}$ be greater than $C_{k+1,j}$ for all competing indices $j$ by the margin $\delta$.  Note, on step $k=1$, that the entrance criteria $C_{1,j}$, $j=1,\dots,p$, are simply the entries of $\cvec_1$ (since $\avec_0 \equiv \zerovec$ and $A_0 \equiv \infty$ and $\Pmat_0$ represents a matrix of zeros), so the first part of (M2) requires $C_{1,j_1}$ to exceed $C_{1,j}$ by at least $\delta$ for all $j \neq j_1$. Lastly, (M2) also serves to bound the smallest nonzero entrance correlation $C_m$ away from zero by the margin $\delta$. We will show in Section \ref{sec:mest} that this aspect of (M2) is important if one wishes to learn the number of steps $m$ in the population path from sample data.

In summary, conditions (M1) and (M2) are sufficient so that the winner $C_{k+1,j_{k+1}} =C_{k+1}$ at step $k+1$ of Lemma~\ref{lem:whichtoenter}(ii) can be identified from sample data.
 
\begin{remark} One interprets conditions (M1) and (M2) to hold whenever vacuous.  For example, if $p=1$, then (M1)-(M2) hold vacuously as there are no variables outside of $\mathcal{A}_1$.        
\end{remark}

\begin{remark} 
As an example, consider a design matrix $\Xmat$ with orthogonal columns such that  $\Xmat^T\Xmat = \Imat$. In this case the LAR population path is completely determined by the vector of marginal ``correlations'' $\cvec_1 = \Xmat^T\bmu$ computed in the first step. This owes to the fact with orthogonal design columns one has $\xvec_j^T\avec_k = 0$ for all $j \notin \calA_k$ at each step $k=1,\dots,m$. In this case, condition (M1) of Theorem \ref{thm:larconsistency} holds automatically, while (M2) requires that the magnitudes of any two non-zero entries of $\cvec_1$ differ by at least $\delta$ and that the smallest non-zero entry of $\cvec_1$ has magnitude at least $\delta$.
\end{remark}

To further clarify our consistency result in Theorem \ref{thm:larconsistency}, we point out that the population-level step correlations $C_k$, the correlation vectors $\cvec_k$, and  the remaining quantities computed in $\LAR(\Xmat,\bmu)$ in Algorithm \ref{alg:lar} in fact depend on the sample size $n$, which we have suppressed in our notation.  That is, we have used the notation $C_k$, $\cvec_k$, etc., rather than $C_{k}(n)$, $\cvec_{k}(n)$, etc.,
where population-level quantities  need not be constant with $n$.   While (R1)--(R4) of Assumption~\ref{assum:regassumptions} prescribe certain limiting characteristics for the data generating process, it is important to observe that these conditions are not sufficient to determine the limiting behavior of all quantities computed in the population path $\LAR(\Xmat,\bmu)$. In particular, (R1)--(R4) do not force the population-level step correlations $C_k \equiv C_{k}(n)$ or the vectors $\cvec_k \equiv \cvec_{k}(n)$ to possess limits. However, we find that it is not necessary to require the existence of limits for the $C_{k}(n)$ or $\cvec_{k}(n)$ in order to ensure that these can be accurately estimated by their sample counterparts in the sense of Theorem \ref{thm:larconsistency}(ii)-(iii). Rather, in Theorem \ref{thm:larconsistency}, in addition to (R1)--(R4), we assume only that variables, at large enough sample sizes, enter the active set one-by-one according to a stable ordering of steps/correlations  in the population path $\LAR(\Xmat,\bmu)$ and, moreover, that this entrance ordering is recoverable by the sample path $\LAR(\Xmat,\yvec)$, where the latter is ensured by (M1) and (M2).

% the path $\LAR(\Xmat,\bmu)$ is determined entirely by the vector $\cvec_1 = \Xmat^T\bmu$. In this setting Theorem \ref{thm:larconsistency} conditions (M1) and (M2) are equivalent to requiring $|(\cvec_1)_{j_k}| - |(\cvec_1)_{j_{k+1}}| > \delta$ for $k=1,\dots,m-1$ and $|(\cvec_1)_{j_m}|  - |(\cvec_1)_j| > \delta$ for all $j \notin\{j_1,\dots,j_m\}$.

% where a prototypical LAR path in which the indices $j_1,\dots,j_m$ enter one by one on steps $1,\dots,m$ before the algorithm terminates with $C_{m+1} = 0$ is ensured by $|(\cvec_1)_{j_1}| > \dots >|(\cvec_1)_{j_m}| > 0$ and $(\cvec_1)_j = 0$ for $j \notin\{j_1,\dots,j_m\}$.  Moreover, the conditions (M1) and (M2) of Theorem \ref{thm:larconsistency} are in this setting equivalent to $|(\cvec_1)_{j_k}| - |(\cvec_1)_{j_{k+1}}| > \delta$ for $k=1,\dots,m-1$ and $|(\cvec_1)_{j_m}|  - |(\cvec_1)_j| > \delta$ for all $j \notin\{j_1,\dots,j_m\}$.

\subsection{Sampling distributions from the LAR sample path}
We next establish sampling  distributions of  LAR sample correlations $\hat{C}_k$ estimating the population step correlations $C_k$.  For  each sample step $k=1,\dots,p$, define the quantities
\begin{equation}
\label{eqn:T}
T_{nk} \equiv \hat s_k\bigg(\frac{1}{\hat A_{k}^{2}} - \frac{1}{\hat A_{k - 1}^{2}}\bigg)^{1/2}\sqrt{n}\big( \hat C_k - C_k \big),
\end{equation}
where the scaling above involves estimated angles $\hat{A}_k$  (with $\hat{A}_0\equiv 0$) along with an estimated sign $\hat{s}_k$ and where we set population counterparts $C_k \equiv 0$ for all $k > m$.  In order to state a result on the joint distribution of $(T_{n1},\dots,T_{np})$ we require some addition notation. In the prototypical population path $\LAR(\Xmat,\bmu)$ involving a sequence of $m\leq p$ columns of $\Xmat$ (collected by the indices $\calA_m \equiv \{j_1,\ldots,j_m\}$)  and corresponding step correlations $C_1>\ldots>C_m>C_{m+1}=0$,  define the following random variables when   the  population path   terminates in   $m<p$ steps: letting $\Rmat$ denote the correlation matrix corresponding to  the covariance matrix $\bSigma$ of Assumption \ref{assum:regassumptions}, define 
   \begin{equation}
\label{eqn:zvec}
\zvec = (Z_j, j \in \{1,\dots,p\}\setminus \calA_m)^T \sim \calN(\zerovec,\sigma^2\Vmat),
\end{equation} 
as a multivariate normal vector in $\mathbb{R}^{p-m}$ with covariance matrix $\Vmat \equiv \Rmat_{\{1,\ldots,p\}\setminus \calA_m|\calA_m}$ corresponding to the conditional correlations of the
$p-m$ nonactive variables in the population path given the $m$ active variables indexed by $\calA_m$.   In addition, for each $j \notin \calA_m$ and collection of distinct indices $\{\pi_1,\dots,\pi_i\} \subset \{1,\dots,p\}\setminus \{\calA_m \cup \{j\}\}$ for $i <p-m$, define
\begin{equation}
\label{eqn:innov}
I_{j}(\zvec) \equiv  \frac{Z_{j}}{\sqrt{\Var (Z_j)}} \quad \text{ and } \quad I_{\{\pi_1,\dots,\pi_i\},j}(\zvec) \equiv \frac{Z_j - \E(Z_j | Z_{\pi_1},\dots,Z_{\pi_i})}{\sqrt{\Var(Z_j|Z_{\pi_1},\dots,Z_{\pi_i})}},
\end{equation}
which denote normal variables $Z_j$ from (\ref{eqn:zvec}) upon standardizing either marginally or conditionally on $Z_{\pi_1},\dots,Z_{\pi_i}$.
% , where $\mathbb{V} Z_j$ and $\mathbb{V}(Z_j|Z_{\pi_1},\dots,Z_{\pi_i})$ respectively denote marginal and conditional variances  and $\E(Z_j | Z_{\pi_1},\dots,Z_{\pi_i})$ denotes a conditional mean.  
Moreover, let $\Pi$ denote the collection of all permutations $\bpi \equiv (\pi_1,\dots,\pi_{p-m})$ of the indices $\{1,\dots,p\}\setminus \calA_m$ and let $\{\calO_{\bpi}, \bpi \in \Pi\}$ be a partition of $\mathbb{R}^{p-m}$ corresponding to values of  $\zvec$ that are compatible with $\bpi$ as a potential ordering   of the $p-m$ variables   that are never active in the population path; see the Appendix for   details. Lastly, let $\mathbb{I}$ denote an indicator function. We can now state a result on the joint distribution of the step correlation quantities $(T_{n1},\ldots,T_{np})$ from (\ref{eqn:T}).

\begin{theorem}
\label{thm:Tasymp}
Under  Theorem \ref{thm:larconsistency} conditions, it holds that 
\[
(T_{n1},\dots,T_{np})\indist (T_1,\dots,T_p)   \quad \mbox{as $n\to\infty$}, 
\] 
where
(i) $(T_1,\dots,T_m)^T \sim \calN(\zerovec,\sigma^2\Imat_m)$; (ii) $(T_1,\dots,T_m)$ is independent of $(T_{m+1},\dots,T_p)$; and (iii) $\sum_{k=m+1}^p T^2_k/\sigma^2 \sim \chi^2_{p-m}$ holds for $(T_{m+1},\dots,T_p)$  equal in distribution to
      \[
      \sum_{\bpi \in \Pi} \Big(I_{\pi_1}(\zvec),I_{\{\pi_{1}\},\pi_{2}}(\zvec) , \dots, I_{\{\pi_{1},\dots,\pi_{p-m-1}\},\pi_{p-m}}(\zvec) \Big)\mathbb{I}(\zvec \in \calO_{\bpi}).
      \] 
\end{theorem}

Theorem~\ref{thm:Tasymp} characterizes the distinct distributional  behaviors  in    estimated step correlation quantities $\{T_{nk}\}_{k=1}^p$ in (\ref{eqn:T}) from LAR, where a dichotomy arises depending on whether sample step correlations are targeting non-zero population  step correlations or not.  
Remarkably, despite the fact that   data steps in a LAR sample path are interrelated, all estimates $\hat{C}_k$ of non-zero population correlations $C_k$, for $1 \leq k\leq m$,  turn out to be independent and identically distributed normal variables in large samples, while also being independent of all remaining estimates $\hat{C}_k$, $m+1 \leq k\leq p $ when $m<p$
(i.e., Theorem~\ref{thm:Tasymp}(i)-(ii)). The explanation for this is that the sample LAR path is consistent for the variable order and signs in the LAR population path (Theorem~\ref{thm:larconsistency});  when these are correctly identified,    we may use \eqref{eqn:Ckangles} along with \eqref{eqn:Aupdate}  to write
\[
T_{nk} =\frac{\evec_k^T\berror_n}{\|\evec_k\|} \quad\mbox{for $k=1,\ldots,m$}  
\]
using orthogonal vectors $\evec_1,\ldots,\evec_m$   from (\ref{eqn:in}), which  are then asymptotically  $\calN(0,\sigma^2)$ distributed under Assumption \ref{assum:regassumptions}.  On the other hand, if there are $p$ columns of $\Xmat$ but the LAR population path terminates in $m<p$ steps, then the last $p-m$ steps of the LAR sample path are essentially competing to estimate noise: all remaining estimated correlations are targeting zero with no population ordering of the corresponding non-active $p-m$  variables, so the sample path    admits these remaining variables in an order selected randomly from among all possible orders.     As a result,  the asymptotic distribution of the collection $\{T_{nk}\}_{k=m+1}^p$ when $m<p$ 
has a non-normal joint distribution with a complicated dependence structure in Theorem~\ref{thm:Tasymp}(iii), which can be viewed as relating to order statistics; see Remark~\ref{rem:Tkasymp}.

\begin{remark}
\label{rem:Tkasymp}
Some observations about Theorem~\ref{thm:Tasymp}:
\begin{enumerate}
    \item When the columns of $\Xmat$ are orthonormal, that is when $\Xmat^T\Xmat = \Imat$, the limiting  variables $(T_{m+1},\dots,T_p)$ (i.e., corresponding to last $p-m$ steps of the sample LAR path) are equal in distribution to 
      \[
      \sum_{\bpi \in \Pi} (Z_{\pi_1},Z_{\pi_2}, \dots, Z_{\pi_{p-m}})\mathbb{I}(|Z_{\pi_1}| > |Z_{\pi_2}| > \dots > |Z_{\pi_{p-m}}|),
      \]
    where $Z_{m+1},\dots,Z_p$ are independent $\calN(0,\sigma^2)$ random variables. That is, the limit distribution of $(T_{m+1},\ldots,T_p)$ is given by taking a  random sample of $p-m$ 
    $\calN(0,\sigma^2)$ variables and arranging these from largest to smallest by their absolute values.
    \item Due to Slutsky's theorem, 
     all statements hold with $\sigma=1$ 
     upon replacing   ``$T_{n1},\dots,T_{np}$'' with ``$T_{n1}/ \hat \sigma_n,\dots,T_{np}/\hat \sigma_n$'', using   $\Pmat_\Xmat$ as the projection matrix of $\Xmat$ to write  
    \begin{equation}
    \label{eqn:varest}
    \hat \sigma_n^2 = (n-p)^{-1} \yvec_n^T(\Imat - \Pmat_\Xmat)\yvec_n.
    \end{equation}

\item Despite the complex dependence among $\{T_{nk}\}_{k=m+1}^p$ when $m<p$, it holds in large samples that  $\sum_{k=m+1}^{p}T_{nk}^2/\hat\sigma^2_n$ is $\chi_{p-m}^2$-distributed and $\sum_{k=1}^{p}T_{nk}^2/\hat\sigma^2_n$ is $\chi_{p}^2$-distributed.
\end{enumerate}

\end{remark}

\subsection{Formally estimating  the termination point for the LAR sample path}
\label{sec:mest}

We conclude this section by introducing an estimator $\bar m$ of the number $m$ of non-zero step correlations $C_k$ in a prototypical LAR population path as in Definition \ref{def:1}. Such an estimator is useful for deciding when terminate the sample LAR algorithm in practice and has been formally missing. To construct the estimator, we define the quantities $\hat W_{n,k} \equiv n(\hat A_k^{-2} - \hat A_{k-1}^{-2})\hat C_k^2/\hat \sigma_n^2$ for $k=1,\dots,p$, which resemble squares of the step correlation quantities $T_{nk}$ from (\ref{eqn:T}) when assuming each population   $C_k=0$  and then scaling by $\hat \sigma_n^2$ from (\ref{eqn:varest}). We then define tail sum statistics 
\begin{equation}
\label{eqn:Sk}
\hat S_{n,k} \equiv \sum_{j=k}^p\hat W_{n,j},\quad k=1,\dots,p.
\end{equation}
By Theorem~\ref{thm:Tasymp} (cf.~3.~of Remark \ref{rem:Tkasymp}), note that $\hat S_{n,m+1}$ at an index $k=m+1$ in (\ref{eqn:Sk}) must approximately have  a central $\chi^2_{p-m}$ distribution when $m<p$ (even if $m=0$).  On the other hand,  for any  given smaller index $k=1,\ldots,m$, the sum $\hat{S}_{n,k}$ will approximately have  a non-central $\chi^2_{p-k+1}(\varphi_{n,k})$ distribution with a non-centrality parameter $\varphi_{n,k} \equiv \sum_{j=k}^{m}  n(A_j^{-2} - A_{j-1}^{-2})C_j^2/\sigma^2$ that diverges   as $n\to \infty$.     
Our estimator of $m$, to follow, attempts to locate the largest index $k$ at which $\hat S_{n,k}$ appears to exhibit a non-central distribution.
  If $\hat S_{n,1} < \chi^2_{p,n^{-1}}$,  we set $ \bar m \equiv  0$ and, otherwise,   we define
\begin{equation}
\label{eqn:mest}
\bar m \equiv  \max\big\{m' \in \{1,\dots,p\}: ~ \hat S_{n,k}>\chi^2_{p-k+1,n^{-1}} ~ \text{ for all  } k = 1,\dots,m'\big\},
\end{equation}
where $\chi^2_{k,n^{-1}}$ denotes the upper $n^{-1}$ quantile of the (central) $\chi^2_k$ distribution.   The estimator $\bar{m}$ is illustrated graphically in Section~\ref{sec:illu}, while Theorem~\ref{thm:mest}  establishes formal consistency.

\begin{theorem}
\label{thm:mest}
Under the Theorem \ref{thm:larconsistency} conditions, $P(\bar m = m)\to 1$ as $n \to \infty$.    
\end{theorem}

Note that in \eqref{eqn:mest} we will have $\hat S_{n,m+1} > \chi_{p-m,n^{-1}}^2$ holding with probability approximately $n^{-1}$ for large $n$, so that the probability of overestimating $m$ is expected to vanish quickly as $n \to \infty$.  If $m >0$, then on each step $k=1,\dots,m$ the tail sum statistic $\hat S_{n,k}$, which has a noncentrality parameter growing at rate $n$, will exceed the threshold $\chi_{p-k+1,n^{-1}}^2$ with probability tending to one, as the latter quantity grows at rate $\log n$ (cf.~proof of Theorem \ref{thm:mest} in the Appendix). While the thresholds $\chi^2_{p-k+1,n^{-1}}$ in \eqref{eqn:mest} appear to be a natural choice in light of the fact that the tail-sum statistics approximately have chi-squared distributions, the consistency of $\bar m$ in Theorem \ref{thm:mest} can be established under more general threshold specifications.  For example, one could replace ``$\chi^2_{p-k+1,n^{-1}}$'' in \eqref{eqn:mest} with ``$q_n$'' where the latter represents any positive sequence satisfying $q_n/n + 1/q_n \to 0$ as $n \to \infty$, noting that the threshold need not depend on the step $k$. This condition guarantees $\hat S_{n,m+1} > q_n$ with vanishing probability by requiring $q_n$ to diverge and, if $m>0$, guarantees $\hat S_{n,k} > q_n$ with probability tending to one for each $k=1,\dots,m$ by requiring $q_n$ to grow at a rate slower than $n$.   In all our simulations and data analyses, we have used the thresholds  $\chi^2_{p-k+1,n^{-1}}$ appearing in \eqref{eqn:mest}.

\section{Practical inference from LAR}
\label{sec:practicalinference}

We next consider practical applications of the results in Section \ref{sec:Tk}.  We begin by introducing a typical graphical depiction of the LAR steps, similar to that presented in Figure 3 of \cite{efron2004least}.  Figure \ref{fig:facetemp_lar} depicts the sample path $\LAR(\Xmat,\yvec)$ from a data set with $n=933$ observations of a response variable and $p=18$ predictor variables from a study on body temperature measurements  \cite{wang2023facial}  (see Section \ref{sec:facetemp} for more details).
The left panel in Figure~\ref{fig:facetemp_lar} shows the absolute correlations $|(\hat \cvec_k)_j|$ for   the variable indices $j=1,\dots,18$ across all steps of the algorithm $k=1,\dots,18$. In each step, we observe that a new variable joins the active set $\hat \calA_k$ as its absolute correlation $|(\hat \cvec_k)_j|$ becomes equal to the current maximum absolute correlation or {\it step correlation} $\hat C_k$, which decreases in each step. The right panel in Figure~\ref{fig:facetemp_lar} plots what we will call   \textit{step coefficients} from the sample path $\LAR(\Xmat,\yvec)$.  These are the non-zero entries of  $p\times 1$ vectors $\hat \bvec_k$ for expressing the LAR step prediction $\hat\yvec_k$  in Algorithm~\ref{alg:lar_efron} as  a linear combination 
$\Xmat \hat \bvec_k = \hat \yvec_k$  of regressor variables at each $k= 1,\dots,p$.  Setting $\hat \bvec_0 = \zerovec$,  we may obtain $\hat \bvec_k$ for each step $k=1,\dots,p$ by the recursion
\begin{figure}
    \centering
    \includegraphics[width=0.9\linewidth]{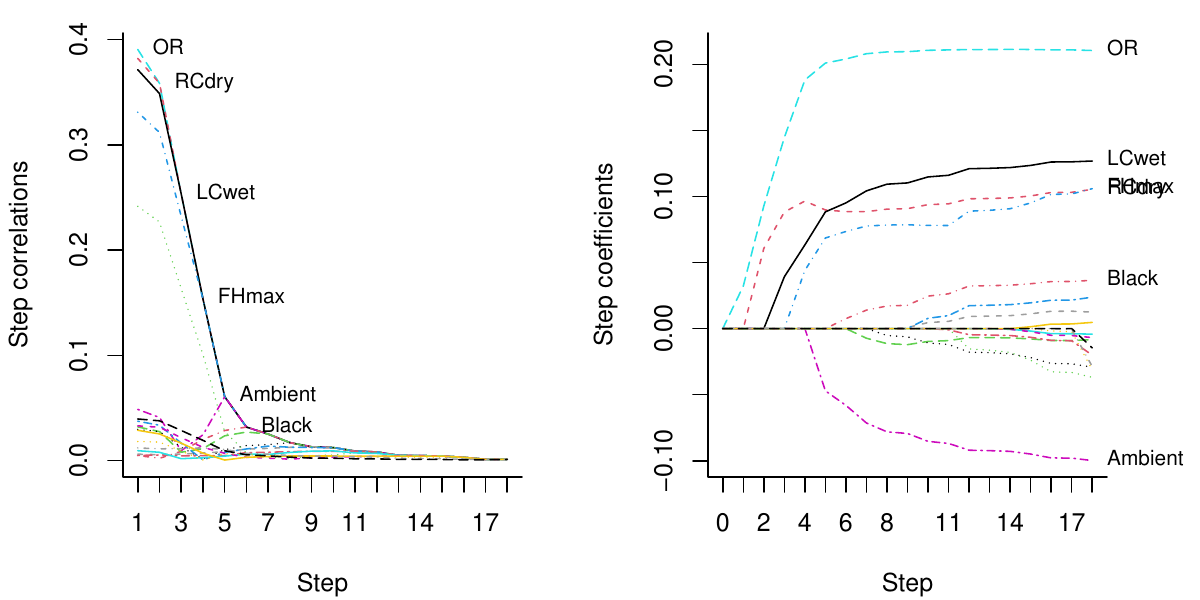}
    \caption{Sample path $\LAR(\Xmat,\yvec)$ on face temperature data described in Section \ref{sec:facetemp}.}
    \label{fig:facetemp_lar}
\end{figure}
\begin{equation}
\label{eqn:bk}
(\hat \bvec_k)_{\hat \calA_k} = (\hat \bvec_{k-1})_{\hat \calA_k} + \hat \gamma_k (\Xmat_{\hat \calA_k}^T\Xmat_{\hat \calA_k})^{-1}\Xmat^T_{\hat \calA_k} \hat \avec_k \quad \text{ and } \quad (\hat \bvec_k)_j = 0 \text{ for } j \notin \hat \calA_k.
\end{equation}
In the right panel of Figure \ref{fig:facetemp_lar}, a new variable begins contributing to the predictor $\hat \yvec_k$ when its step coefficient becomes non-zero. From both panels, the first five or six steps intuitively appear as most important: 
after these steps, the  correlations $\hat C_k$ seem to be relatively  small, as do the coefficients of   newcomer-variables  to  linear predictors $\hat \yvec_k$.  Only  those variables entering on the first six steps have been labeled  in Figure \ref{fig:facetemp_lar}, as we estimate $\bar m = 6$  based on \eqref{eqn:mest}   for these data; Section \ref{sec:facetemp} gives more details.

In the following, we consider how to use the results from Section \ref{sec:Tk} to make inference about step correlations $C_k$ and   step coefficients $\bvec_k$  from LAR population path $\LAR(\Xmat,\bmu)$ using estimates, as illustrated above, given by the LAR sample path $\LAR(\Xmat,\yvec)$.

\subsection{Bootstrap for step correlations}
\label{sec:stepcorrelations}

We first consider an approach to inference on the population step correlations $C_k$ based on bootstrap, where the latter is practically important here 
due to the complicated distributions of estimators in Theorem~\ref{thm:Tasymp}. 
  Given an observed response $\yvec$ and matrix $\Xmat$,   
the standard residual bootstrap commonly 
defines a bootstrap version
$\berror^*$ of an $n\times 1$  homoscedastic, uncorrelated error vector  $\berror$ by  drawing $n$ entries uniformly and with replacement from centered/scaled   residuals  $( \hat \berror - n^{-1} \onevec^T \hat \berror)/\sqrt{n/(n-p)}$ given by least squares estimation $\hat \berror \equiv (\Imat - \Pmat_\Xmat)\yvec$. Under weak conditions,
 bootstrap errors $\berror_n^*\equiv n^{1/2}\berror^* $ also exhibit the same distributional behavior given in (R3) of Assumption~\ref{assum:regassumptions} for the original errors $\berror_n \equiv n^{1/2} \berror$ 
(cf.~\cite{freedman1981bootstrapping,bickel1983bootstrapping,mammen2012does}), which we also assume here as a baseline for bootstrapping (cf.~Theorem~\ref{thm:bootstrap}); see also Remark \ref{rem:bootresid}.  However, to develop a valid bootstrap with LAR, 
it turns out that we must deviate from the standard residual bootstrap  when defining   bootstrap responses    $\yvec^* \equiv \bar \bmu + \berror^*$, which critically require a {\it bootstrap response mean} $\bar \bmu$ as analog of  the population  mean $\bmu$.   Typically in residual bootstrap for linear models, the role of the bootstrap mean  $\bar \bmu$ is played by the least squares estimator $\hat{\yvec} \equiv \Pmat_\Xmat\yvec$, though this standard  choice will   generally  {\it fail} 
for inference about     step correlations; see also Remark~\ref{rem:mbar}.  The reason for this failure is that 
the least squares estimator 
$\hat{\yvec}$ does not mimic  the possibility in the bootstrap world that  some LAR population correlations may be zero, where the latter aspect impacts the joint distribution     of the  
 estimated entrance   correlations (Theorem~\ref{thm:Tasymp}).   Hence, the standard residual bootstrap must be appropriately modified in order to be valid for LAR step correlations.

To this end, given an estimate $\bar m$ of the number $m$ of non-zero population correlations  from \eqref{eqn:mest}, define the bootstrap mean as $\bar \bmu  \equiv \hat \Pmat_{\bar m} \yvec$ in bootstrap responses
  $\yvec^* \equiv \bar \bmu + \berror^*$,
  where $\hat \Pmat_{\bar m}$ is the orthogonal projection defined by the columns $\{\xvec_j, j \in \hat \calA_{\bar m}\}$ active in the sample path at step $\bar m$. Additionally, define bootstrap versions of population  step correlations as $\bar C_k \equiv \hat C_k \mathbb{I}(k \leq \bar m)$, using an indicator function $\mathbb{I}$ to threshold the original $\LAR(\Xmat,\yvec)$   estimators $\hat{C}_k$ to zero for indices $k>\bar{m}$.
  We then   estimate the joint distribution of the studentized quantities $(\hat T_{n1},\dots,\hat T_{np})   \equiv (T_{n1},\dots,T_{np})/\hat{\sigma}_n$ 
  based on  $(T_{n1},\dots,T_{np})$ from (\ref{eqn:T}) and $\hat \sigma_n^2$ from (\ref{eqn:varest}) with a bootstrap counterpart $(\hat T^*_{n1},\dots,\hat T^*_{np})$ defined by 
\[
\hat T_{nk}^* \equiv \hat s_k^*\bigg(\frac{1}{\hat A_{k}^{*2}}- \frac{1}{\hat A_{k - 1}^{*2}}\bigg)^{1/2}\sqrt{n}\big( \hat C^*_k - \bar C_k \big)/\hat \sigma^*_n,\quad k = 1,\dots,p,
\]
where $\hat{C}_k^*$, $\hat{s}_k^*$, and $\hat A_k^*$ for $k =1,\dots,p$ (with $\hat A_0^* \equiv\infty$)  respectively denote step correlations, associated signs, and angles from the bootstrap sample path $\LAR(\Xmat,\yvec^*)$ with $\yvec^*\equiv \hat \Pmat_{\bar m} \yvec+ \berror^*$ and where $\hat{\sigma}_n^{*2} \equiv (n-p)^{-1} \yvec_n^{*T}(\Imat - \Pmat_\Xmat)\yvec_n^*$ for $\yvec_n^* \equiv n^{1/2} \yvec^*$. 
The following 
result formally establishes 
that the proposed bootstrap is valid for $(\hat T_{n1},\dots,\hat T_{np})$, despite the quite
complicated and potentially non-normal distribution of the latter.

\begin{theorem}
\label{thm:bootstrap}
Under the conditions of Theorem \ref{thm:larconsistency} with $\bar m$ as in \eqref{eqn:mest}
and  assuming bootstrap errors $\berror_n^*$
satisfy Assumption~\ref{assum:regassumptions}(R3) in probability, the bootstrap procedure is consistent for the distribution of estimated step correlations:
\[
\sup_{\calB \in \calB(\mathbb{R}^p)} \Big|\PP_*\big((\hat T^*_{n1},\dots,\hat T^*_{np}) \in \calB\big) -  \PP\big((\hat T_{n1},\dots,\hat T_{np})\in \calB\big ) \Big| \inprob 0 \quad \mbox{as $n\to \infty$},
\] 
where $\PP_*$ denotes bootstrap probability.
\end{theorem}

We may then construct an approximate $(1-\alpha)100\%$ bootstrap confidence interval for population step correlation $C_k$  as follows for   $k=1,\dots,p$.  Define an  interval  as
\begin{equation}
\label{eqn:Ik}
\hat I_k \equiv \left\{\begin{array}{rl}
     [\hat C_k - \hat T_{nk}^*(1-\alpha/2)\hat q_k, \hat C_k - \hat T_{nk}^*(\alpha/2)\hat q_k ],~& \hat s_k = 1  \\[2pt]  
    ~[\hat C_k - \hat T_{nk}^*(\alpha/2)\hat q_k, \hat C_k - \hat T_{nk}^*(1-\alpha/2)\hat q_k ],~& \hat s_k = -1, \\  
\end{array}\right.
\end{equation}
where $\hat q_k \equiv  \hat s_k(\hat A_k^{-2} - \hat A_{k-1}^{-2})^{1/2}\hat \sigma_n/\sqrt{n}$ and $\hat T_{nk}^*(1-\alpha/2)$ and $\hat T_{nk}^*(\alpha/2)$ denote the upper and lower $\alpha/2$ quantiles, respectively, of the bootstrap distribution of $\hat T_{nk}^*$.  Because  $C_k$ is non-negative, any negative lower endpoint  of $\hat{I}_k$ may be replaced by zero.   The intervals $\hat I_k$, $k=1,\dots,p$, will have asymptotically correct coverage by  Theorem~\ref{thm:bootstrap}.

\begin{remark}
\label{rem:bootresid}
Theorem 20 again justifies a modified residual bootstrap for approximating the distribution of LAR sample output, though this result   presupposes that the standard residual bootstrap is valid for mimicking the same normal limit assumed for the  original homoscedastic, uncorrelated model errors in Assumption~\ref{assum:regassumptions}(R3).   This assumption is plausible  with the model errors as white noise because residual bootstrap also recreates white noise by independently resampling residuals.  
For   a more formal justification,  with  
Assumptions~\ref{assum:regassumptions}(R1)--(R2),   suppose 
$\max_{1 \leq i \leq n } \|\rvec_{ni}\|/\sqrt{n} \rightarrow 0$ as $n\to \infty$ with $\rvec_{ni}$ as the $i$th row of the design matrix $\Xmat_n$ (e.g., holding  if rows were  iid generated with a finite second moment and then fixed) and, for the   errors $\berror_n^T =(\varepsilon_{1,n},\ldots, \varepsilon_{n,n})$,  suppose  $\limsup_{M\to \infty} \limsup_{n\to \infty} n^{-1} \sum_{i=1}^n  \E\varepsilon_{i,n}^2 \mathbb{I}(|\varepsilon_{i,n}| \geq M) =0$ (i.e., a uniform integrability of the average second moment) and $  \mathrm{Var}(n^{-1} \sum_{i=1}^n \varepsilon_{i,n}\mathbb{I}(|\varepsilon_{i,n}| \leq M)) \rightarrow 0$ for any given $M>0$ as $n\to \infty$ (i.e., a weak dependence condition). The latter are automatic  for iid errors with finite variance,  but can also allow many
other types of white noise distributions for the errors in (R1). 
Under these conditions, one may establish that
Assumption~\ref{assum:regassumptions}(R3)
follows, and that standard residual bootstrap  errors also  satisfy  (R3) in probability, through     verifying Lindeberg's condition.  If model errors are heteroscedastic or exhibit serial correlations, then Assumption~\ref{assum:regassumptions}(R3)   can fail 
along with the residual bootstrap for re-creating errors. 
\end{remark}

\subsection{Bootstrap for step coefficients}
\label{sec:bootstepcoefficients}

% Given the sequence of sample path active sets $\hat \calA_1,\dots,\hat \calA_p$, define the quantities
In  a prototypical $\LAR(\Xmat,\bmu)$ population path from Definition \ref{def:1},  there exists a series of $p\times 1$ vectors  $\bvec_k$, defined by the steps $k=1,\ldots,m$ 
of Algorithm~\ref{alg:lar},
whereby the LAR population approximation $\bmu_k = \Xmat \bvec_k $ of the response mean $\bmu$ at step $k$ can be written as a linear combination  of the $p$ columns of the design matrix $\Xmat$.  The vectors $\bvec_k$, $k=1,\ldots m$, represent population {\it step coefficients}, whereby  each $\bvec_k$ has $k$ non-zero elements.  We may express $\bvec_k$ recursively over population steps $k=1,\ldots,m$ as  
\begin{equation}
\label{eqn:bk2}
( \bvec_k)_{ \calA_k} = (  \bvec_{k-1})_{  \calA_k} +   \gamma_k (\Xmat_{  \calA_k}^T\Xmat_{ \calA_k})^{-1}\Xmat^T_{  \calA_k}   \avec_k \quad \text{ and } \quad ( \bvec_k)_j = 0 \text{ for } j \notin  \calA_k
\end{equation}
using $\bvec_0 \equiv \mathbf{0}$. We   define     $\bvec_k \equiv \bvec_m$ for  $k \geq m$ as there are no population steps beyond  $m \leq p$.  Just as we have not explicitly indicated in our notation the dependence on $n$ of the step correlations $C_k \equiv C_{k}(n)$ and vectors $\cvec_k \equiv \cvec_{k}(n)$, we have likewise written the step coefficients as $\bvec_k$ instead of $\bvec_{k}(n)$.  To see how the step coefficients depend on $n$, note in particular that at step $m$ we have $\Xmat\bvec_m =\Pmat_\Xmat \bmu$, since $\bmu_m = \Pmat_\Xmat \bmu$ by Lemma \ref{lem:muk}. Now (R4) in Assumption \ref{assum:regassumptions} gives $\bmu = n^{-1/2}\Xmat_n \bbeta_n$ so that $\Pmat_\Xmat \bmu = \bmu$, whereby $\Xmat \bvec_m = n^{-1/2}\Xmat_n\bbeta_n$.  From here we see that $(\bvec_m)_j = (\|\xvec_{nj}\|/\sqrt{n}) (\bbeta_n)_j$ for each $j=1,\dots,p$, so that the entries in $\bvec_m$, which we may refer to as the ``terminal'' vector of population step coefficients, are equal to the entries of $\bbeta_n$ rescaled according to the norms of the original design columns. 

The sample step coefficients $\hat{\bvec}_k$, given in (\ref{eqn:bk}) based on the $\LAR(\Xmat,\yvec)$ sample path,  can now be seen to  estimate population step coefficients $\bvec_k$ in 
 (\ref{eqn:bk2}) from the 
 $\LAR(\Xmat,\bmu)$ population path. Based on an estimated
number $\bar m$ of non-zero population correlations  from \eqref{eqn:mest}, the most natural estimator of the terminal vector of population step coefficients $\bvec_m$ is not the sample path step coefficient vector $\hat \bvec_k$ at $k=\bar m$, but rather the vector $\bar \bvec_{\bar m}$ (say) such that $\Xmat \bar \bvec_{\bar m}= \hat \Pmat_{\bar m}\yvec $, which has entries given by $(\bvec_{\bar m})_{\calA_{\bar m}} = (\Xmat_{\calA_{\bar m}}^T\Xmat_{\calA_{\bar m}})^{-1}\Xmat_{\calA_{\bar m}}^T \yvec$ and $(\bvec_{\bar m})_j = 0$ for $j \notin\calA_{\bar m}$.
The vector $\bar \bvec_{\bar m}$ is the step coefficient vector resulting from using $\hat \gamma_{\bar m} = \hat C_{\bar m} / \hat A_{\bar m}$ at step $\bar m$, which updates the prediction $\hat \yvec_{\bar m - 1}$ to the projection $\hat \Pmat_{\bar m}\yvec$ (cf. Proposition \eqref{prop:LARknown}(iv)). This is the action one would take if one knew that the population step correlations $C_k$ were equal to zero beyond step $\bar m$.
% corresponds to the projection update with $\hat \gamma_{\bar m} = \hat C_{\bar m} / \hat A_{\bar m}$ (cf. Proposition \eqref{prop:LARknown}(iv)). This is because if one suspects the step correlations $C_k$ should be zero beyond step $\bar m$, then one would update $\hat \yvec_{\bar m - 1}$ with the choice $\hat \gamma_{\bar m} = \hat C_{\bar m}/\hat A_{\bar m}$, resulting in the orthgonal projection $\bar \bmu = \hat \Pmat_{\bar m}\yvec$ at the completion of step $\bar m$.  
In order to make inferences on population step coefficients $\bvec_k$, we therefore  consider the active sets $\hat \calA_1,\dots,\hat \calA_{\bar m}$ along the $\LAR(\Xmat,\yvec)$ sample path and define quantities 
\[
B_{nk,j} \equiv \sqrt{n}( (\hat \bvec_k)_j - (\bvec_k)_j) / \hat \sigma_n, \quad  j \in \hat \calA_k,\quad k= 1,\dots,\bar m,
\]
where, for $k=\bar m$ we re-define $\hat \bvec_{\bar m}$ as $\hat \bvec_{\bar m} \equiv \bar \bvec_{\bar m}$; for clarity, note that there are $|\hat \calA_k|=k$ such quantities $B_{nk,j}$ at each step $k$.  
Based on  the continuous mapping theorem with Theorems~\ref{thm:larconsistency} and \ref{thm:Tasymp}, combined  with the fact that weights $\hat{\gamma}_k = [\hat{C}_{k}-\hat{C}_{k+1}]/\hat{A}_{k}$ in Algorithm~\ref{alg:lar_efron}
determine step coefficients $\hat \bvec_k$ (cf.~Proposition~\ref{prop:LARknown}(iii) and Remark~\ref{rem:2}), a well-defined  joint limit distribution  can be shown to exist for the above quantities  $\{B_{nk,j} :j \in \hat{\mathcal{A}}_k, k=1,\ldots,\bar{m}\}$, where this limit, while  normal, has a complicated covariance structure; see the Appendix for details.  Fortunately, the same bootstrap procedure from Section~\ref{sec:stepcorrelations}  applies, and it holds that the   consistency of the bootstrap for step correlations in Theorem~\ref{thm:bootstrap} translates to   consistency for step coefficients as well.

Given $\bar m$,  define  bootstrap responses
  $\yvec^* \equiv \bar \bmu + \berror^*$ based on a bootstrap version of the mean as $\bar \bmu  \equiv \hat \Pmat_{\bar m} \yvec$  
as before. Then let $\hat \bvec^*_k$, $k=1,\dots,\bar m$ be the vectors of step coefficients from the first $\bar m$ steps of the bootstrap sample path $\LAR(\Xmat,\yvec^*)$, with $\hat \bvec^*_{\bar m}$ re-defined as $\hat \bvec^*_{\bar m} \equiv \bar \bvec^*_{\bar m}$ in analogy to the re-definition of $\hat \bvec_{\bar m}$ as $ \hat \bvec_{\bar m} \equiv \bar \bvec_{\bar m}$ at the sample level.
% where $\bar \bvec^*_{\bar m} = (\Xmat_{\hat \calA^*_{\bar m}}^T\Xmat_{\hat \calA^*_{\bar m}})^{-1}\Xmat_{\hat \calA^*_{\bar m}}^T\yvec^*$.
% and define corresponding bootstrap versions of step coefficients $\bvec_k$  by $\hat \bvec_k$, where for $k<\bar m$ and  $\hat \bvec_k \equiv  \bar \bvec_{\bar m} $   for $k = m$
% in terms of sample coefficients $\hat \bvec_k$ as well as $\bar \bvec_{\bar m}   $ 
% as the solution to $\Xmat \bar \bvec_{\bar m} = \hat \Pmat_{\bar m} \yvec \equiv \bar \bmu$ (i.e., entries are $(\bar \bvec_{\bar m})_j = (\Xmat_{\bar{m}}^T \Xmat_{\bar{m}})^{-1} \Xmat_{\bar{m}}^T \yvec)_j  $ 
% for $j\in\hat\calA_{\bar m}$ and  $(\bar \bvec_{\bar m})_j =0$ otherwise). 
Based on the sequence of sample path active sets $\hat \calA_1,\dots,\hat\calA_{\bar m}$, we approximate the distribution of coefficient quantities $B_{nk,j}$, $j \in \hat \calA_k$, $k = 1,\dots, \bar m$ with the distribution of
bootstrap counterparts given as 
\begin{equation}
\label{eqn:Bstar}
B^*_{nk,j} =\sqrt{n}( (\hat \bvec^*_k)_j - (\hat \bvec_k)_j) / \hat \sigma^*_n, \quad  j \in \hat \calA_k, \quad k= 1,\dots,\bar m.
\end{equation}
% where $\hat \bvec^*_k$, $k=1,\dots,\bar m$, are the vectors of step coefficients from the first $\bar m$ steps of the bootstrap sample path $\LAR(\Xmat,\yvec^*)$.
We    then construct $(1-\alpha)100\%$ bootstrap confidence intervals for the population path step coefficients $(\bvec_k)_j$ for $j \in \hat \calA_k$ on steps $k =1,\dots,\bar m$ as
\begin{equation}
\label{eqn:Jkj}
\hat J_{k,j} \equiv \Big[(\hat \bvec_k)_j - B^*_{nk,j}(1-\alpha/2) \frac{\hat \sigma_n}{\sqrt{n}}, (\hat \bvec_k)_j - B^*_{nk,j}(\alpha/2)\frac{\hat \sigma_n}{\sqrt{n}}\Big],
\end{equation}
where $B_{nk,j}^*(1-\alpha/2)$ and $B_{nk,j}^*(\alpha/2)$ denote the upper and lower $\alpha/2$ quantiles, respectively, of the bootstrap distribution of $B_{nk,j}^*$ and $\hat \bvec_{\bar m } \equiv \bar \bvec_{\bar m}$ as already stated.

% In practice, we then regard $\bar \bvec_{\bar {m}}$ satisfying $\Xmat \bar \bvec_{\bar {m}} = \hat \Pmat_{\bar m}\yvec$ as a point estimator of the population step coefficient vector $\bvec_m$ at its 
% terminal step $m$ satisfying $\Xmat \bvec_m = \bmu$, and one need not make   inference  about coefficients $\bvec_k$ when $k >m$ as $\bvec_k \equiv \bvec_m$  for $k > m$.
% We do not, however, attempt to construct confidence intervals for step coefficients beyond step $m$.  This is because for all $j \notin \calA_m$ we have $(\bvec_k)_j = 0$ for all $k \geq m$. So, trusting our estimate $\bar m$, we simply regard $(\bvec_{\bar m})_j$ for $j \notin \calA_{\bar m}$ as zero-valued. If one wishes to construct from the inferred LAR path an estimator of the terminal predictor in the population path, we propose using $\bar \bmu = \Xmat \bar \bvec$ as the estimator of $\bmu_m = \Xmat \bvec_m$ and $\bar \bvec$ as the point estimator of $\bvec_m$. One can then use as confidence intervals for $(\bvec_m)_j$, $j \in \hat \calA_{\bar m}$ the intervals $\hat J_{\bar m,j}$, $j \in \hat \calA_{\bar m}$.
As    estimators $\bar m $   and  $\hat \calA_1,\dots, \hat \calA_{\bar{m}}$ converge in probability  to their population analogs $m$ and 
  $\calA_1,\dots,\calA_m$, we do not need to make any adjustments to the intervals in \eqref{eqn:Jkj} to account for the ``selection'' of $\bar m$ and $\hat \calA_1,\dots, \hat \calA_m$ and, furthermore,  the bootstrap intervals in   \eqref{eqn:Jkj}  will    achieve their nominal coverage levels in large samples.  We  prescribe no confidence intervals for step coefficients at steps $k>\bar{m}$ exceeding the estimate $\bar{m}$ because such intervals lose meaning.  This owes partly to the fact that 
$(\bar{m},\hat\calA_{\bar{m}})$
and $(m, \calA_{m})$ match asymptotically, while population step coefficients $\bvec_k=\bvec_m$ are constant for $k>m$ and   require no estimation.       Additionally, over steps $k >m$, intervals for step coefficients also  lack appropriate coverage interpretations (technically, 
unlike the $k \leq m$ steps,  
the entries $(\hat \bvec_k)_j$ when $k>m$ have  conditional distributions, such as 
$\sqrt{n}( (\hat \bvec_{k})_j - (\bvec_{k})_j) / \hat \sigma_n | \hat \calA_{k} \setminus \hat \calA_{k-1} = \{j\}$  involving  conditioning on   events that remain random in large samples,  which the  bootstrap does not capture).  
% \rd{So: once we have $\bar m$ in hand, should we then re-define the sample estimate of $\hat \bvec_k$ for $k \geq \bar m$ to be $\hat \bvec_{k} = \bar \bvec_{\bar{m}} $ and then define the bootstrap counterpart as  $\hat \bvec_{k}^* = \bar \bvec_{\bar{m}}^* $ for $k \geq \bar m$?   Then, fit/move the following text to above (taken from Section~5.3) into the above:\\
% Given the estimate $\bar m$ of the number $m$ of non-zero step correlations, $\bar \bvec$ is the most natural point estimator of the population step coefficients at the terminal step. This is because if one knew that the step correlations $C_k$ should be zero beyond step $\bar m$, one would update $\hat \yvec_{\bar m - 1}$ with the choice $\hat \gamma_{\bar m} = \hat C_{\bar m}/\hat A_{\bar m}$, resulting in the predictor $\bar \bmu = \hat \Pmat_{\bar m}\yvec$ at the completion of step $\bar m$. 

\begin{remark}
\label{rem:mbar}
The standard residual bootstrap  (i.e., using the least squares estimator $\hat{\yvec} \equiv \Pmat_\Xmat\yvec$ as the bootstrap mean $\bar \bmu$) would  only be  valid here in the specific case that the number $m \leq p$ of non-zero population correlations matches the number $p$ of variables.  Essentially, all  $p$ variables would need to appear in the LAR population path, which may not   hold.  Consequently, we   use a consistent estimator $\bar m$ of $m$
in order to modify the  bootstrap to better mimic non-zero correlations  at the population level, including when $m<p$.    In the bootstrap,  
  $\LAR(\Xmat, \bar \bmu)$
with  
$\bar \bmu  \equiv \hat \Pmat_{\bar m}$  then  plays the role   of the true population path   $\LAR( \Xmat, \bmu)$. 

\end{remark}

\subsection{Visualizing the inferred LAR path}

\begin{figure}
    \centering
    \includegraphics[width=0.9\linewidth]{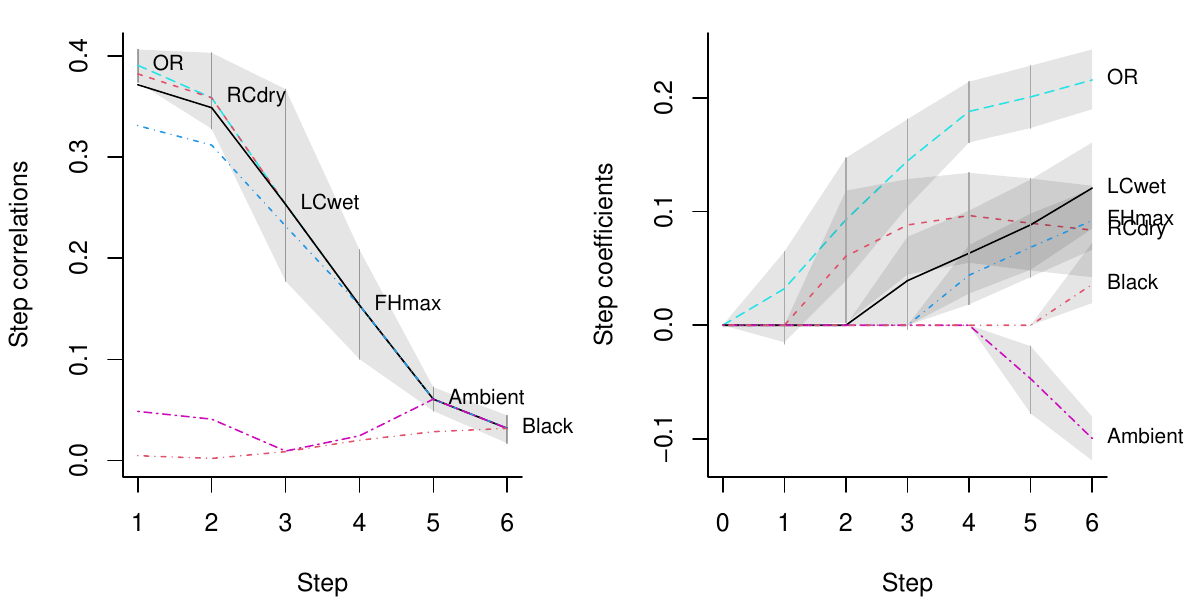}
    \caption{Inferred LAR path for the face temperature data based on $\bar m = 6$.}
    \label{fig:facetemp_larinf}
\end{figure}

Figure \ref{fig:facetemp_larinf} suggests a template for how to summarize inference results about the LAR population path.  This depicts the inferred LAR path underlying the face temperature data based on the estimate of $\bar m = 6$ non-zero population correlations.  The left panel plots the absolute correlations $|(\hat \cvec_k)_j|$ for the first $\bar m$ steps of the sample path (over indices $j\in \hat \calA_{\bar{m}}$ reflecting the $\bar m$ variables involved), and draws pointwise confidence intervals defined in \eqref{eqn:Ik} for the population step correlations $C_1,\dots,C_{\bar m}$;  
estimated  correlations $\hat{C}_k$ are connected with line segments across steps, while  
interval endpoints  are likewise connected  to form a ``tube'' for visual aid.  The right panel similarly shows the confidence intervals defined in \eqref{eqn:Jkj} for the population step coefficients $(\bvec_k)_j$ for indices in the sample path active sets on steps $k=1,\dots, \bar m$. The line segments trace, up to step $\bar m$, the sample path step coefficients $(\hat \bvec_k)_j$ for $j \in \hat\calA_{\bar m}$; at the final step $\bar m$,  recall terminal coefficients  $(\hat \bvec_{\bar{m}})_j = (\bar \bvec_{\bar{m}})_j$ for variables $j\in \hat \calA_{\bar{m}}$  correspond to coefficients of the least squares estimator based on those variables. 
Further discussion of   inference for the face temperature data is given in Section~\ref{sec:facetemp}.

\section{Simulation studies}
\label{sec:sim}

Here we present results of simulations designed to test the performance of our proposed inference methods under various settings.  We pay particular attention to the effect of the size of the margin $\delta$ specified in condition (M2) of Theorem \ref{thm:larconsistency}. Recall that $\delta$ fixes the margin by which the maximal entrance criterion $C_{k+1,j_{k+1}}$ must exceed the criteria $C_{k+1,j}$ of all competing indices $j$ for entrance on a step $k+1$. For smaller values of $\delta$, it is therefore anticipated that the sample path will be less likely to correctly recover the order in which the variables enter the population path active set. Moreover, since $\delta$ plays also the role of bounding the smallest nonzero entrance correlation $C_m$ away from zero, smaller values of 
$\delta$ entail greater difficulty in ascertaining from sample data the true number of steps $m$ in the population path.

In the following we investigate, under various settings, the coverage of the bootstrap intervals $\hat I_k$ in \eqref{eqn:Ik} for the step correlations $C_k$ and $\hat J_{k,j}$ in $\eqref{eqn:Jkj}$ for the step coefficients $(\bvec_k)_j$ as well as the rate at which the estimator $\bar m$ defined in Section \ref{sec:mest} under and overestimates $m$. In addition, we consider that, in practice, one may be most interested in estimating the terminal set of coefficients $\bvec_m$ corresponding to the projection $\Pmat_{\Xmat} \bmu = \Xmat \bvec_m$, regardless of whether one correctly estimates $m$ or correctly recovers the order of entrance of all variables in the population path. We therefore also investigate the coverage of bootstrap confidence intervals $\hat J_{\bar m,j}$ for the terminal population step coefficients $(\bvec_m)_j$ for indices $j \in \hat \calA_{\bar m}$ belonging to the sample path active set at the estimated terminal step $\bar m$. 

We generate each data set with an $n \times p$ design matrix $\Xmat$ and an $n \times 1$ response mean $\bmu$ satisfying Definition \ref{def:1} of a prototypical LAR path with some specified number $m$ of non-zero step correlations, where we consider all combinations of $n \in \{200,500,1000\}$, $p \in \{20,50,100\}$, and $m \in \{3,6\}$.  In order to study the effect of the size of the margin $\delta$ in (M2) of Theorem \ref{thm:larconsistency}, we require the population path $\LAR(\Xmat,\bmu)$ to satisfy (M2) for $\delta$ greater than or equal to some given threshold $\delta_0$, where we consider differing values $\delta_0\in\{0, 0.05,0.1,0.2\}$ for all combinations of $n$, $p$, and $m$. Note that under $\delta_0 = 0$ the condition (M2) is not imposed.

We construct each data set satisfying these conditions as follows: We generate an $n \times p$ matrix $\Xmat_n$ with rows given by independent realizations of the multivariate normal distribution with mean zero and covariance matrix $\bSigma = ((1/2)^{|i-j|})_{1 \leq i , j , \leq p}$. Then we construct a $p \times 1$ vector $\bbeta_n$ having entries equal to zero except at $m$ randomly chosen indices, where these $m$ values in $\bbeta_n$ are independently drawn from the uniform distribution on $[-2,2]$, and construct the response mean $\bmu_n$ as $\bmu_n \equiv  \Xmat_n \bbeta_n$. After standardizing $\Xmat_n$ and $\bmu_n$ as in (R4) of Assumption \ref{assum:regassumptions} to obtain $\Xmat$ and $\bmu$, we additionally center $\bmu$ and the columns of $\Xmat$ and compute  the path $\LAR(\Xmat,\bmu)$, which returns a number of non-zero step correlations $C_k$ as well as the largest value of $\delta$ for which the path satisfies (M2) of Theorem \ref{thm:larconsistency}.  If the former is equal to $m$ and the latter is greater than or equal to $\delta_0$, we keep $\Xmat$ and $\bmu$; otherwise, we discard $\Xmat$ and $\bmu$ and repeat the generation again, doing so until we obtain $\Xmat$ and $\bmu$ such that $\LAR(\Xmat,\bmu)$ satisfies the specified conditions.  Then we generate $\berror_n$ having independent $\calN(0,1)$ entries and set $\yvec_n = \bmu_n + \berror_n$. After centering $\yvec_n$ we set $\yvec \equiv n^{-1/2}\yvec_n$.  Note that we disregard condition (M1) of Theorem \ref{thm:larconsistency} in the generation of our datasets; that is, we do not check whether the path $\LAR(\Xmat,\bmu)$ satisfies (M1) for any particular values of $\tau$ and $\xi$. In each setting, we generate \numprint{1000} data sets.

On each data set thus constructed, we compute the sample path $\LAR(\Xmat,\yvec)$ and obtain active sets $\hat \calA_k$ as well as the estimate $\bar m$ of $m$ from (\ref{eqn:mest}). Then, based on $500$ Monte Carlo bootstrap draws, we construct 95\% confidence intervals $\hat I_k$ for $k=1,\dots,\bar m$ for the step correlations $C_1,\dots,C_{\bar m}$ as well as 95\% intervals $\hat J_{k,j}$ for $j \in \hat \calA_k$, $k = 1,\dots,\bar m$ for the step coefficients $(\bvec_k)_j$ over the same indices.  We record a realized coverage of these confidence intervals on each data set as follows:  For the intervals $\hat I_k$, $k=1,\dots,\bar m$, we record the proportion of the intervals which cover their targets as $\bar m^{-1}\sum_{k=1}^{\bar m} \mathbb{I}( C_k \in \hat I_k )$, where $\mathbb{I}$ denotes an indicator function; analogously for intervals $\hat J_{nk,j}$ we record a proportion of the intervals which covered their targets as
\[
  \frac{2}{\bar m(\bar m + 1)}\sum_{k=1}^{\bar m} \sum_{j \in \hat \calA_k}\mathbb{I}( (\bvec_k)_j \in \hat J_{k,j})
\]
over  indices $j \in \hat \calA_k$ and  steps $k=1,\dots,\bar m$, corresponding to a total of   $\bar m(\bar m +1)/2$ intervals.
Table \ref{tab:Ckbkcov} shows the average of such coverages at each  combination  of $n$, $p$, $m$, and threshold $\delta_0 \in \{0,0.05,0.10,0.20\}$ used to set $\delta \geq \delta_0$ in the LAR population  path condition (M2) of Theorem \ref{thm:larconsistency}.

% latex table generated in R 4.4.0 by xtable 1.8-4 package
% Wed Oct 14 16:14:00 2026
\begin{table}[ht]
\centering
\begin{tabular}{lll|llll|llll}
  \multicolumn{3}{c}{}&\multicolumn{4}{c}{$\{C_k, k=1,\dots, \bar{m}\}$}&\multicolumn{4}{c}{$\{(\mathbf{b}_k)_j, j \in \hat{\mathcal{A}}_k, k = 1,\dots,\bar{m}\}$}\\ 
  \multicolumn{3}{c|}{}&&\multicolumn{2}{c}{$\delta_0$}&&&\multicolumn{2}{c}{$\delta_0$}&\\ 
$n$ & $p$ & $m$ & $0$ & $0.05$ & $0.10$ & $0.20$ & $0$ & $0.05$ & $0.10$ & $0.20$ \\ 
  \hline
200 & 20 & 3 & 0.94 & 0.94 & 0.93 & 0.93 & 0.77 & 0.77 & 0.83 & 0.89 \\ 
   &  & 6 & 0.90 & 0.90 & 0.90 & 0.92 & 0.78 & 0.80 & 0.83 & 0.87 \\ 
   & 50 & 3 & 0.93 & 0.93 & 0.92 & 0.93 & 0.73 & 0.72 & 0.76 & 0.89 \\ 
   &  & 6 & 0.91 & 0.91 & 0.91 & 0.92 & 0.75 & 0.76 & 0.77 & 0.85 \\ 
   & 100 & 3 & 0.93 & 0.92 & 0.92 & 0.92 & 0.67 & 0.68 & 0.72 & 0.83 \\ 
   &  & 6 & 0.91 & 0.90 & 0.91 & 0.91 & 0.70 & 0.72 & 0.74 & 0.81 \\ 
  500 & 20 & 3 & 0.94 & 0.93 & 0.94 & 0.95 & 0.82 & 0.83 & 0.90 & 0.93 \\ 
   &  & 6 & 0.91 & 0.92 & 0.92 & 0.94 & 0.82 & 0.85 & 0.89 & 0.92 \\ 
   & 50 & 3 & 0.94 & 0.94 & 0.94 & 0.95 & 0.78 & 0.81 & 0.88 & 0.93 \\ 
   &  & 6 & 0.93 & 0.92 & 0.93 & 0.94 & 0.79 & 0.82 & 0.87 & 0.92 \\ 
   & 100 & 3 & 0.94 & 0.93 & 0.94 & 0.95 & 0.75 & 0.75 & 0.83 & 0.94 \\ 
   &  & 6 & 0.93 & 0.93 & 0.93 & 0.93 & 0.77 & 0.78 & 0.84 & 0.91 \\ 
  1000 & 20 & 3 & 0.94 & 0.94 & 0.94 & 0.95 & 0.85 & 0.89 & 0.94 & 0.94 \\ 
   &  & 6 & 0.93 & 0.93 & 0.94 & 0.94 & 0.85 & 0.88 & 0.92 & 0.93 \\ 
   & 50 & 3 & 0.94 & 0.94 & 0.94 & 0.94 & 0.84 & 0.85 & 0.93 & 0.94 \\ 
   &  & 6 & 0.93 & 0.93 & 0.94 & 0.95 & 0.83 & 0.86 & 0.91 & 0.94 \\ 
   & 100 & 3 & 0.95 & 0.95 & 0.95 & 0.94 & 0.79 & 0.83 & 0.90 & 0.94 \\ 
   &  & 6 & 0.93 & 0.94 & 0.94 & 0.94 & 0.81 & 0.84 & 0.90 & 0.94 \\ 
\end{tabular}
\caption{Coverage of $95\%$ intervals for population step correlations/coefficients.} 
\label{tab:Ckbkcov}
\end{table}

We see that the coverages are closer to the nominal $0.95$ rate for larger sample sizes $n$ and larger $\delta_0$. While the number $m$ 
of non-zero step correlations has some small effect, in contrast,  
the threshold $\delta_0$ appears to have more influence on the nominal performance of  bootstrap confidence intervals.  Under a larger threshold $\delta_0$, the LAR sample path is more often able to reproduce the LAR population path and its sequence of  active sets, as perhaps expected. 
This study of confidence intervals   
is also configured to accommodate the possibility that the estimator
$\bar m$ may not match $m$; for example, for data simulations with $\bar m >m$, intervals then need to capture $C_k=0$ or entries $(\bvec_k)_j =(\bvec_m)_j$ over $m <k \leq \bar m$.  In this manner,   given a possibly incorrect estimate $\bar m$ of $m$ and the LAR sample path sequence, we have examined making overall correct inference on the underlying  population path step correlations and coefficients up to and including the estimated terminal step $\bar m$.
Covering   individual variable coefficients $(\bvec_k)_j$, for all variables  eventually admitted and  across   all  intermediate steps    before   termination   
(in contrast to at termination step $\bar{m}$, cf.~Table \ref{tab:bmcov}), is 
more difficult than covering  step correlations $C_k$.  
Before termination, the LAR algorithm does not know  what variables will ultimately be admitted,  impacting coverage of all such $(\bvec_k)_j$.   In contrast, much better coverage is possible for step correlations, as  direct output from LAR,  when accommodating for cases where estimated active sets   $\hat \calA_k$   may not match   population path counterparts $\calA_k$.     

Table \ref{tab:mest} next summarizes the performance of the estimator $\bar m$ from \eqref{eqn:mest}, reporting  the proportion of simulated data sets for which $\bar m$ underestimated and overestimated $m$.  We see that, for smaller thresholds $\delta_0$, underestimation of $m$ is more likely. This is not unexpected, as the margin $\delta$ in (M2) not only fixes a margin of separation between the highest and next-highest entrance criterion at each population step, but also fixes a lower bound for the smallest nonzero population correlation.  When this bound is small or nonexistent ($\delta_0 = 0$), nonzero step correlations in the population path are allowed to be arbitrarily small, such that the contributions of these steps are imperceptible in the sample data. That is, for small $\delta_0$, the estimator $\bar{m}$ may often judge from the data to stop admitting variables before the true number $m$ of population steps, because  population  step correlations are so small that the LAR sample path cannot distinguish these from zero.  Practically, however, this behavior also implies that, when the LAR sample path decides to terminate at step $\bar{m}$, the data algorithm believes that it has essentially captured the response mean (the value at the terminal point of the LAR population path), and this often holds essentially true as illustrated next.

\begin{table}[ht]
\centering
\begin{tabular}{lll|ll|ll|ll|ll}
  \multicolumn{3}{c}{}&\multicolumn{2}{c}{$\delta_0 = 0$}&\multicolumn{2}{c}{$\delta_0 = 0.05$}&\multicolumn{2}{c}{$\delta_0 = 0.10$}&\multicolumn{2}{c}{$\delta_0 = 0.20$}\\ 
$n$ & $p$ & $m$ & $< m$ & $> m$ & $< m$ & $> m$ & $< m$ & $> m$ & $< m$ & $> m$ \\ 
  \hline
200 & 20 & 3 & 0.41 & 0.02 & 0.30 & 0.02 & 0.17 & 0.02 & 0.07 & 0.01 \\ 
   &  & 6 & 0.57 & 0.07 & 0.41 & 0.11 & 0.29 & 0.09 & 0.19 & 0.05 \\ 
   & 50 & 3 & 0.45 & 0.01 & 0.36 & 0.03 & 0.26 & 0.03 & 0.08 & 0.02 \\ 
   &  & 6 & 0.65 & 0.05 & 0.53 & 0.07 & 0.44 & 0.08 & 0.23 & 0.04 \\ 
   & 100 & 3 & 0.49 & 0.05 & 0.41 & 0.05 & 0.32 & 0.04 & 0.14 & 0.06 \\ 
   &  & 6 & 0.71 & 0.07 & 0.59 & 0.08 & 0.53 & 0.08 & 0.34 & 0.07 \\ 
  500 & 20 & 3 & 0.30 & 0.01 & 0.18 & 0.01 & 0.07 & 0.00 & 0.02 & 0.00 \\ 
   &  & 6 & 0.45 & 0.05 & 0.27 & 0.05 & 0.14 & 0.03 & 0.09 & 0.01 \\ 
   & 50 & 3 & 0.34 & 0.00 & 0.23 & 0.01 & 0.10 & 0.00 & 0.02 & 0.01 \\ 
   &  & 6 & 0.56 & 0.01 & 0.37 & 0.02 & 0.19 & 0.02 & 0.07 & 0.01 \\ 
   & 100 & 3 & 0.37 & 0.01 & 0.31 & 0.00 & 0.17 & 0.01 & 0.02 & 0.00 \\ 
   &  & 6 & 0.60 & 0.01 & 0.48 & 0.02 & 0.31 & 0.01 & 0.10 & 0.01 \\ 
  1000 & 20 & 3 & 0.23 & 0.00 & 0.08 & 0.00 & 0.01 & 0.00 & 0.00 & 0.00 \\ 
   &  & 6 & 0.38 & 0.04 & 0.16 & 0.04 & 0.07 & 0.01 & 0.04 & 0.01 \\ 
   & 50 & 3 & 0.23 & 0.00 & 0.14 & 0.00 & 0.03 & 0.00 & 0.00 & 0.01 \\ 
   &  & 6 & 0.45 & 0.01 & 0.26 & 0.01 & 0.08 & 0.01 & 0.02 & 0.00 \\ 
   & 100 & 3 & 0.31 & 0.00 & 0.18 & 0.00 & 0.06 & 0.00 & 0.00 & 0.00 \\ 
   &  & 6 & 0.48 & 0.00 & 0.31 & 0.01 & 0.13 & 0.01 & 0.03 & 0.00 \\ 
\end{tabular}
\caption{Under and overestimation of $m$ by $\bar m$.} 
\label{tab:mest}
\end{table}

 Table \ref{tab:bmcov} summarizes  the coverage performance of intervals for step coefficients at the estimated terminal step $\bar m$ as follows:  
 focusing  on the terminal estimated active set $\hat \calA_{\bar m}$ in a simulation run, we  compute 
the proportion  $\sum_{j \in \hat \calA_{\bar m}} \mathbb{I}((\bvec_m)_j \in \hat J_{\bar m,j})/\bar m$
of   intervals $\hat J_{\bar m,j}$, $j \in \hat \calA_{\bar m}$, at the estimated terminal step $\bar m$
that contain the corresponding true LAR   step coefficients $(\bvec_m)_j$ defined by the population value $m$,  and  report  this average proportion over the \numprint{1000} simulated data sets. We see that these intervals exhibit close-to-nominal coverage even under setting in which, referring to Table~\ref{tab:mest}, the estimator $\bar m$ tends to underestimate the true number of steps $m$. 
 % Table \ref{tab:mest} indicates that the threshold $\delta_0$ has an impact on the probability that $\bar m$ correctly estimates $m$. For smaller $\delta_0$, correct estimation of $m$ appears more likely under larger $n$, smaller $p$, and smaller $m$.  The intervals $\hat J_{\bar m,j}$ for $(\bvec_m)_j$, $j \in \hat \calA_{\bar m}$ exhibit   good performance even under settings in which $\bar m$  has  difficulty in correctly estimating $m$.  At step $\bar m$, the sample LAR path will have admitted, with high probability, all important variables into the active set, so that, whether these were admitted in the correct order or not over the preceding steps, the terminal population coefficients can be well-estimated.

% latex table generated in R 4.4.0 by xtable 1.8-4 package
% Wed Oct 14 16:14:00 2026
\begin{table}[ht]
\centering
\begin{tabular}{lll|llll}
  \multicolumn{3}{c}{}&&\multicolumn{2}{c}{ $\delta_0$}&\\
$n$ & $p$ & $m$ & $0$ & $0.05$ & $0.10$ & $0.20$ \\ 
  \hline
200 & 20 & 3 & 0.93 & 0.92 & 0.93 & 0.93 \\ 
   &  & 6 & 0.91 & 0.92 & 0.93 & 0.93 \\ 
   & 50 & 3 & 0.92 & 0.93 & 0.93 & 0.94 \\ 
   &  & 6 & 0.92 & 0.92 & 0.93 & 0.93 \\ 
   & 100 & 3 & 0.92 & 0.92 & 0.92 & 0.92 \\ 
   &  & 6 & 0.91 & 0.91 & 0.92 & 0.93 \\ 
  500 & 20 & 3 & 0.93 & 0.93 & 0.94 & 0.94 \\ 
   &  & 6 & 0.92 & 0.94 & 0.94 & 0.94 \\ 
   & 50 & 3 & 0.94 & 0.94 & 0.94 & 0.94 \\ 
   &  & 6 & 0.93 & 0.93 & 0.94 & 0.94 \\ 
   & 100 & 3 & 0.93 & 0.93 & 0.94 & 0.95 \\ 
   &  & 6 & 0.93 & 0.94 & 0.94 & 0.93 \\ 
  1000 & 20 & 3 & 0.93 & 0.94 & 0.94 & 0.94 \\ 
   &  & 6 & 0.92 & 0.93 & 0.94 & 0.94 \\ 
   & 50 & 3 & 0.93 & 0.94 & 0.94 & 0.94 \\ 
   &  & 6 & 0.93 & 0.94 & 0.94 & 0.95 \\ 
   & 100 & 3 & 0.94 & 0.95 & 0.95 & 0.94 \\ 
   &  & 6 & 0.94 & 0.94 & 0.94 & 0.94 \\ 
\end{tabular}
\caption{Coverage of 95\% confidence intervals for terminal coefficients $(\bvec_m)_j$, $j \in \hat \calA_{\bar m}$.} 
\label{tab:bmcov}
\end{table}

 To summarize all of these findings, it is clear that the margin $\delta$ in (M2) of Theorem \ref{thm:larconsistency} plays an important role when one wishes to recover the true order of entrance of the variables in the population path or the true number of steps $m$ in the population path. Under smaller values of $\delta$, the sample path must choose between variables separated by smaller margins in the competition to enter at each step, and our estimator of the number of true population steps $m$ may miss latter steps for which population step correlations are still nonzero but small.  On the other hand, if the focus of inference is on the reconstruction of the response mean $\bmu \equiv \Pmat_{\Xmat}\bmu = 
 \bmu_m= \Xmat \bvec_m = n^{-1/2}\Xmat_n \bbeta_n$, it appears, from Table~\ref{tab:bmcov}, that one can still construct reliable confidence intervals for entries in $\bvec_m$ even under very small values of $\delta$.  While one tends to underestimate $m$ for small $\delta$, the coverage of entries in $\bvec_m$ remained close to nominal in all settings considered.  In other words,  in those situations where  the smallest nonzero population step correlations may be close enough to zero to result in underestimation of $m$, such population correlations tend also be so small as to contribute only negligibly to the response mean $\bmu$; in which case, those steps at the tail of the population path with near-zero step correlations appear to be negligible at the data level, having little effect on the accuracy of overall mean inference.

\section{Illustrations on real data}
\label{sec:illu}

\subsection{Face temperature data}
\label{sec:facetemp}
The face temperature data, for which Figure \ref{fig:facetemp_lar} depicts the sample LAR path and for which Figure \ref{fig:facetemp_larinf} depicts the inferred LAR path, are publicly available on PhysioNet \citep{PhysioNet} at \url{https://physionet.org/content/face-oral-temp-data/1.0.0/} \citep{wang2023facial}. These data were used in \cite{wang2021infrared} to study the use of infrared thermography for detecting elevated body temperature.
The data include on each of $n = 933$ subjects an oral temperature reading, which we use as the response in $\yvec_n$, as well as several measurements derived from infrared thermographs of the subject's face, the values of some demographic variables, the subject's distance to the thermal camera, the relative humidity, the ambient temperature, and an indicator of whether the subject was wearing cosmetics, which we place in the design matrix $\Xmat_n$. We do not use all the thermographic temperature variables in the data set, as some are very highly related to others, being taken in the same facial vicinity (e.g.~measurements at several locations on the forehead).  The number of covariates  considered was $p = 18$. After standardizing $\yvec_n$ and $\Xmat_n$ to obtain $\yvec$ and $\Xmat$ as in (R4) of Assumption \ref{assum:regassumptions} we additionally center $\yvec$ and $\Xmat$ before proceeding.

% latex table generated in R 4.4.0 by xtable 1.8-4 package
% Tue May 20 10:53:54 2025
\begin{table}[ht]
\centering
\begin{tabular}{rrrrrrrrrr}
 Variable & $\hat S_{n,k}$ & $\chi^2_{p-k+1,n^{-1}}$ & $\hat C_k$ & 2.5\% & 97.5\% & $\hat \bvec_{\bar m}$ & 2.5\% & 97.5\% \\ 
  \hline
OR & 2403.291 & 42.097 & 0.391 & 0.373 & 0.408 & 0.216 & 0.190 & 0.243 \\ 
RCdry & 460.676 & 40.578 & 0.358 & 0.325 & 0.403 & 0.084 & 0.042 & 0.123\\ 
  LCwet & 210.710 & 39.044 & 0.253 & 0.170 & 0.354  & 0.121 & 0.085 & 0.161\\ 
  FHmax & 180.969 & 37.493 & 0.154 & 0.103 & 0.210 & 0.092 & 0.067 & 0.120 \\ 
  Ambient & 150.449 & 35.922 & 0.061 & 0.049 & 0.072 & $-0.100$ & $-0.119$ & $-0.081$\\ 
  Black & 44.194 & 34.331 & 0.032 & 0.016 & 0.045 & 0.036 & 0.020 & 0.073\\ \hline
  White & 27.992 & 32.716 & 0.025 & 0.000 & 0.061 \\ 
  Male & 22.874 & 31.075 & 0.017 & 0.002 & 0.033 \\ 
  Distance & 16.910 & 29.403 & 0.013 & 0.000 & 0.027 \\ 
  Asian & 14.217 & 27.697 & 0.012 & 0.003 & 0.022 \\ 
  Age31to40 & 10.015 & 25.949 & 0.009 & 0.000 & 0.021 \\ 
  FHtc & 8.789 & 24.151 & 0.008 & 0.003 & 0.012 \\ 
  Hispanic & 2.035 & 22.292 & 0.005 & 0.000 & 0.023 \\ 
  Cosmetics & 1.924 & 20.355 & 0.005 & 0.000 & 0.013 \\ 
  Humidity & 1.425 & 18.313 & 0.004 & 0.000 & 0.016 \\ 
  Age18to20 & 1.154 & 16.119 & 0.003 & 0.000 & 0.017 \\ 
  Age21to25 & 1.062 & 13.677 & 0.001 & 0.000 & 0.004 \\ 
  Age26to30 & 0.799 & 10.699 & 0.001 & 0.000 & 0.002 \\ 
\end{tabular}
\caption{Results of inference on the face temperature LAR path. We obtain $\bar m = 6$ from the rule in \eqref{eqn:mest}.}
\label{tab:facetemp}
\end{table}

Table \ref{tab:facetemp} lists the variables from the face temperature data in the order in which they entered the sample path active set.  The tail sum statistics $\hat S_{n,k}$ from (\ref{eqn:Sk}) and the thresholds $\chi^2_{p-k+1,n^{-1}}$ are shown for all sample path steps $k=1,\dots,p$ as well as the step correlations $\hat C_k$ and the lower and upper limits of the confidence intervals $\hat I_k$; from the output in the table, we see that the rule in \eqref{eqn:mest} suggests there are $\bar m = 6$ non-zero step correlations in the LAR population path.  For these variables, the rightmost three columns give the non-zero entries of the coefficient vector $\hat \bvec_{\bar m} = \bar \bvec_{\bar m}$, which serves as a point estimate for the vector of terminal step coefficients $\bvec_m$ in the LAR population path, as well as the lower and upper bounds of confidence intervals $\hat J_{\bar m ,j}$, $j \in \hat \calA_{\bar m}$. Note that most of the confidence intervals for $C_k$, $k > 6$, have lower limits equal or close to zero in Table~\ref{tab:facetemp}, also suggesting that $\bar m = 6$ is likely equal to the true value of $m$ in the population path.  We note that all bootstrap confidence intervals were based on 500 Monte Carlo draws.

\begin{figure}
    \centering
    \includegraphics[width=0.9\linewidth]{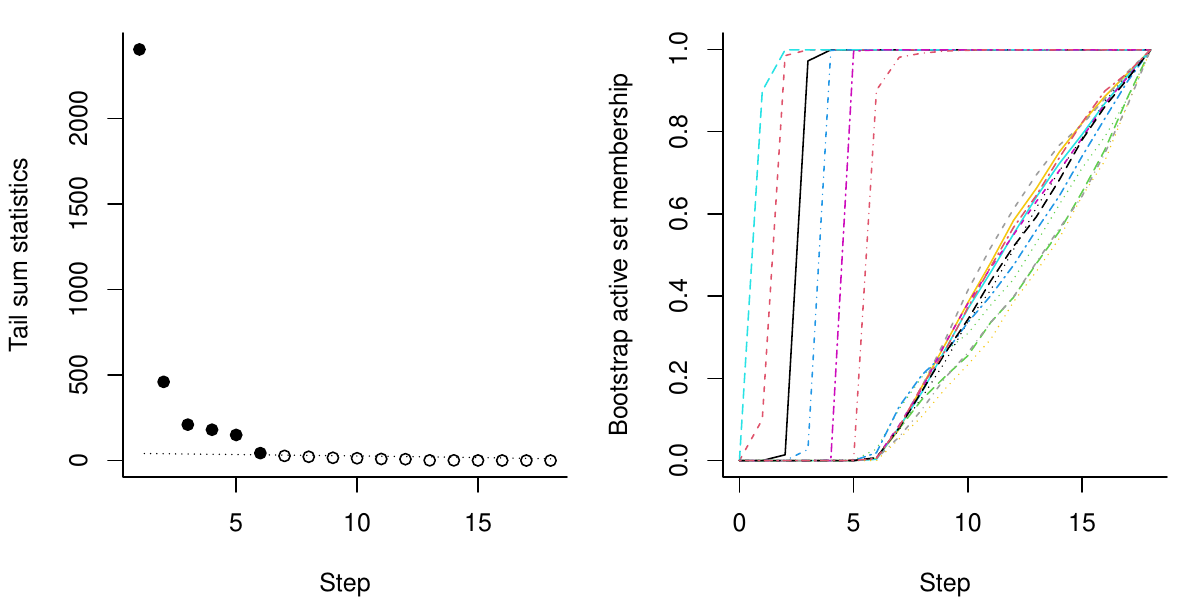}
    \caption{Left: Estimation of $m$. Right: Bootstrap probability of active set membership for each variable index.} 
    \label{fig:facetemp_larinf2}
\end{figure}

Figure \ref{fig:facetemp_larinf2} shows some additional details  regarding inference on the LAR path   in connection to the variable ordering in Table \ref{tab:facetemp}.  The left panel of Figure \ref{fig:facetemp_larinf2}  depicts how the estimate $\bar m$ from \eqref{eqn:mest} is computed by plotting the tail sum statistics  $\hat S_{n,k}$ from \eqref{eqn:Sk}, as listed in Table~\ref{tab:facetemp}, against the step $k$; the dotted line traces the points $(k,\chi^2_{p-k+1,n^{-1}})$ for $k=1,\dots,p$. We see that step~$7$ is the first step on which $\hat S_{n,k}$ falls below the dotted line, resulting in $\bar m = 6$.  Across steps $k = 1,\dots,p$, the right panel of Figure \ref{fig:facetemp_larinf2} displays the proportion of times, out      
  of the 500 Monte Carlo bootstrap samples, that  each variable entered the bootstrap sample path or active set $\hat \calA_{k}^*$ on step $k$; each variable  corresponds to a distribution curve in this plot
  that increases to one as the step  $k$ increases to $p$, due to the fact that a variable stays active once joining an active set at a step in a given bootstrap sample.  
    We see that the first 6 variables entered nearly always in the same order, while the remaining variables appear to have entered in random order over the remaining steps. This plot further suggests that for these data  the LAR sample path can reliably discern  which and when relevant variables should enter.

The results in this section can be reproduced using the R package `larinf', available at \url{https://github.com/gregorkb/larinf}.

    %, suggesting that the conditions (M1) and (M2) of Theorem \ref{thm:larconsistency} are well satisfied.

\subsection{Diabetes data}
\label{sec:diabetes}

The diabetes data provide an example used in the original LAR paper \cite{efron2004least}, though without any inference there, which we can now provide. The data contain the values of $p = 10$ predictors on $n = 442$ patients as well as a continuous response which is a measure of disease progression.  The data set, publicly available in the R package `lars' \citep{larspackage}, contains a design matrix $\Xmat$ with columns already centered and normalized and a response vector $\yvec_n$, which we center and then rescale as in (R4) of Assumption \ref{assum:regassumptions} to obtain $\yvec$. On these data, we estimated the number of non-zero step correlations as $\bar m = 5$. Figure \ref{fig:diabetes_larinf} depicts our inference on the population LAR path (cf. Figure 3 of \cite{efron2004least}), while  Table \ref{tab:diabetes} provides
corresponding  numerical details in the same fashion as Table \ref{tab:facetemp}.  
See the Appendix for additional summaries of the diabetes data analysis.

\begin{figure}
    \centering
    \includegraphics[width=0.9\linewidth]{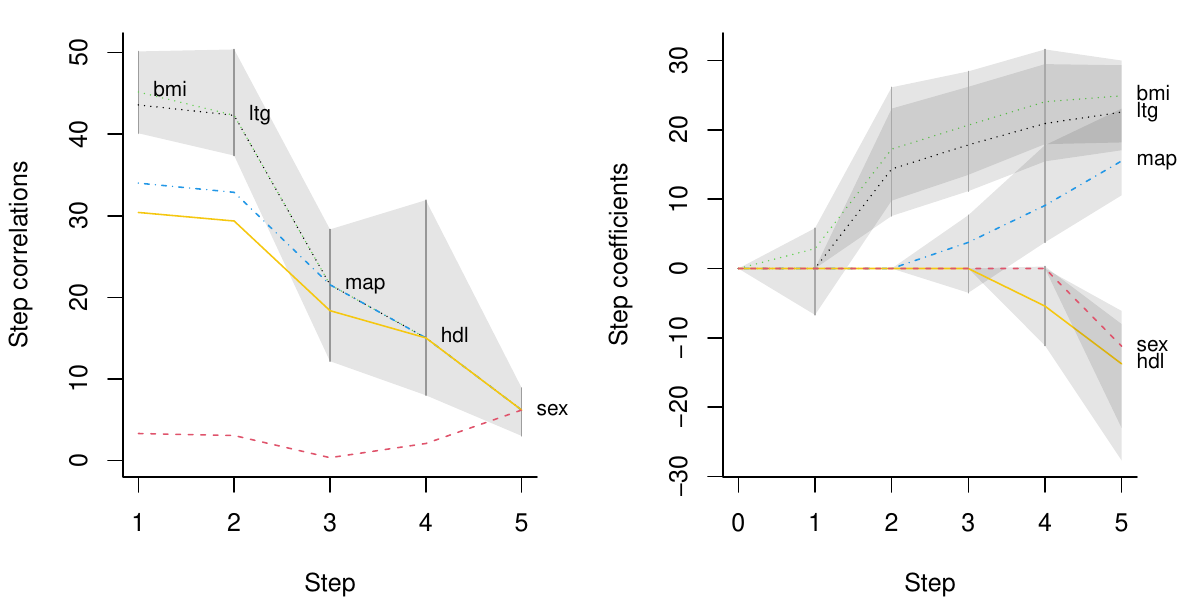}
    \caption{Inferred LAR path from the diabetes data based on $\bar m = 5$.}
    \label{fig:diabetes_larinf}
\end{figure}

% latex table generated in R 4.4.0 by xtable 1.8-4 package
% Tue May 20 11:17:27 2025
\begin{table}[ht]
\centering
\begin{tabular}{rrrrrrrrr}
  Variable & $\hat S_{n,k}$ & $\chi^2_{p-k+1,n^{-1}}$ & $\hat C_k$ & 2.5\% & 97.5\% & $\hat \bvec_{\bar m}$ & 2.5\% & 97.5\% \\ 
  \hline
  bmi & 463.800 & 27.385 & 45.160 & 39.806 & 49.124 & 24.903 & 18.186 & 29.997 \\ 
  ltg & 155.713 & 25.729 & 42.300 & 37.516 & 49.848 & 22.560 & 17.060 & 29.319 \\ 
  map & 52.193 & 24.033 & 21.542 & 11.581 & 28.600 & 15.517 & 10.566 & 23.063\\ 
  hdl & 33.742 & 22.291 & 15.034 & 9.128 & 24.955 & $-13.752$ & $-27.697$ & $-8.015$\\
  sex & 23.515 & 20.492 & 6.190 & 2.929 & 8.862 & $-11.215$ &  $-23.019$ & $-6.119$\\ \hline 
  glu & 8.167 & 18.619 & 4.223 & 0.000 & 31.976 \\ 
  tc & 7.466 & 16.648 & 3.280 & 0.000 & 5.990 \\ 
  tch & 2.835 & 14.533 & 0.950 & 0.000 & 2.919 \\ 
  ldl & 1.994 & 12.183 & 0.261 & 0.000 & 0.607 \\ 
  age & 0.028 & 9.323 & 0.242 & 0.000 & 2.612 \\ 
\end{tabular}
\caption{Inference results  on the diabetes LAR path. We obtain $\bar m = 5$ from the rule in \eqref{eqn:mest}.}
\label{tab:diabetes}
\end{table}

% \section{Further technical aspects about LAR}
% \label{sec:further}

\section{Mid-path ties in the population path}
\label{sec:midpathties}
In Section~\ref{sec:lar}, we  introduced the LAR population path $\LAR(\Xmat,\bmu)$ from Algorithm~\ref{alg:lar}  and described
the notion of a prototypical population path (Definition~\ref{def:1}) in which  exactly one variable $j_k$ enters the active set with a non-zero step correlation $C_k>0$ on step $k$, over a series of steps $k=1,\ldots,m$ until the algorithm ends at a terminal step $m\leq p$. 
A prototypical population path thus disallows entries of multiple variables on any step.  Here we aim to better explain that the LAR data algorithm is intended for such prototypical population paths, at least in some default sense, because LAR sample and population paths can {\it fail} to match outside of this setting to the extent that inference based on the LAR sample path can break down.

For example, if we do not have   a prototypical population path $\LAR(\Xmat,\bmu)$, then  two or more variables (being not active at some step $k-1$) could   simultaneously join the active set $\mathcal{A}_k$ on some single step $k$ 
in the path $\LAR(\Xmat,\bmu)$ with a positive step correlation $C_k>0$.  
We refer to this occurrence as a \textit{mid-path tie},
which is allowed in the population path under Algorithm~\ref{alg:lar}.   
A problem, though, 
is that    mid-path ties can {\it never} happen  at the data level with a continuous response $\yvec$; that is,  a LAR sample path $\LAR(\Xmat,\yvec)$  will always admit only {\it one} new variable per step.    Consequently, if the LAR algorithm at the population level has mid-path ties, then the data level simply cannot capture such ties, because the LAR sample path   takes  separate steps to admit   variables  entering at the population level on one step.   
% To complicate matters further, the sample path may also admit such mid-tie variables in a random order. 

To understand the consequences of this, consider a population path which admits one variable at a time on steps $1,\dots,k-1$ but then simultaneously admits $l > 1$ variables on step $k$.  One might suppose that the sample path would admit these same $l$ variables on $l$ consecutive steps $k,\dots,k+l-1$ and that each of the corresponding sample step correlations $\hat C_k,\dots,\hat C_{k+l-1}$ would serve as an estimator of a single population counterpart $C_k$.  We find, however, that this need not be the case. The population path updates the angle $A_{k-1}$ and equiangular vector $\avec_{k-1}$ to $A_k$ and $\avec_k$ in a single step admitting $l$ new variables to the active set simultaneously; in contrast, the sample path must update $\hat A_{k-1}$ and $\hat \avec_{k-1}$ by adding one variable at a time to the active set, resulting in sequences $\hat A_{k},\dots \hat A_{k + l - 1}$ and $\avec_k,\dots, \hat \avec_{k+l-1}$ of sample path angles and equiangular vectors which depend on the order in which the variables enter. As a result, it is possible for the step correlations $\hat C_k,\dots,\hat C_{k+l-1}$ to exhibit dramatically different behavior according to the order in which the variables enter, which, even for large sample sizes, may not be stable.  In this case $\hat C_k,\dots,\hat C_{k+l-1}$ \textit{cannot} serve as estimators of a single population-level counterpart $C_k$.  Indeed, in a LAR sample path, it is also possible that variables other than the $l$ tied ones from the population level to enter on  sample steps $k,\dots,k+l-1$, as the entrance criterion in \eqref{eqn:Ckplus1j} also depends on the sequence of sample path angles and equiangular vectors. Key issues in the case of mid-path ties are thus that i) the order of entrances of variables in the sample path may remain unstable even for large sample sizes and ii) the inability of the sample path to admit multiple variables simultaneously in a single step causes a breakdown in the correspondence between sample and population step correlations. For these reasons, the sample LAR path cannot be said to meaningfully estimate the population LAR path when the latter has mid-path ties.

As a formal treatment of mid-path ties 
is beyond our study as well as an open issue to us,   we provide a numerical illustration involving one mid-path tie with two variables that 
enter on a population step.   In particular, consider  a setting with $p=4$ variables, in which variable $\xvec_1$ enters the LAR population path on step 1 with positive sign, while $\xvec_2$ and $\xvec_3$ enter simultaneously on step~2 with positive signs and variable $\xvec_4$ does not enter the population path (i.e., has entrance correlation zero).  To construct such a path, we created a  
design matrix $\Xmat$ by generating and normalizing $\Xmat_n$ having  $n = 500$ independent $\calN(\zerovec,\bSigma)$ rows with $\bSigma = (0.9^{|i-j|})_{1 \leq i , j , \leq 4}$, and then set a response mean as $\bmu_n \equiv  \sqrt{n}(\xvec_1 + \avec_3)$, where $\avec_3$ is the equiangular vector to the first three columns $[\xvec_1~\xvec_2~\xvec_3]$ of $\Xmat$. Setting $\bmu = n^{-1/2}\bmu_n$ yields a LAR population path $\LAR(\Xmat,\bmu)$ with the intended mid-path tie. 
From this single construction of $\bmu_n$, we generated \numprint{20000} realizations of $\yvec_n = \bmu_n + \berror_n$, where $\berror_n \sim \calN(\zerovec,\Imat_n)$. Figure \ref{fig:asymp_dist_ties_Ck} displays  the approximated marginal distributions of each sample path step correlation $\hat C_1$, $\hat C_2$, $\hat C_3$, and $\hat C_4$ in $\LAR(\Xmat,\yvec)$ with $\yvec = n^{-1/2}\yvec_n$.      We see that the marginal distribution of $\hat C_3$, in particular, is a mixture distribution, where mixture components arise here specifically due to the competition between variables $\xvec_2$ and $\xvec_3$ to enter the active set on step 2 of the sample path 
whereby  the ``loser''   enters the sample path on step~3.  
However, the  distribution of $\hat C_3$ indicates a problem   here.  That is, the sampling  
distribution for $\hat C_3$  turns out to never completely concentrate at either population counterpart $C_2\approx 0.9$ or $C_3=0$ in increasing sample sizes, which indicates that $\hat C_3$ is {\it not} consistent  for any $C_k$ value.  Hence, mid-path ties in the LAR population path (i.e., where variables $\xvec_2,\xvec_3$ become active on a single step) can cause the LAR data algorithm (where the same variables can never enter  on a single step) to fail to match LAR counterparts at the population level.

\begin{figure}[ht]
\includegraphics[width = \textwidth]{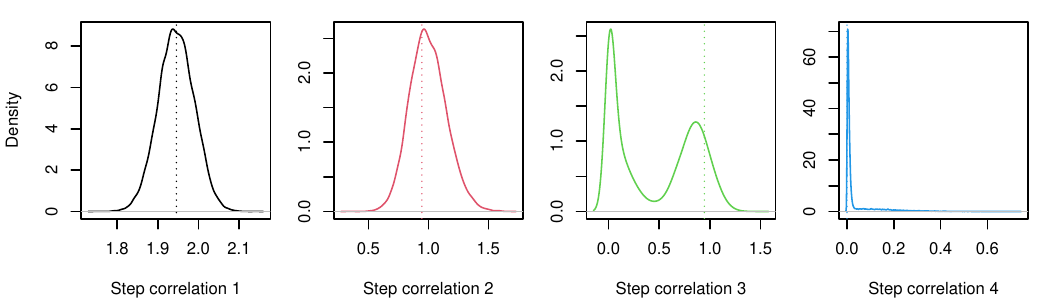}
% \caption{Monte Carlo approximations to the densities of $\hat C_1$, $\hat C_2$, $\hat C_3$, $\hat C_4$ based on \numprint{20000} realizations in a setting with $n = 500$  and $p = 4$ under a construction of $\Xmat$ and $\bmu$ such that in the LAR population path $\xvec_1$ enters on step 1, $\xvec_2$ and $\xvec_3$ enter simultaneously on step 2, and $\xvec_4$ does not enter.  The values of the population step correlations $C_1>C_2>0=C_3=C_4$ are indicated by dotted vertical lines. The distribution of $\hat C_3$ in the third panel becomes bimodal due to the fact that the sample path cannot admit $\xvec_2$ and $\xvec_3$ simultaneously; the two modes correspond to the orders in which $\xvec_2$ and $\xvec_3$ entered the active set on steps 2 and 3 of the sample path.}
\caption{Distributions of sample LAR step correlations in a case in which there is a mid-path tie in the population path.}
\label{fig:asymp_dist_ties_Ck}
\end{figure}

While not further investigating    mid-path ties  here, our point again is that, outside of the prototypical LAR population path (i.e., with no mid-path ties),   the original LAR  data algorithm   may   require modification  to ensure that  sample output is even  meaningful.  From our understanding, this aspect is not apparent in the original LAR work of  \cite{efron2004least},  which  largely focuses on a conceptualization of LAR as admitting  one variable at a time.  Remark~\ref{rem:3} of Section~\ref{sec:understanding} also speaks to the latter point, where   a prototypical LAR population path 
is  connected to a complexity notion 
from
\cite{efron2004least}  about a LAR sample path. 

We lastly remark that, in the setup of Theorem \ref{thm:larconsistency}, it is only the assumption of a prototypical LAR path (Definition~\ref{def:1}) that is violated by mid-path ties; note that the conditions (M1) and (M2) do not place any restriction on the number of variables allowed to enter on each step of the LAR population path. However, even with a prototypical population path satisfying (M2) by a very \textit{small} margin $\delta > 0$,  a sample LAR path may still exhibit, under small sample sizes, behavior expected in a mid-path population tie scenario because small values $\delta > 0$ entail that variables have close or nearly tied entrances at the population level, and sample data would have difficulty resolving such near ties.

\section{Discussion}
\label{sec:discussion}
We have presented  new explanations and results toward better understanding and use of the output of the least angle regression (LAR) algorithm proposed by    
 \cite{efron2004least}.  Equi-angular vectors, step correlations, and variable orderings from a LAR path 
are shown to admit helpful decompositions for algebraically interpreting predictors.  In fact, variable entrance into LAR is seen to be determined by a criterion involving a type of penalized sum of squares (Lemma 9), which provides a useful alternative viewpoint of LAR against the original algorithmic statement in   \cite{efron2004least}.
Additionally, we have introduced a notion that the output of LAR from a data sample can be viewed as estimating a larger population truth given by  
 the population-level LAR path (i.e., how the LAR algorithm would  prescribe or organize the   true response mean).    Because an underlying response mean is built from  step correlations and variable orderings as the quantities defining a LAR  population path, 
the output of LAR from a sample is seen as an estimation procedure that attempts to uncover this population structure.   
Our results establish that a LAR  sample  path is consistent in the sense that, with sufficient data, the algorithm will admit variables into the active set in the correct order, with the correct signs, and up to the right number of steps as the population version; moreover,  sample  step correlations, when properly scaled, converge jointly to a well-defined limiting distribution that admits estimation via a bootstrap procedure for uncertainty quantification.  Our results also suggest a new and formal rule for deciding from data when the LAR algorithm should be stopped.  
An end-product of our work is an inferred LAR path, which consists of interval estimates for the population step correlations as well as for the step regressor coefficients, as indexed by      active sets in the sample path up to the number of steps estimated to exist in the population path.  Under our conditions, we recover the true active sets and the true number of steps in the population path with probability tending to one, so inference on these population path quantities do not require post-selection-type adjustments.

This work suggests that several further studies of LAR are possible.  We have focused on the $p < n$ case, though we believe similar results may be obtained in a $p>n$ setting in which $p$ is allowed to grow with $n$.  Another  point of consideration, as suggested by a reviewer, is the extent to which LAR consistency results may still hold with diminishing gaps in population correlations over increasing sample sizes (e.g.,  assuming conditions (M1) and (M2) of Theorem~14 hold for sequences $\tau \equiv \tau_n \downarrow 0 $ and $\delta \equiv \delta_n \downarrow 0$); numerical studies here suggest that reasonable inference for the LAR population path may still be possible.
The case of mid-path population ties, where the LAR population path might admit more than one variable on a single step, is also important for further investigation in order to improve use of the LAR sample path.   
We have illustrated numerically that the original LAR formulation is perhaps not intended for such cases, which may motivate  a modification to the LAR sample algorithm 
in order to force this to admit variables simultaneously     when variables appear nearly tied to enter on a given sample step.
In other words, one may need to somehow force a LAR sample path to admit appropriate variables on a single step in order to match the behavior of a LAR population path with variable ties.  
Connections and comparisons to Lasso 
may also warrant further exploration.   For example, it is presently unclear how the assumptions needed for consistency of LAR   compare to those required of Lasso, but these may differ so that    LAR/Lasso could possibly have different operating conditions.       Further,  the developments here for LAR raise questions about possible 
alternative understandings of  the Lasso solution path. The Lasso modification to the LAR algorithm, however, can change the number of steps the algorithm takes to finish, as a variable which has entered the active set may afterwards become inactive. As a result, we anticipate that summaries of an inferred Lasso-modified LAR path would be more complicated than those we present in Tables \ref{tab:facetemp} and \ref{tab:diabetes}, in which each variable enters the active set once and only once.   We leave these investigations for the future.

% An R package `larinf', which we make available at \url{https://github.com/gregorkb/larinf}, contains complete code for implementing our methods.

% Acknowledgements and Disclosure of Funding should go at the end, before appendices and references

\acks{
The authors wish to sincerely thank  a reviewer for many thoughtful comments and suggestions that greatly improved the manuscript. The authors gratefully acknowledge the computational resources provided by the Hyperion high performance computing cluster at the University of South Carolina. We also acknowledge the technical assistance and resources provided by Research Computing at the University of South Carolina (RRID:SCR\_027488). Research  was partially supported by   NSF-DMS 2515719.}

 %\newpage

\newpage
\appendix

\section{Proofs of results from Sections \ref{sec:lar} and \ref{sec:understanding} }
\label{supp:A} 

\begin{proof}{ \bf of Lemma \ref{lem:gammak}:} For simplicity, we establish the analog result considering the LAR population path in place of the sample path.  In which case, it suffices to show that $\gamma_k$ as defined in \eqref{eqn:gammak} is the maximum value of $\gamma \in (0, C_k/ A_k]$ such that
\begin{equation}
\label{eqn:gammasetup}
|\xvec_j^T(\bmu - (\bmu_{k-1} +  \gamma \avec_k))| \leq |\xvec_l^T(\bmu - (\bmu_{k-1} + \gamma \avec_k))|
\end{equation}
for all $j\notin \calA_k$ and all $l \in \calA_k$; by showing a solution exists for  $\gamma \in (0, C_k/ A_k]$, then there is no need to consider larger $\gamma> C_k/A_k$. For each $j \notin \calA_k$ the left side of \eqref{eqn:gammasetup} can be written
\[
    |\xvec_j^T(\bmu - (\bmu_{k-1} + \gamma \avec_k))| = |(\cvec_k)_j - \gamma (\wvec_k)_j|
\]
and for any $l \in \calA_k$, the right side of \eqref{eqn:gammasetup} can be written
\[
|\xvec_l^T(\bmu - (\bmu_{k-1} + \gamma \avec_k))| = |\sign((\cvec_k)_l)(C_k - \gamma A_k)| = C_k - \gamma A_k,
\]
as we consider  $\gamma \leq C_k/A_k$. Now re-write  (\ref{eqn:gammasetup}) as 
\begin{equation} \label{eqn:gammaineq}
-C_k + \gamma A_k \leq (\cvec_k)_j - \gamma (\wvec_k)_j\leq C_k - \gamma A_k.
\end{equation}
The maximum value of $\gamma \leq C_k / A_k$ satisfying the above expression is found by solving either
\begin{equation}\label{eqn:gamma1}
-C_k + \gamma A_k = (\cvec_k)_j - \gamma (\wvec_k)_j,
\end{equation}
from the first inequality if  $(\cvec_k)_j - (C_k/A_k)(\wvec_k)_j < 0$ holds or by solving
\begin{equation}\label{eqn:gamma2}
(\cvec_k)_j - \gamma (\wvec_k)_j =  C_k - \gamma  A_k
\end{equation}
from the second inequality if $(\cvec_k)_j - (C_k/A_k)(\wvec_k)_j > 0$ holds.   This is depicted in Figure~\ref{fig:gammaplot}, in which the solid lines trace the values $C_k - \gamma A_k$ and $- C_k + \gamma A_k$ over $\gamma \in (0,C_k/A_k]$ and the dashed lines trace examples of $(\cvec_k)_j - \gamma (\wvec_k)_j$ over $\gamma \in(0,C_k/A_k]$ in the two cases that either $(\cvec_k)_j - (C_k/A_k)(\wvec_k)_j > 0$ or $(\cvec_k)_j - (C_k/A_k)(\wvec_k)_j < 0$;  whether one should solve   \eqref{eqn:gamma1} or   \eqref{eqn:gamma2} here depends on whether $(\cvec_k)_j - \gamma (\wvec_k)_j$ intersects with $C_k - \gamma A_k$ or with $- C_k + \gamma A_k$ (i.e., exactly one of these cases must occur if
$(\cvec_k)_j - (C_k/A_k)(\wvec_k)_j  \neq  0$ using that 
$|(\cvec_k)_j| < C_k$ for any 
$j \notin \calA_k$), which corresponds to the cases $(\cvec_k)_j - (C_k/A_k)(\wvec_k)_j > 0$ and $(\cvec_k)_j - (C_k/A_k)(\wvec_k)_j < 0$, respectively.  Note that, as we consider only $j \notin \calA_k$, whereby $|(\cvec_k)_j| < C_k$, it follows in either case that the solution will satisfy $\gamma >0$.  In a final possible case that $(\cvec_k)_j - (C_k/A_k)(\wvec_k)_j  = 0$, we obtain $\gamma = C_k/A_k>0$  (i.e., which is positive by $C_k,A_k > 0$). These three cases lead to the expression for $\gamma_{k,j}>0$ in \eqref{eqn:gammak}. Setting $\gamma_k = \min_{j \notin \calA_k} \{\gamma_{k,j}\}$ then results in the largest value of $\gamma$ satisfying \eqref{eqn:gammasetup}.
\end{proof}

\begin{figure}
\begin{center}
\begin{tikzpicture}[scale=0.7]
  \draw[-,dotted] (-.1,0) -- (8,0); \node[below] at (8,0) {$\gamma$};\node[left] at (-.1,0) {$0$};\node[below] at (0,-.1) {$0$};
  \draw[-,dotted] (0,-4) -- (0,4); \node[below] at (7,-.1) {$C_k / A_k$};
  \draw[-,dotted] (7,-4) -- (7,4);
  \draw[-] (0,3) -- (7,0); \node[left] at (0,3) {$C_k$};
  \draw[-] (0,-3) -- (7,0);\node[left] at (0,-3) {$ - C_k$};
  \draw[-,dashed] (0,1) -- (7,1.5); \node[left] at (0,1) {$(\cvec_k)_j$}; \node[right] at (7,1.5) {$(\cvec_k)_j - \dfrac{C_k}{A_k}(\wvec_k)_j > 0$};
    \draw[-,dashed] (0,1) -- (7,-3); \node[right] at (7,-3) {$(\cvec_k)_j - \dfrac{C_k}{A_k}(\wvec_k)_j < 0$};
  \fill[gray, opacity = 0.1] (0,-3) -- (7,0) --  (0,3) --  (0,-3);
\end{tikzpicture}
\caption{Depiction of finding the largest $\gamma \in(0, C_k / A_k]$ which satisfies \eqref{eqn:gammaineq}. The solid lines trace the values $C_k - \gamma A_k$ and $- C_k + \gamma A_k$ and the dashed lines show examples of $(\cvec_k)_j - \gamma (\wvec_k)_j$ in the two cases $(\cvec_k)_j - (C_k/A_k)(\wvec_k)_j > 0$ and $(\cvec_k)_j - (C_k/A_k)(\wvec_k)_j < 0$.}
\label{fig:gammaplot}
\end{center}
\end{figure}
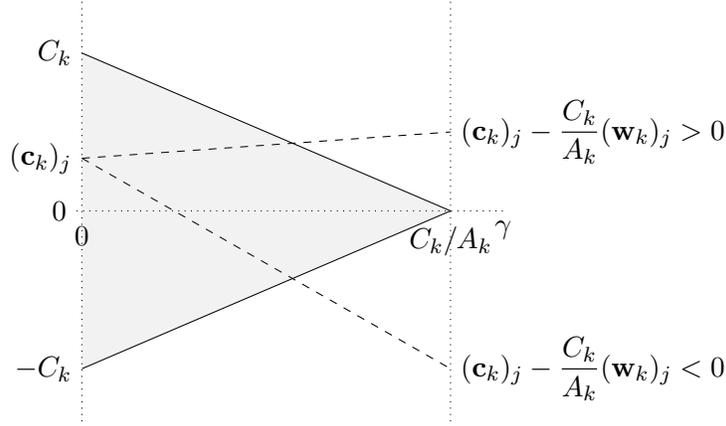

\begin{proof}{\bf of Proposition \ref{prop:LARknown}:} Claim (ii) follows from the definitions of $\cvec_k$, $C_k$, $\calA_k$, and $\Xmat_k$ given in Steps 1--4 of Algorithm \ref{alg:lar}, from which we have
\[
\Xmat_k^T(\bmu - \bmu_{k-1}) = [\sign((\cvec_k)_j) \xvec_j^T(\bmu - \bmu_{k-1}), ~j \in \calA_k]  
 = [\sign((\cvec_k)_j) (\cvec_k)_j, ~j \in \calA_k]  
 = C_k \onevec.\]
We next prove (iii). Suppose $j$ enters the active set on step $k+1$. Then $C_{k+1} = |(\cvec_{k+1})_j|$ and $\gamma_k = \gamma_{k,j}$, where $\gamma_{k,j}$ is defined in \eqref{eqn:gammak}, and we may write
\begin{align*}
    (\cvec_{k+1})_j  = \xvec_j^T(\bmu - \bmu_k)  
    = \xvec_j^T(\bmu - (\bmu_{k-1} + \gamma_{k,j} \avec_k))  &= (\cvec_k)_j - \gamma_{k,j}(\wvec_k)_j\\
                         &= \left\{\begin{array}{ll}
                             \dfrac{A_k(\cvec_k)_j- (\wvec_k)_jC_k}{A_k - (\wvec_k)_j},& r_{k,j} = 1\\
                             &\\
                              \dfrac{A_k(\cvec_k)_j - (\wvec_k)_jC_k}{A_k + (\wvec_k)_j},& r_{k,j} = - 1 \\
                              &\\
                              0, & r_{k,j} =0,
                         \end{array}\right.
\end{align*}
following \eqref{eqn:gammak}, where $r_{k,j} = \sign((\cvec_k)_j - (C_k/A_k) (\wvec_k)_j)$. We may similarly write
\begin{align*}
C_k - \gamma_k A_k  = C_k - \gamma_{k,j} A_k  = \left\{\begin{array}{ll}
                             \dfrac{A_k(\cvec_k)_j- (\wvec_k)_jC_k}{A_k - (\wvec_k)_j},& r_{k,j} = 1\\
                             &\\
                              -\dfrac{A_k(\cvec_k)_j - (\wvec_k)_jC_k}{A_k + (\wvec_k)_j},& r_{k,j} = - 1 \\
                              &\\
                              0, & r_{k,j} =0.
                         \end{array}\right.\end{align*}
In the case $r_{k,j} = 0$, we have $(\cvec_{k+1})_j = C_k - \gamma_k A_k = 0$. If $r_{k,j} \neq 0$, we will have $C_k - \gamma_{k,j} A_k > 0$ (see the proof of Lemma \ref{lem:gammak}), so that in the case $r_{k,j} = 1$ we have
\[
(\cvec_{k+1})_j = \dfrac{A_k(\cvec_k)_j- (\wvec_k)_jC_k}{A_k - (\wvec_k)_j} > 0
\]
and in the case $r_{k,j} = -1$ we have
\[
(\cvec_{k+1})_j = \dfrac{A_k(\cvec_k)_j - (\wvec_k)_jC_k}{A_k + (\wvec_k)_j} < 0.
\]
From here we see that in all cases we have $|(\cvec_{k+1})_j| = C_k - \gamma_k A_k$, which proves (iii). In addition, the above gives
\begin{equation}
\label{eqn:signdetermine}
\sign((\cvec_{k+1})_j) = r_{k,j} \quad \text{ if $j$ enters on step $k+1$}.
\end{equation}

We now prove (i).  Suppose $j \in \calA_k$. Then 
\begin{align*}
    (\cvec_{k+1})_j  = \xvec_j^T(\bmu - \bmu_k)  = \xvec_j^T(\bmu - (\bmu_{k-1} + \gamma_k \avec_k)) 
                    &= (\cvec_k)_j - \gamma_k\sign( (\cvec_k)_j) A_k\\
                    &= \sign((\cvec_k)_j)(C_k - \gamma_kA_k) = \sign((\cvec_k)_j) C_{k+1},
\end{align*}    
where the third equality uses $\sign((\cvec_k)_j)\xvec_j^T\avec_k = A_k$ for $j \in \calA_k$ and the last equality uses (iii). This gives $|(\cvec_{k+1})_j| = C_{k+1}$. It follows that $j \in \calA_{k+1}$, proving (i). Note that the above also shows 
\begin{equation}
\label{eqn:signsdontchange}
\sign((\cvec_{k+1})_j ) = \sign((\cvec_k)_j) \quad \text{ for } j \in \calA_k,
\end{equation}
provided $C_{k+1}>0$; if $C_{k+1} = 0$ then $\sign((\cvec_k)_j) = 0$. 
Lastly we prove (iv). For $j \in \calA_k$,  we have 
\begin{align*}
    \xvec_j^T(\bmu - (\bmu_{k-1} + (C_k/A_k)\avec_k)) &= (\cvec_k)_j - (C_k/A_k)\xvec_j^T\avec_k  
     = (\cvec_k)_j - (C_k/A_k)\sign((\cvec_k)_j) A_k 
    = 0,
\end{align*}
since $|(\cvec_k)_j| = C_k$ for each $j \in \calA_k$. Therefore, the vector $\bmu - (\bmu_{k-1} + (C_k/A_k)\avec_k)$ lies in the orthogonal complement to the column space of $\Xmat_k$. On the other hand, $\bmu_{k-1} + (C_k/A_k)\avec_k$ lies in the column space of $\Xmat_k$, so the decomposition
\[
\bmu = (\bmu_{k-1} + (C_k/A_k)\avec_k) + (\bmu - (\bmu_{k-1} + (C_k/A_k)\avec_k)),
\]
is an orthogonal decomposition of $\bmu$ in which the first term on the right side is the orthogonal projection of $\bmu$ onto the column space of $\Xmat_k$.

\end{proof}

\begin{proof}{\bf of Lemma \ref{lem:aAupdate}:}
Claim (i) follows from \eqref{eqn:signsdontchange} in the proof of Proposition \ref{prop:LARknown}. Claims (ii) and (iii) hold for $k = 1$ due to the fact that if $j_1$ enters on step $1$ with sign $s_1$ we will have $A_1 = 1$, $\avec_1 = s_1\xvec_{j_1}$, $\evec_1 = \xvec_{j_1}$, and $u_1 = 1$.  We next prove claims (ii) and (iii) for $k > 1$. Suppose $j_k$ enters the active set on step $k$ with sign $s_k$. By a block matrix inversion formula and by the definitions of $\avec_{k-1}$, $A_{k-1}$, and $\evec_k$, we may write
\begin{align*}
    A_k^{-2} &= \onevec^T \left(\begin{array}{ll}
         \Xmat_{k-1}^T\Xmat_{k-1}& s_k \Xmat_{k-1}^T \xvec_{j_k}  \\
         s_k \xvec_{j_k}^T\Xmat_{k-1}& \xvec_{j_k}^T\xvec_{j_k} 
    \end{array}\right)^{-1}\onevec\\
    &= \onevec^T(\Xmat_{k-1}^T\Xmat_{k-1})^{-1}\onevec + (1 - s_k \xvec_{j_k}^T\Xmat_{k-1}(\Xmat_{k-1}^T\Xmat_{k-1})^{-1}\onevec)^2(\xvec_{j_k}^T(\Imat - \Pmat_{k-1})\xvec_{j_k})^{-1}\\
    &= A_{k-1}^{-2} + (1 - s_k \xvec_{j_k}^T\avec_{k-1}/A_{k-1})^2(\evec_k^T\evec_k)^{-1},
\end{align*}
which proves claim (iii).  To prove claim (ii), one may similarly use a block matrix inversion formula to show
\[
\Pmat_{k-1} \frac{1}{A_k}\avec_k = \frac{1}{A_{k-1}}\avec_{k-1} 
\]
as well as
\[
(\Imat - \Pmat_{k-1})\frac{1}{A_k}\avec_k= s_k u_k\evec_k,
\]
from which one can write the orthogonal decomposition
\[
\frac{1}{A_k}\avec_k  = \frac{1}{A_{k-1}}\avec_{k-1} + s_k u_k\evec_k.
\]  
To show $u_k > 0$, we begin with the fact that $\gamma_{k-1} = \gamma_{k-1,j_k} > 0$ (see the proof of Lemma \ref{lem:gammak}).  We may write
\[
\gamma_{k-1,j_k} = \frac{C_{k-1} - r_{k-1,j_k} (\cvec_{k-1})_{j_k}}{A_{k-1} - r_{k-1,j_k} \xvec_{j_k}^T\avec_{k-1}} = \frac{C_{k-1} - s_k (\cvec_{k-1})_{j_k}}{\evec_k^T\evec_kA_{k-1}u_k} > 0,
\]
making use of $s_k = r_{k-1,j_k}$ from \eqref{eqn:signdetermine}.   Now, the numerator $C_{k-1} - s_k (\cvec_{k-1})_{j_k}$ is positive by the fact that $j_k \notin \calA_{k-1}$, giving $u_k >0$.
\end{proof}

\begin{proof}{\bf of Lemma \ref{lem:Ck}:} First consider step $k=1$. Assume $j_1$ enters on step $1$ with sign $s_1$. Then $\avec_1 = s_1\xvec_{j_1}$, $A_1 = 1$, and $\evec_k = \xvec_{j_1}$, so we have
\[
\Big(\frac{1}{A_1^2} - \frac{1}{A_0^2}\Big)^{-1}\Big(\frac{1}{A_1}\avec_1 - \frac{1}{A_0}\avec_0\Big)^T\bmu = s_1\xvec_{j_1}^T\bmu = C_1,
\]
since $A_0 = \infty$ and $\avec_0 = \zerovec$. Moreover
\[
\frac{s_1 \evec_1^T \bmu}{ (1 - s_1 \xvec_{j_1}^T\avec_0 / A_0)} = s_1 \xvec_{j_1}^T\bmu = C_1.
\]
For $k >1$, suppose $j_k$ enters on step $k$ with sign $s_k$ and multiply both sides of $\Xmat_k^T(\bmu - \bmu_{k-1}) = C_k\onevec$, i.e. Proposition \ref{prop:LARknown}(ii), by $\onevec^T(\Xmat_k^T\Xmat_k)^{-1}$. This gives
\[
\onevec^T(\Xmat_k^T\Xmat_k)^{-1}\Xmat_k^T(\bmu - \bmu_{k-1}) = C_k\onevec^T(\Xmat_k^T\Xmat_k)^{-1}\onevec.
\]
From the definitions of $\avec_k$ and $A_k$, we may rewrite the previous equation as
\[
\frac{1}{A_k}\avec_k^T(\bmu - \bmu_{k-1}) = C_k\frac{1}{A_k^2},
\]
the left hand side of which can be rewritten as
\[
\frac{1}{A_{k-1}}\avec_{k-1}^T(\bmu - \bmu_{k-1}) + \left(\frac{1}{A_k}\avec_k - \frac{1}{A_{k-1}}\avec_{k-1}\right)^T(\bmu - \bmu_{k-1}).
\]
The first term of the above can be written as
\[
\frac{1}{A_{k-1}}\avec_{k-1}^T(\bmu - \bmu_{k-1}) = \onevec^T(\Xmat_{k-1}^T\Xmat_{k-1})^{-1}\Xmat_{k-1}^T(\bmu - \bmu_{k-1}) = \frac{1}{A_{k-1}^2}C_k,
\]
so we have
\[
\frac{1}{A_{k-1}^2}C_k + \left(\frac{1}{A_k}\avec_k - \frac{1}{A_{k-1}}\avec_{k-1}\right)^T(\bmu - \bmu_{k-1}) = \frac{1}{A_k^2}C_k.
\]
From here we may write
\[
    C_k =  \Big(\frac{1}{A_k^2} - \frac{1}{A_{k-1}^2}\Big)^{-1}\Big(\frac{1}{A_k}\avec_k - \frac{1}{A_{k-1}}\avec_{k-1}\Big)^T(\bmu - \bmu_{k-1}).
\]
Now Lemma \ref{lem:aAupdate}(ii) gives
\[
\Big(\frac{1}{A_k}\avec_k - \frac{1}{A_{k-1}}\avec_{k-1}\Big)^T\bmu_{k-1} = u_ks_k \evec_k^T\bmu_{k-1}= 0,
\]
since $\bmu_{k-1}$ is in the column space of $\Xmat_{k-1}$, to which the vector $\evec_k$ is orthogonal. This gives the first equality in \eqref{eqn:Ckangles}; the second equality as well as the positivity of $1 - s_k \xvec_{j_k}^T\avec_{k-1}/A_{k-1}$ follow from Lemma \ref{lem:aAupdate}(ii)-(iii). Lastly, we may use \eqref{eqn:Ckangles} and Lemma \ref{lem:aAupdate}(ii) to write
\[
C_k\Big(\frac{1}{A_k}\avec_k - \frac{1}{A_{k-1}}\avec_{k-1}\Big) = \frac{s_k \evec_k^T\bmu}{(1 - s_k \xvec_{j_k}^T\avec_{k-1}/A_{k-1})} u_k s_k \evec_k= \evec_k(\evec_k^T\evec_k)^{-1}\evec_k^T \bmu,
\]
which establishes \eqref{eqn:Pek}.
\end{proof}

\begin{proof}{\bf of Lemma \ref{lem:muk}:} Proposition \ref{prop:LARknown}(iii) gives $\gamma_k = (C_k - C_{k+1})/A_k$ for each $k=1,\dots,m$, by which we may write
\begin{align*}
    \bmu_k  = \sum_{j=1}^k \gamma_j \avec_j  = \sum_{j=1}^k (C_j - C_{j+1})\frac{1}{A_j} \avec_j  = \sum_{j=1}^kC_j\Big(\frac{1}{A_j} \avec_j - \frac{1}{A_{j-1}}\avec_{j-1}\Big) - C_{k+1}\frac{1}{A_k}\avec_k.
\end{align*}
From here, the first equality in \eqref{eqn:sumproj} follows from 
\[
\frac{1}{A_k}\avec_k = \sum_{j=1}^k\Big(\frac{1}{A_j} \avec_j - \frac{1}{A_{j-1}}\avec_{j-1}\Big)
\]
and the second from 
\[
\Pmat_k \bmu = \sum_{j=1}^k\Pmat_{\evec_k}\bmu= \sum_{j=1}^kC_j\Big(\frac{1}{A_j} \avec_j - \frac{1}{A_{j-1}}\avec_{j-1}\Big),
\]
which we see from \eqref{eqn:Pek} of Lemma \ref{lem:Ck}. Now, $C_{m+1} = 0$ implies $\Xmat^T(\bmu - \bmu_m) = \zerovec$, where \eqref{eqn:sumproj} gives $\bmu_m = \Pmat_m \bmu$.  So we may write
\[
\Xmat^T(\bmu - \bmu_m) = \Xmat^T(\bmu - \Pmat_m\bmu) = \Xmat^T(\Imat - \Pmat_m)\bmu = \zerovec.
\]
Letting $\Pmat_\Xmat$ denote the orthogonal projection onto the column space of $\Xmat$, we may from the above write 
\[
\Pmat_\Xmat(\Imat - \Pmat_m)\bmu = \zerovec \iff \Pmat_\Xmat\bmu = \Pmat_m \bmu,
\]
which completes the proof.
\end{proof}

\begin{proof}{\bf of Lemma \ref{lem:whichtoenter}:}
To prove (i), choose any $j \notin \calA_k$ and write
\begin{align*}
    (\cvec_k)_j - \frac{C_k}{A_k}(\wvec_k)_j  =  \xvec_j^T(\bmu - \bmu_{k-1}) - \frac{C_k}{A_k}(\wvec_k)_j   
                    &= \xvec_j^T(\bmu - \Pmat_k \bmu + \Pmat_k \bmu - \bmu_{k-1}) - \frac{C_k}{A_k}(\wvec_k)_j\\
                    &= \xvec_j^T(\Imat - \Pmat_k)\bmu + \frac{C_k}{A_k} \xvec_j^T\avec_k  -     \frac{C_k}{A_k}(\wvec_k)_j \\
                    &= \xvec_j^T(\Imat - \Pmat_k)\bmu, 
\end{align*}            
using $\Pmat_k \bmu - \bmu_{k-1} = (C_k/A_k)\avec_k$ from \eqref{eqn:Pkmu} and $(\wvec_k)_j = \xvec_j^T\avec_k$. We have $A_k > (\wvec_k)r_{k,j}$ from the fact that $\gamma_{k,j} = (C_k - (\cvec_k)_j r_{k,j})/(A_k - (\wvec_k)r_{k,j}) > 0$ for all $j \notin \calA_k$, where the numerator is always positive. 
To prove (ii), note that $j_k$ will enter on step $k$ if it minimizes $\gamma_{k-1,j}$ over all $j \notin \calA_{k-1}$. Equivalently, $j_k$ will enter if it maximizes $C_{k-1} - \gamma_{k-1,j}A_{k-1}$ over all $j \notin \calA_{k-1}$.  We have
\begin{align*}
    C_{k-1} - \gamma_{k-1,j} A_{k-1} &= C_{k-1} - \Big(\frac{C_{k-1} - (\cvec_{k-1})_j r_{k-1,j}}{A_{k-1} - (\wvec_{k-1})_j r_{k-1,j}}\Big)A_{k-1}\\
    &=\frac{A_{k-1}(\cvec_{k-1})_j r_{k-1,j} - C_{k-1}(\wvec_{k-1})_j r_{k-1,j}}{A_{k-1} - (\wvec_{k-1})_j r_{k-1,j}}\\
    &=\frac{r_{k-1,j}((\cvec_{k-1})_j) - (C_{k-1}/A_{k-1})(\wvec_k)_j)}{1 - r_{k-1,j}\xvec_j^T\avec_{k-1}/A_{k-1}} \\
    &=\frac{|\xvec_j^T(\Imat - \Pmat_{k-1})\bmu|}{1 - r_{k-1,j}\xvec_j^T\avec_{k-1}/A_{k-1}} 
    = C_{k,j}.
\end{align*}
The denominator of $C_{k,j}$ is positive since $\gamma_{k-1,j} > 0$ for all $j \notin \calA_{k-1}$. 
Claim (iii) follows from Proposition \ref{prop:LARknown}(iii) and from \eqref{eqn:signdetermine}.
\end{proof}

\section{Proofs of Section 4 results}
\label{sec:proofsasymp}

We will need the following result to prove Theorem \ref{thm:larconsistency}:

\begin{lemma}
\label{lem:denombound}
At each step $k$ of $\LAR(\Xmat,\bmu)$, we have
\begin{equation}
\label{eqn:denomboundfinite}
1 - r_{k,j}\xvec_j^T\avec_k/A_k \geq h_{k,j},
\end{equation}
where
\begin{equation}
\label{eqn:hkjdef}
h_{k,j} = \left\{\begin{array}{ll}
     1 - \dfrac{|(\cvec_k)_j|}{C_k}& \text{ if }  \dfrac{|\xvec_j^T\avec_k|}{A_k} < \dfrac{|(\cvec_k)_j|}{C_k} \text{ and } \sign(\xvec_j^T\avec_k) = \sign((\cvec_k)_j)\\
     1& \text{ otherwise} 
\end{array}\right.
\end{equation}
for all $j \notin \calA_k$.
\end{lemma}

\begin{proof}{\bf of Lemma \ref{lem:denombound}:}
Beginning from \eqref{eqn:gammak}, we may write
\[
r_{k,j} = \sign\Big((\cvec_k)_j - \frac{C_k}{A_k}(\wvec_k)_j \Big)=\sign\Big(\frac{(\cvec_k)_j}{C_k} - \frac{\xvec_j^T\avec_k}{A_k} \Big),
\] 
since $(\wvec_k)_j = \xvec_j^T\avec_k$ and $C_k > 0$. Now we consider several cases:  First, suppose $|\xvec_j^T\avec_k|/A_k > |(\cvec_k)_j|/C_k$. In this case $r_{k,j} =-\sign(\xvec_j^T\avec_k)$ so that
\begin{equation}
\label{eqn:lowerbound1}
1 - r_{k,j}\xvec_j^T\avec_k/A_k  = 1 + |\xvec_j^T\avec_k|/A_k  \geq 1.
\end{equation}
Now suppose $\xvec_j^T\avec_k/A_k  = (\cvec_k)_j/C_k$. In this case $r_{k,j} = 0$ so that $1 - r_{k,j}\xvec_j^T\avec_k/A_k = 1$. Now suppose $\xvec_j^T\avec_k/A_k  = -(\cvec_k)_j/C_k \neq 0$. Then $r_{k,j} = \sign((\cvec_k)_j) = -\sign(\xvec_j^T\avec_k)$, giving again \eqref{eqn:lowerbound1}. Now suppose $|\xvec_j^T\avec_k|/A_k < |(\cvec_k)_j|/C_k$. In this case $r_{k,j} = \sign((\cvec_k)_j)$. Here, if $\sign(\xvec_j^T\avec_k) = -\sign((\cvec_k)_j)$, we will have again \eqref{eqn:lowerbound1}; on the other hand, if $\sign(\xvec_j^T\avec_k) = \sign((\cvec_k)_j)$, we have
\[
1 - r_{k,j}\xvec_j^T\avec_k/A_k = 1 - |\xvec_j^T\avec_k|/A_k > 1 - |(\cvec_k)_j| / C_k.
\]

\end{proof}

\begin{proof}{\bf of Theorem \ref{thm:larconsistency}:} As a preliminary, we note that the normalization of the columns of the design matrix $\Xmat_n$ and the definition of the rescaled response $\yvec \equiv n^{-1/2}\yvec_n$ and its population counterpart $\bmu \equiv n^{-1/2}\Xmat_n \bbeta_n$ in (R4) give $\yvec = \bmu + n^{-1/2}\berror_n$. Therefore on step $k=1$, with $\hat \cvec_1 = \Xmat^T\yvec$ and $\cvec_1 = \Xmat^T\bmu$, we have $\hat \cvec_1 - \cvec_1 =n^{-1/2}\Xmat^T\berror_n$. Our first task will thus be to establish
\begin{equation}
\label{eqn:noisetozero}
\|n^{-1/2}\Xmat \berror_n\| \inprob 0
\end{equation}
as $n \to \infty$.  To establish \eqref{eqn:noisetozero}, let $\Dmat = \diag(\|\xvec_{n1}\|^{-1},\dots,\|\xvec_{np}\|^{-1})$ such that $\Xmat = \Xmat_n \Dmat_n$ as in (R4). Then $n^{-1/2}\Xmat \berror_n = n^{-1/2}n^{1/2}\Dmat_n n^{-1/2}\Xmat_n^T\berror_n$, where by (R2) $n^{1/2}\Dmat_n \to \Dmat= \diag(\sigma_{11}^{-1},\dots,\sigma_{pp}^{-1})$, where $\sigma_{11}^2,\dots,\sigma_{pp}^2$, denote the diagonal entries of $\bSigma$, and $n^{-1/2}\Xmat_n^T\berror_n$ converges to a multivariate normal random vector by (R3).  Thus \eqref{eqn:noisetozero} follows, giving $\|\hat \cvec_1 - \cvec_1\| \inprob 0$ as $n \to \infty$.

Now, due to (M2), the event
\begin{equation}
\label{eqn:c1event}
    \max_{1 \leq j \leq p} |(\hat \cvec_1)_j - (\cvec_1)_j|  < \frac{\delta}{2}
\end{equation}
ensures $|(\hat \cvec_1)_{j_1}| > \max_{j \neq j_1} |(\hat \cvec_1)_j|$, which results in $\hat \calA_1 = \{j_1\}$ and, consequently, $\hat s_1 = \sign((\hat \cvec_1)_{j_1})$ and $\hat C_1 = \hat s_1(\hat \cvec_1)_{j_1}$.  Moreover, since $C_1 = s_1 (\cvec_1)_{j_1}\geq \delta$ by (M2), the event \eqref{eqn:c1event} ensures $\sign((\hat \cvec_1)_{j_1}) = \sign((\cvec_1)_{j_1})$, so that $\hat s_1 = s_1$.   Thus on the event \eqref{eqn:c1event} we have $\hat C_1 = s_1 (\hat \cvec_1)_{j_1}$ so that $\hat C_1 - C_1 = s_1 (\hat \cvec_1)_{j_1} - s_1(\cvec_1)_{j_1}$.  From here, for any $\epsilon > 0$ we have
\begin{align*}
\PP(|\hat C_1 - C_1| < \epsilon) &\geq \PP(|\hat C_1 - C_1| < \epsilon ~\cap~ \max_{1 \leq j \leq p} |(\hat \cvec_1)_j - (\cvec_1)_j|  < \frac{\delta}{2} ) \\
& = \PP(|s_1 (\hat \cvec_1)_{j_1} - s_1 ( \cvec_1)_{j_1}|  < \epsilon ~\cap~ \max_{1 \leq j \leq p} |(\hat \cvec_1)_j - (\cvec_1)_j|  < \frac{\delta}{2} )\\
& \to 1
\end{align*}
as $n \to \infty$ by the fact that $\max_{1 \leq j \leq p} |(\hat \cvec_1)_j - (\cvec_1)_j|  \leq \|\hat \cvec_1 - \cvec_1\| \inprob 0$ as $n \to \infty$. So (i), (ii), and (iii) are established for $k=1$.

Now, suppose (i), (ii), and (iii) hold on steps $1,\dots,k$, where $1 \leq k \leq m-1$. Then, conditionally on the event
\[
\calL_k = \cap_{i=1}^k\{ \hat \calA_{i}\setminus\hat\calA_{i-1} = \{j_{i}\} \text{ and } \hat s_i = s_i\},
\]
we will show
\begin{enumerate}[label=(\alph*)] \itemsep0cm
    \item $P( \hat \calA_{k+1}\setminus \hat\calA_{k} = \{j_{k+1}\} \text{ and } \hat s_{k+1} = s_{k+1}) \to 1$
    % $P( \hat \calA_{k+1} = \{j_1,\dots,j_{k+1}\} \text{ and } \hat s_{j_{k+1}} = s_{j_{k+1}}) \to 1$,\\
    \item $\hat C_{k+1} - C_{k+1} \inprob 0$, and
    \item $\max_{j \notin \calA_{k+1}} |(\hat \cvec_{k+1})_j - (\cvec_{k+1})_j | \inprob 0$
\end{enumerate}
as $n \to \infty$. This will establish, by induction, that (i), (ii), and (iii) hold for all $k=1,\dots,m$.

As a first step we define $\mu_{k,j} \equiv \xvec_j^T(\Imat - \Pmat_k)\bmu$ and $\varepsilon_{k,j} \equiv n^{-1/2}\xvec_j^T(\Imat - \Pmat_k)\berror_n$ for all $j \notin \calA_k$, which allows us to write $\xvec_j^T(\Imat - \Pmat_k)\yvec = \mu_{k,j} + \varepsilon_{k,j}$ as well as $r_{k,j} = \sign(\mu_{k,j})$ and $\hat r_{k,j} = \sign(\mu_{k,j} + \varepsilon_{k,j})$, according to Lemma \ref{lem:whichtoenter}. Thus
\[
C_{k+1,j} = \frac{r_{k,j}\mu_{k,j}}{1 - r_{k,j}\xvec_j^T\avec_k/A_k} \quad \text{ and } \quad \hat C_{k+1,j} = \frac{\hat r_{k,j}(\mu_{k,j} + \varepsilon_{k,j})}{1 - \hat r_{k,j}\xvec_j^T\avec_k/A_k} 
\]
for all $j \notin \calA_k$. Moreover $\varepsilon_{k,j} \inprob 0$ as $n \to \infty$ for each $j \notin \calA_k$.

Now, due to (M2), the simultaneous occurrence of the event
\begin{equation}
\label{eqn:Ckplus1event}
|\hat C_{k+1,j_{k+1}} - C_{k+1,j_{k+1}}| < \frac{\delta}{2}
\end{equation}
and the event
\begin{equation}
\label{eqn:Ckplus1jevent}
\max_{j \notin \calA_k\cup\{j_{k+1}\}} \Big|\hat C_{k+1,j} - \frac{\hat r_{k,j}\mu_{k,j}}{1 - \hat r_{k,j} \xvec_j^T\avec_k/A_k}\Big| < \frac{\delta}{2}
\end{equation}
ensures $\hat \calA_{k+1} \setminus \hat \calA_k = \{j_{k+1}\}$ and, consequently, $\hat s_{k+1} = \hat r_{k,j_{k+1}}$ and $\hat C_{k+1} = \hat C_{k+1,j_{k+1}}$. To see why, note that, as illustrated in Figure \ref{fig:margincond2}, (M2) gives
\begin{equation}
\label{eqn:M2alternate}
C_{k+1,j_{k+1}} \geq \max_{j \notin \calA_k\cup\{j_{k+1}\}} C_{k+1,j} + \delta.
\end{equation}
Furthermore, we have
\[
C_{k+1,j} = \frac{r_{k,j}\mu_{k,j}}{1 - r_{k,j}\xvec_j^T\avec_k/A_k} \geq \frac{\hat r_{k,j}\mu_{k,j}}{1 - \hat r_{k,j}\xvec_j^T\avec_k/A_k}
\]
for all $j \notin \calA_k\cup\{j_{k+1}\}$, since $1 - r_{k,j}\xvec_j^T\avec_k/A_k$ and $1 - \hat r_{k,j}\xvec_j^T\avec_k/A_k$ are both positive (see Lemma \ref{lem:whichtoenter}).  We may therefore replace \eqref{eqn:M2alternate} with 
\begin{equation}
\label{eqn:M2alternate_alternate}
C_{k+1,j_{k+1}} \geq \max_{j \notin \calA_k\cup\{j_{k+1}\}} \frac{\hat r_{k,j}\mu_{k,j}}{1 - \hat r_{k,j}\xvec_j^T\avec_k/A_k} + \delta.
\end{equation}
Given \eqref{eqn:M2alternate_alternate}, the simultaneous occurrence of events \eqref{eqn:Ckplus1event} and \eqref{eqn:Ckplus1jevent} implies $\hat \calA_{k+1} \setminus \hat \calA_k = \{j_{k+1}\}$.

To establish that the event in \eqref{eqn:Ckplus1event} occurs with high probability, we will make use of a lower bound for $|\mu_{k,j_{k+1}}|$ obtained as follows: First, (M2) gives
\[
C_{k+1,j_{k+1}} = \frac{r_{k,j_{k+1}}\mu_{k,j_{k+1}}}{1 - r_{k,j_{k+1}} \xvec_{j_{k+1}}^T\avec_k/A_k} \geq \delta.
\]
From here, we invoke Lemma \ref{lem:denombound} to write
\begin{equation}
\label{eqn:mukjlowerbound}
|\mu_{k,j_{k+1}}| \geq \delta (1 - r_{k,j_{k+1}} \xvec_{j_{k+1}}^T\avec_k/A_k) \geq \delta h_{k,j_{k+1}} \geq \delta \tau,
\end{equation}
where the last inequality comes from applying (M1) to the definition of $h_{k,j_{k+1}}$ in \eqref{eqn:hkjdef}. Now, for any $\epsilon > 0$ we have
\begin{align} \nonumber
\PP(|\hat C_{k+1,j_{k+1}} - C_{k+1,j_{k+1}}| <  \epsilon) &= \PP(|\frac{\hat r_{k,j_{k+1}}(\mu_{k,j_{k+1}} + \varepsilon_{k,j_{k+1}})}{1 - \hat r_{k,j_{k+1}} \xvec_{j_{k+1}}^T\avec_k/A_k} - \frac{r_{k,j_{k+1}}\mu_{k,j_{k+1}}}{1 - r_{k,j_{k+1}} \xvec_{j_{k+1}}^T\avec_k/A_k}|< \epsilon) \\\nonumber
& \geq \PP(\frac{r_{k,j_{k+1}}\varepsilon_{k,j_{k+1}}}{1 - r_{k,j_{k+1}} \xvec_{j_{k+1}}^T\avec_k/A_k} < \epsilon \cap |\varepsilon_{k,j_{k+1}}| < \delta\tau) \\\nonumber
&\geq \PP(|\varepsilon_{k,j_{k+1}}| < \tau  \min\{\epsilon,\delta\})\\ \label{eqn:Ckplus1consistent}
&\to 1
\end{align}
as $n \to \infty$, where the first inequality comes from the fact that $\hat r_{k,j_{k+1}}$ will be equal to $r_{k,j_{k+1}}$ when $|\varepsilon_{k,j_{k+1}}| < |\mu_{k,j_{k+1}}|$. The second inequality comes from using again $1 - r_{k,j_{k+1}}\xvec_{j_{k+1}}^T\avec_k/A_k \geq \tau$ by applying (M1) as in \eqref{eqn:mukjlowerbound}.

We next show that the event in \eqref{eqn:Ckplus1jevent} occurs with high probability.  For each $j \notin \calA_k\cup\{j_{k+1}\}$ we have
\begin{equation}
\label{eqn:noiseoverdenom}
|\hat C_{k+1,j} - \frac{\hat r_{k,j}\mu_{k,j}}{1 - \hat r_{k,j}\xvec_k^T\avec_k/A_k}| = \frac{\hat r_{k,j}\varepsilon_{k,j}}{1 - \hat r_{k,j}\xvec_j^T\avec_k/A_k},
\end{equation}
where, applying Lemma \ref{lem:denombound} to the sample path, we may write $1 - \hat r_{k,j}\xvec_j^T\avec_k/A_k \geq \hat h_{k,j}$, with
\begin{equation}
\label{eqn:hkjhatdef}
\hat h_{k,j} \equiv \left\{\begin{array}{ll}
     1 - \dfrac{|(\hat \cvec_k)_j|}{\hat C_k}& \text{ if }  \dfrac{|\xvec_j^T\avec_k|}{A_k} < \dfrac{|(\hat \cvec_k)_j|}{\hat C_k} \text{ and } \sign(\xvec_j^T\avec_k) = \sign((\hat \cvec_k)_j)\\
     1& \text{ otherwise.} 
\end{array}\right.
\end{equation}
We can enforce a lower bound on $\hat h_{k,j}$ as follows:  Suppose the case $|\xvec_j^T\avec_k|/A_k < |(\hat \cvec_k)_j|/\hat C_k$ and $\sign(\xvec_j^T\avec_k) = \sign((\hat \cvec_k)_j)$ occurs (note that in all other cases we have $\hat h_{k,j} =1$).  If the event 
\begin{equation}
\label{eqn:ckCkevent}
\Big|\frac{|(\hat\cvec_k)_j|}{\hat C_k} - \frac{|(\cvec_k)_j|}{C_k}\Big| < \frac{\xi}{2} \quad \cap \quad |(\hat \cvec_k)_j - (\cvec_k)_j| < \frac{\delta \xi}{4},
\end{equation}
occurs, then $|\xvec_j^T\avec_k|/A_k < |(\hat \cvec_k)_j|/\hat C_k$ implies
\[
\frac{|\xvec_j^T\avec_k|}{A_k} < \frac{|(\cvec_k)_j|}{C_k} + \xi,
\]
and, if $|(\cvec_k)_j| > \delta \xi / 2$, the event \eqref{eqn:ckCkevent} also implies $\sign(\xvec_j^T\avec_k) = \sign((\cvec_k)_j)$ so that we can invoke the lower bound $1 - |(\cvec_k)_j|/C_k \geq \tau$ from (M1).  Applying this bound gives
\[
1 - \frac{|(\hat \cvec_k)_j|}{\hat C_k} \geq 1 - \frac{|(\cvec_k)_j|}{C_k} - \frac{\xi}{2} \geq \tau - \frac{\xi}{2} > 0.
\]
On the other hand, if $|(\cvec_k)_j| \leq \delta \xi / 2$, the event \eqref{eqn:ckCkevent} along with $C_k \geq \delta$ from (M2) implies
\[
1 - \frac{|(\hat \cvec_k)_j|}{\hat C_k} \geq 1 - \frac{|(\cvec_k)_j|}{C_k} - \frac{\xi}{2}  \geq 1 - \frac{\delta \xi}{2C_k} - \frac{\xi}{2} \geq 1 - \frac{\delta \xi}{2\delta} - \frac{\xi}{2} = 1 - \xi > 0.
\]
So in all cases the event \eqref{eqn:ckCkevent} implies
\begin{equation}
\label{eqn:denomhatlowerbound}
1 - \hat r_{k,j}\xvec_j^T\avec_k/A_k \geq \min\{\tau - \xi/2,1 - \xi\} > 0.
\end{equation}
From here, for any $\epsilon > 0$ we have
\begin{align*}
    \PP( &\frac{\hat r_{k,j}\varepsilon_{k,j}}{1 - \hat r_{k,j}\xvec_j^T\avec_k/A_k}<\epsilon) \\
    & \geq \PP(|\varepsilon_{k,j}| < \epsilon \min\{\tau - \xi/2,1-\xi\} \cap \Big|\frac{|(\hat\cvec_k)_j|}{\hat C_k} - \frac{|(\cvec_k)_j|}{C_k}\Big| < \frac{\xi}{2} \cap  |(\hat \cvec_k)_j - (\cvec_k)_j| < \frac{\delta \xi}{4})\\
    &\to1
\end{align*}
as $n \to \infty$ for all $j \notin \calA_k \cup \{j_{k+1}\}$, by the assumption that (ii) and (iii) hold for step $k$. By the identity in \eqref{eqn:noiseoverdenom}, the above establishes that the event in \eqref{eqn:Ckplus1jevent} occurs with probability tending to one as $n \to \infty$.

Now that events \eqref{eqn:Ckplus1event} and \eqref{eqn:Ckplus1jevent} have been shown to occur with probability tending to one, we may conclude $\PP(\hat \calA_{k+1}\setminus\hat \calA_k = \{j_{k+1}\}) \to 1$ as $n \to \infty$. On $\hat \calA_{k+1}\setminus\hat \calA_k = \{j_{k+1}\}$ we have $\hat C_{k+1} = \hat C_{k+1,j_{k+1}}$ as well as $\hat s_{k+1} = \hat r_{k,j_{k+1}}$. In consequence, \eqref{eqn:Ckplus1consistent} gives $\hat C_{k+1} - C_{k+1}\inprob 0$, and, since $\hat r_{k,j_{k+1}} = r_{k,j_{k+1}}$ when $|\varepsilon_{k,j_{k+1}}|$ is small enough (see \eqref{eqn:Ckplus1consistent}), where $r_{k,j_{k+1}} = s_{k+1}$, we may also conclude $\PP(\hat s_{k+1} = s_{k+1}) \to 1$ as $n \to \infty$.
We have therefore established (a) and (b).

Now we prove (c). To begin, we use $\bmu_k = \sum_{i=1}^k \gamma_i \avec_i$ and $\gamma_i = (C_i - C_{i+1})/A_i$ to write $\cvec_{k+1} = \Xmat^T(\bmu - \bmu_k)$ as
\[
\cvec_{k+1} = \Xmat^T \bmu - \sum_{i=1}^k\Big(\frac{s_i\mu_{i-1,j_i}}{1 - s_i \xvec_{j_i}^T\avec_{i-1}/A_{i-1}} -\frac{s_{i+1}\mu_{i,j_{i+1}}}{1 - s_{i+1} \xvec_{j_{i+1}}^T\avec_i/A_i} \Big)\frac{1}{A_i}\Xmat^T\avec_i.
\]
So, provided $\hat \calA_{k+1}\setminus \calA_k = \{j_{k+1}\}$ and $\hat s_{k+1} = s_{k+1}$, we may write
\[
\hat \cvec_{k+1} - \cvec_{k+1}= \frac{1}{\sqrt{n}}\Xmat^T\berror_n - \sum_{i=1}^k\Big(\frac{s_i\varepsilon_{i-1,j_i}}{1 - s_i \xvec_{j_i}^T\avec_{i-1}/A_{i-1}} -\frac{s_{i+1}\varepsilon_{i,j_{i+1}}}{1 - s_{i+1} \xvec_{j_{i+1}}^T\avec_i/A_i} \Big)\frac{1}{A_i}\Xmat^T\avec_i.
\]
At this point we introduce $\omega < \infty$ such that
\begin{equation}
\label{eqn:omega}
\|A_k^{-1}\Xmat^T\avec_k\| \leq \omega < \infty
\end{equation}
for $k=1,\dots,m$ for large enough $n$. The existence of such an $\omega$ is guaranteed by (R2), according to which 
\[
\Xmat^T\Xmat = \Dmat_n \Xmat_n^T\Xmat_n \Dmat_n = n^{1/2}\Dmat_n n^{-1}\Xmat_n^T\Xmat_n n^{1/2}\Dmat_n \to \Dmat \bSigma \Dmat = \Rmat
\]
as $n \to \infty$, so that, for each $k=1,\dots,m$, we have
\[
 A_k^{-1}\Xmat^T\avec_k = \Xmat^T\Xmat_k(\Xmat_k^T\Xmat_k)^{-1}\onevec  = \Xmat^T\Xmat_{\calA_k}(\Xmat_{\calA_k}^T\Xmat_{\calA_k})^{-1}\svec_k \to \Rmat_{\{1,\dots,p\},\calA_k}(\Rmat_{\calA_k,\calA_k})^{-1}\svec_k,
\]
as $n \to \infty$, where we have introduced the vector of signs $\svec_k = (s_1,\dots,s_k)^T$. The norm of the limiting vector is finite.
Making use of the quanity $\omega$, we have for large enough $n$ the bound
\[
\|\hat \cvec_{k+1} - \cvec_{k+1}\|^2 \leq \|n^{-1/2}\Xmat^T\berror_n\|^2 + \frac{\omega^2}{\tau^2}\sum_{i=1}^k(s_i \varepsilon_{i-1,j_i} - s_{i+1}\varepsilon_{i,j_{i+1}})^2,
\]
where we have bounded $1 - s_i \xvec_{j_i}^T\avec_{i-1}/A_{i-1}$ and $1 - s_{i+1} \xvec_{j_{i+1}}^T\avec_i/A_i$ from below by $\tau$ by applying (M1) as in \eqref{eqn:mukjlowerbound}. Now, for any $\epsilon > 0$ we have
\begin{align*}
    \PP(\|\hat \cvec_{k+1} - \cvec_{k+1}\|^2 < \epsilon) &\geq \PP(\|n^{-1/2}\Xmat^T\berror_n\|^2 + \frac{\omega^2}{\tau^2}\sum_{i=1}^k(s_i \varepsilon_{i-1,j_i} - s_{i+1}\varepsilon_{i,j_{i+1}})^2 < \epsilon \\
     & \quad \quad \quad \quad \quad \quad \quad \quad \quad \quad \cap \hat \calA_{k+1}\setminus \calA_k = \{j_{k+1}\} \cap \hat s_{k+1} = s_{k+1})\\
     &\to 1
\end{align*}
as $n \to \infty$. This proves (c), so (i), (ii), and (iii) are proven by induction.

To prove (iv), it suffices to consider $k=m+1$, due to the fact that the $\hat C_k$ are strictly decreasing.  Since $C_{m+1} = 0$, we have
\[
C_{m+1,j} = \frac{r_{m,j}\mu_{m,j}}{1 - r_{m,j}\xvec_j^T\avec_m/A_m} = 0
\]
and thus $\mu_{m,j} = 0$ for all $j \notin \calA_m$. So on the event
\begin{equation}
\label{eqn:L_m}
\calL_m = \cap_{k=1}^m\{ \hat \calA_{k}\setminus\hat\calA_{k-1} = \{j_{k}\} \text{ and } \hat s_k = s_k\}
\end{equation}
we may write
\[
\hat C_{m+1,j} = \frac{\hat r_{m,j} \varepsilon_{m,j}}{1 - \hat r_{m,j}\xvec_j^T\avec_m/A_m},
\]
where $\hat r_{m,j} = \sign(\varepsilon_{m,j})$. From here, for any $\epsilon >0$ we have
\begin{align}\nonumber
    \PP( &\hat C_{m+1,j}<\epsilon) \\\nonumber
    & \geq \PP(|\varepsilon_{m,j}| < \epsilon \min\{\tau - \xi/2,1-\xi\} \cap \Big|\frac{|(\hat\cvec_m)_j|}{\hat C_m} - \frac{|(\cvec_m)_j|}{C_m}\Big| < \frac{\xi}{2} \cap  |(\hat \cvec_m)_j - (\cvec_m)_j| < \frac{\delta \xi}{4})\\\label{eqn:Cmplus1j}
    &\to 1
\end{align}
as $n \to \infty$ by (ii) and (iii), where we have enforced a lower bound on $1 - \hat r_{m,j}\xvec_j^T\avec_m/A_m$ with the same arguments as those leading to \eqref{eqn:denomhatlowerbound}. Since \eqref{eqn:Cmplus1j} holds for all $j \notin \calA_m$, (iv) is proven.

\end{proof}

\begin{proof}{\bf of Theorem \ref{thm:Tasymp}:} Theorem \ref{thm:larconsistency} gives that the event $\calL_m$ defined in \eqref{eqn:L_m}
holds with probability tending to 1 as $n\to \infty$.  On $\calL_m$, the sample path angle $\hat A_k$, equi-angular vector $\hat \avec_k$, and sign $\hat s_k$ are equal to their population path counterparts $A_k$, $\avec_k$, and $s_k$, respectively, on steps $k=1,\dots,m$. On $\calL_m$ we may therefore write
\begin{equation}
\label{eqn:Tjk}
T_{nk} = s_k\Big(\frac{1}{A_{k}^{2}} - \frac{1}{A_{k - 1}^{2}}\Big)^{1/2}\sqrt{n}(\hat C_k - C_k),
\end{equation}
for each $k=1,\dots,m$ and
\[
\hat C_k - C_k = \Big(\frac{1}{A_k^2} - \frac{1}{A_{k-1}^2}\Big)^{-1}\Big(\frac{1}{A_{k}}\avec_{k} - \frac{1}{A_{k-1}}\avec_{k-1}\Big)^T\frac{1}{\sqrt{n}}\berror_n
\]
from \eqref{eqn:Ckangles}, noting that $\yvec - \bmu = n^{-1/2}\berror_n$. From here, \eqref{eqn:Aupdate} and \eqref{eqn:aupdate} and the scaling $\berror_n = n^{1/2}\berror$ give the simplification
\begin{equation}
\label{eqn:Tksimple}
T_{nk} = \frac{\evec_k^T\berror_n}{\|\evec_k\|},
\end{equation}
where $\evec_1,\dots,\evec_k$ are defined in \eqref{eqn:in} for $k=1,\dots,m$.  Then, noting that the vectors $\evec_{1},\dots,\evec_{m}$ are mutually orthogonal, we have
\[
(T_{n1},\dots,T_{nm})^T = \Big(\frac{1}{\|\evec_1\|}\evec_1,\dots,\frac{1}{\|\evec_m\|}\evec_m\Big)^T\berror_n \indist \calN(\zerovec,\sigma^2\Imat_m)
\]
as $n\to \infty$ under Assumption~\ref{assum:regassumptions}, which proves (i).
We next consider the asymptotic behavior of $T_{n,m+1},\dots,T_{np}$.  Beyond step $m$, there is no sequence of active sets which occurs with probability tending to 1, so we must consider all possible sequences of active sets.
Assume $\calL_m$ occurs and suppose the index $j \notin \calA_m$ enters the active set on step $m + 1$, such that $\hat \calA_{m+1}\setminus \calA_m = \{j\}$.  This would lead to $\hat \evec_{m+1} = (\Imat - \Pmat_m)\xvec_j$, so,  noting that $(\Imat - \Pmat_m)\bmu =\zerovec$, we may, as in \eqref{eqn:Tksimple},  write
\begin{equation}
\label{eqn:Tnmplus1}
T_{n,m+1} = \frac{\hat \evec_{m+1}^T\berror_n}{\|\hat \evec_{m+1}\|}= \frac{\xvec_j^T(\Imat - \Pmat_m)\berror_n}{\|(\Imat - \Pmat_m)\xvec_j\|} \quad \text{ if } \hat 
 \calA_{m+1}\setminus \calA_m = \{j\}.
\end{equation}
Still assuming the event $\calL_m$ occurs, suppose the variables $\pi_1,\dots,\pi_i$ enter on steps $m+1,\dots,m+i$ for some $i \geq 1$, and then the variable $j$ enters on step $m+i+1$.  This would lead to $\hat \evec_{m+i+1}$ as $\hat \evec_{m+i+1}= (\Imat - \Pmat_{m,\{\pi_1,\dots,\pi_i\}})\xvec_j$, where $\Pmat_{m,\{\pi_1,\dots,\pi_i\}}$ denotes the orthogonal projection onto the span of the columns $\{\xvec_j, j \in \calA_m\cup\{\pi_1,\dots,\pi_i\} \}$. So, just as in \eqref{eqn:Tnmplus1}, we may write
\begin{align}
\label{eqn:Tnmplusiplus1}
&T_{n,m+i+1} = \frac{\hat \evec_{m+i+1}^T\berror_n}{\|\hat \evec_{m+i+1}\|} = \frac{\xvec_j^T(\Imat - \Pmat_{m,\{\pi_1,\dots,\pi_i\}})\berror_n}{\|(\Imat - \Pmat_{m,\{\pi_1,\dots,\pi_i\}})\xvec_j\|} \\ \nonumber
& \quad \quad \quad \text{ if }
       \hat \calA_{m+i+1}\setminus \hat\calA_{m+i} = \{j\}
      \text{ and } \hat \calA_{m+i}\setminus \calA_m = \{\pi_1,\dots,\pi_i\}.
\end{align}

Now set $\Rmat_n = \Xmat^T\Xmat$ and $\Vmat_n = \Rmat_{n,\calA_m^C|\calA_m}$, and index the rows and columns of $\Vmat$ with the indices $j \notin \calA_m$.  Then, for a generic vector $\tvec = (t_{j},j \notin \calA_m)^T \in \mathbb{R}^{p-m}$ define for each $j \notin \calA_m$ the function $J_{nj}(\tvec) \equiv t_j$ and for each $j \notin \calA_m$ and each subset of indices $\{\pi_1,\dots,\pi_i\} \subset \{1,\dots,p\}\setminus\{\calA_m \cup \{j\}\}$, define the function
\[
J_{n,\{\pi_1,\dots,\pi_i\},j}(\tvec) \equiv t_j - (t_{\pi_1},\dots,t_{\pi_i})^T(\Vmat_{n,\{\pi_1,\dots,\pi_i\},\{\pi_1,\dots,\pi_i\}})^{-1}\Vmat_{n,\{\pi_1,\dots,\pi_i\},j}.
\]
Then define standardized versions of $J_{nj}(\tvec)$ and $J_{n,\{\pi_1,\dots,\pi_i\},j}(\tvec)$ as
\begin{align}\label{eqn:Inj}
I_{nj}(\tvec) &\equiv J_{nj}(\tvec)/\sqrt{\Vmat_{n,j,j}} \\\label{eqn:Inpij}
I_{n,\{\pi_1,\dots,\pi_i\},j}(\tvec) &\equiv J_{n,\{\pi_1,\dots,\pi_i\},j}(\tvec)/\sqrt{\Vmat_{n,\{j\}|\{\pi_1,\dots,\pi_i\}}}.
\end{align}
Then, defining the random vector $\zvec_n \equiv (Z_{n,j},j \notin \calA_m)^T$, where
\begin{equation}
\label{eqn:zvecentries}
Z_{nj} \equiv \xvec_j^T(\Imat - \Pmat_m)\berror_n, \quad j \notin \calA_m,
\end{equation}
we find we can restate \eqref{eqn:Tnmplus1} as 
$T_{n,m+1} = I_{nj}(\zvec_n)$ if $ \hat \calA_{m+1}\setminus \calA_m = \{j\}$ 
and \eqref{eqn:Tnmplusiplus1} as
$T_{n,m+i+1}  = I_{n,\{\pi_1,\dots,\pi_i\},j}(\zvec_n) $ if $  \hat \calA_{m+i+1}\setminus \hat\calA_{m+i} = \{j\}$ and  $\hat \calA_{m+i}\setminus \calA_m = \{\pi_1,\dots,\pi_i\}$. 
Next we show that on the event $\calL_m$ the order of entrance of the columns $j \notin \calA_m$ into the active set is determined by the random vector $\zvec_n$. For each $j \notin \calA_m$ define the scalar
\begin{equation}
\label{eqn:anj}
a_{nj} \equiv \frac{\xvec_j^T\avec_m}{A_m},
\end{equation}
and for each $j\notin \calA_m$ and each subset of indices $\{\pi_1,\dots,\pi_i\} \subset \{1,\dots,p\}\setminus\{\calA_m \cup \{j\}\}$, define a function of a generic vector $\tvec = (t_j, j \notin \calA_m)^T \in \mathbb{R}^{p-m}$ as
\begin{equation}
\label{eqn:anpij}
a_{n,\{\pi_1,\dots,\pi_i\},j}(\tvec) \equiv \frac{\xvec_j^T\avec_m(u_{n1}\xvec_{\pi_1},\dots,u_{ni}\xvec_{\pi_i})}{A_m(u_{n1}\xvec_{\pi_1},\dots,u_{ni}\xvec_{\pi_i})},
\end{equation}
where $\avec_m(u_{n1}\xvec_{\pi_1},\dots,u_{ni}\xvec_{\pi_i})$ and $A_m(u_{n1}\xvec_{\pi_1},\dots,u_{ni}\xvec_{\pi_i})$ are the equiangular vector and angle computed from the matrix $[\Xmat_m ~ u_{n1}\xvec_{\pi_1} \cdots u_{ni}\xvec_{\pi_i}]$, with signs $u_{n1},\dots,u_{ni}$ given by
\[
(u_{n1},\dots,u_{ni}) = \big(\sign(J_{n\pi_1}(\tvec)),\sign(J_{n,\{\pi_1\},\pi_2}(\tvec)),\dots, \sign(J_{n,\{\pi_1,\dots,\pi_{i-1}\},\pi_i}(\tvec)\big).
\]
Now, for each $j \notin \calA_m$ define the set
\begin{align} \label{eqn:Enj}
&\calE_{nj} \equiv \Bigg\{ \tvec = (t_l, l \notin \calA_m)^T\in \mathbb{R}^{p-m} :~\underset{j' \notin \calA_m}{\argmax}~\Big\{\frac{|J_{nj'}(\tvec)|}{1 - \sign(J_{nj'}(\tvec)) a_{nj'}}\Big\} = j \Bigg\}
\end{align}
and for each $j\notin \calA_m$ and each subset of indices $\{\pi_1,\dots,\pi_i\} \subset \{1,\dots,p\}\setminus\{\calA_m \cup \{j\}\}$ define the set
\begin{align}\label{eqn:Enpij}
&\calE_{n,\{\pi_1,\dots,\pi_i\},j} \equiv \Bigg\{ \tvec = (t_l, l \notin \calA_m)^T\in \mathbb{R}^{p-m} : \\ \nonumber
& \hspace{.2in}\underset{j'\notin \calA_m \cup\{\pi_1,\dots,\pi_i\}}{\argmax}~\Big\{\frac{|J_{n,\{\pi_1,\dots,\pi_i\},j'}(\tvec)|}{1 - \sign(J_{n,\{\pi_1,\dots,\pi_i\},j'}(\tvec)) a_{n,\{\pi_1,\dots,\pi_i\},j'}(\tvec)}\Big\} = j  \Bigg\}.
\end{align}
Next, for each permutation $\bpi = (\pi_1,\dots,\pi_{p-m})$ of the indices $j \notin \calA_m$ define the set
$
\calO_{n,\bpi} \equiv \calE_{n\pi_1}\cap\calE_{n,\{\pi_1\},\pi_2}\cap \dots \cap\calE_{n,\{\pi_1,\dots,\pi_{p-m-1}\},\pi_{p-m}}$,
which is the set such that when $\zvec_n \in \calO_{n,\bpi}$, the indices $\pi_1,\dots,\pi_{p-m} \notin \calA_m$ enter the sample path active set on the steps $m+1,\dots,p$, respectively.  Now we may write
\begin{equation}  \label{eqn:Tnkmplus1top}
(T_{n,m+1},\dots,T_{np}) = \sum_{\bpi \in \Pi} \Big(I_{n,\pi_1}(\zvec_n),I_{n,\{\pi_{1}\},\pi_{2}}(\zvec_n),
 \dots, I_{n,\{\pi_{1},\dots,\pi_{p-m-1}\},\pi_{p-m}}(\zvec_n) \Big)\mathbb{I}(\zvec_n \in \calO_{n,\bpi}).
\end{equation}
It then follows from (R2) that 
$\Vmat_{n,j,j}  \to \Vmat_{j,j}$,  
$\Vmat_{n,\{j\}|\{\pi_1,\dots,\pi_i\}} \to \Vmat_{\{j\}|\{\pi_1,\dots,\pi_i\}}$, and
\[(\Vmat_{n,\{\pi_1,\dots,\pi_i\},\{\pi_1,\dots,\pi_i\}})^{-1}\Vmat_{n,\{\pi_1,\dots,\pi_i\},j}  \to (\Vmat_{\{\pi_1,\dots,\pi_i\},\{\pi_1,\dots,\pi_i\}})^{-1}\Vmat_{\{\pi_1,\dots,\pi_i\},j}\]
as $n \to \infty$.  We may then define asymptotic counterparts to  $J_{nj}(\tvec)$ and $J_{n,\{\pi_1,\dots,\pi_i\},j}(\tvec)$ as $J_j(\tvec) \equiv t_j$ and 
$J_{\{\pi_1,\dots,\pi_i\},j}(\tvec) \equiv t_j - (t_{\pi_1},\dots,t_{\pi_i})^T(\Vmat_{\{\pi_1,\dots,\pi_i\},\{\pi_1,\dots,\pi_i\}})^{-1}\Vmat_{\{\pi_1,\dots,\pi_i\},j}$,
respectively, as well as asymptotic counterparts to their standardized versions $I_{nj}(\tvec)$ and $I_{n,\{\pi_1,\dots,\pi_i\},j}(\tvec)$ as
$I_j(\tvec)  \equiv t_j / \sqrt{\Vmat_{j,j}}$ and $
I_{\{\pi_1,\dots,\pi_i\},j}(\tvec)  \equiv J_{\{\pi_1,\dots,\pi_i\},j}(\tvec)/\sqrt{\Vmat_{\{j\}|\{\pi_1,\dots,\pi_i\}}}$.
% \begin{align*}
% I_j(\tvec) &= t_j / \Vmat_{j,j}^{1/2} \\
% I_{\{\pi_1,\dots,\pi_i\},j}(\tvec) &= J_{\{\pi_1,\dots,\pi_i\},j}(\tvec) /\Vmat_{\{j\}|\{\pi_1,\dots,\pi_i\}}^{1/2}
% \end{align*}
Then we have, for each $\tvec = (t_j, j \notin \calA_m)^T\in \mathbb{R}^{p-m}$, pointwise convergence such that
\begin{equation}
\label{eqn:Jconverge}
J_{nj}(\tvec) \to J_j(\tvec) \quad \text{ and } \quad J_{n,\{\pi_1,\dots,\pi_i\},j}(\tvec) \to J_{\{\pi_1,\dots,\pi_i\},j}(\tvec)
\end{equation}
as well as
\begin{equation}
\label{eqn:Iconverge}
I_{nj}(\tvec) \to I_j(\tvec) \quad \text{ and } \quad I_{n,\{\pi_1,\dots,\pi_i\},j}(\tvec) \to I_{\{\pi_1,\dots,\pi_i\},j}(\tvec)
\end{equation}
as $n \to \infty$ for each $j \notin \calA_m$ and each subset of indices $\{\pi_1,\dots,\pi_i\} \subset \{1,\dots,p\}\setminus\{\calA_m \cup \{j\}\}$.
Additionally, for each $j \notin \calA_m$ let $a_j \equiv \Rmat_{j,\calA_m}(\Rmat_{\calA_m,\calA_m})^{-1}(s_1,\dots,s_m)^T$ and for each subset of indices $\{\pi_1,\dots,\pi_i\} \subset \{1,\dots,p\}\setminus\{\calA_m \cup \{j\}\}$ define for a vector $\tvec = (t_j , j \notin \calA_m)^T\in \mathbb{R}^{p-m}$ the function
\[
a_{\{\pi_1,\dots,\pi_i\}}(\tvec) \equiv \Rmat_{j,\{\calA_m,\pi_1,\dots,\pi_i\}}(\Rmat_{\{\calA_m,\pi_1,\dots,\pi_i\},\{\calA_m,\pi_1,\dots,\pi_i\}})^{-1} (s_1,\dots,s_m,u_1,\dots,u_i)^T,
\]
where
$(u_{1},\dots,u_{i}) =\big(\sign(J_{\pi_1}(\tvec)),\sign(J_{\{\pi_1\},\pi_2}(\tvec)),\dots, \sign(J_{\{\pi_1,\dots,\pi_{i-1}\},\pi_i}(\tvec)\big)$.
Then (R2) gives $a_{nj} \to a_j$ for each $j \notin \calA_m$ as well as the pointwise convergence  $a_{n,\{\pi_1,\dots,\pi_i\}}(\tvec) \to a_{\{\pi_1,\dots,\pi_i\}}(\tvec)$ for each $\tvec = (t_j , j \notin \calA_m)^T \in \mathbb{R}^{p-m}$ for each $j \notin \calA_m$ and each subset of indices $\{\pi_1,\dots,\pi_i\} \subset \{1,\dots,p\}\setminus\{\calA_m \cup \{j\}\}$.
From here we define for each $j \notin \calA_m$ an asymptotic version of $\calE_{nj}$ as
\begin{align} \label{eqn:Ej}
&\calE_j \equiv \Big\{ \tvec =(t_j, j \notin \calA_m) \in \mathbb{R}^{p-m} :~\underset{j'\notin \calA_m}{\argmax}~\Big\{\frac{|J_{j'}(\tvec)|}{1 - \sign(J_{j'}(\tvec)) a_{j'}}\Big\}  = j\Big\}
\end{align}
and for each $j \notin \calA_m$ and each subset of indices $\{\pi_1,\dots,\pi_i\} \subset \{1,\dots,p\}\setminus\{\calA_m \cup \{j\}\}$ an asymptotic version of $\calE_{n,\{\pi_1,\dots,\pi_i\},j}$ as
\begin{align}\label{eqn:Epij}
&\calE_{\{\pi_1,\dots,\pi_i\},j} \equiv \Big\{ \tvec = (t_j,j\notin\calA_m)\in \mathbb{R}^{p-m} : \\ \nonumber
&\underset{j'\notin \calA_m \cup\{\pi_1,\dots,\pi_i\}}{\argmax}~\Big\{\frac{|J_{\{\pi_1,\dots,\pi_i\},j'}(\tvec)|}{1 - \sign(J_{\{\pi_1,\dots,\pi_i\},j'}(\tvec)) a_{\{\pi_1,\dots,\pi_i\},j'}(\tvec)} \Big\} = j \Big\}
\end{align}
Moreover, define the set
\begin{equation}
\label{eqn:orderingset}
\calO_{\bpi} \equiv \calE_{\pi_1}\cap\calE_{\{\pi_1\},\pi_2}\cap \dots \cap\calE_{\{\pi_1,\dots,\pi_{p-m-1}\},\pi_{p-m}}.
\end{equation}
Then for each argument $\tvec = (t_j,j \notin \calA_m)\in \mathbb{R}^{p-m}$, we have
$\mathbb{I}(\tvec \in \calO_{n,\bpi}) \to \mathbb{I}(\tvec \in \calO_{\bpi})$
as $n \to \infty$, except possibly for a collection of arguments $\tvec$ having Lebesgue measure zero in $\mathbb{R}^{p-m}$.  Note also that the indicator function $\mathbb{I}(\tvec \in \calO_{\bpi})$ has a set of discontinuities with Lebesgue measure zero in $\mathbb{R}^{p-m}$.  Now, because $\zvec_n \indist \zvec$ and the latter random vector has a normal distribution on $\mathbb{R}^{p-m}$ that has measure zero on sets of Lebesgue measure zero, the extended continuous mapping theorem gives
\[
(T_{n,m+1},\dots,T_{np}) \indist\sum_{\bpi \in \Pi} \Big(I_{\pi_1}(\zvec),I_{\{\pi_{1}\},\pi_{2}}(\zvec), \dots, I_{\{\pi_{1},\dots,\pi_{p-m-1}\},\pi_{p-m}}(\zvec) \Big)\mathbb{I}(\zvec \in \calO_{\bpi})
\]
as $n \to \infty$, where $\Pi$ is the set of all permutations $\bpi = (\pi_1,\dots,\pi_{p-m})$ of the indices $j \notin \calA_m$.
This proves the second part of (iii).
Statement (ii) follows from noting that for each $k = 1,\dots,m$ we have $\Cov(\zvec_n, \evec_k^T\berror_n) = 0$ since $(\Imat - \Pmat_m)\xvec_{j_k} = 0$.  Since on event $\calL_m$, $T_{n1},\dots,T_{nm}$ are functions of $\evec_1^T\berror_n,\dots,\evec_m^T\berror_n$ and $T_{n,m+1},\dots,T_{np}$ are functions of $\zvec_n$, the result follows.
Now we prove the first part of (iii).  Assume $n$ is large and that the event $\calL_m$ holds.  Then, by the same arguments by which we could write \eqref{eqn:Tksimple}, \eqref{eqn:Tnmplus1}, and \eqref{eqn:Tnmplusiplus1}, we may write
% on steps $k = m+1,\dots,p$, let $\hat \evec_k = (\Imat - \hat \Pmat_{k-1})\hat \bvec_k$, where $\hat \bvec_k$ is the column of $\Xmat$ entering the active set on step $k$ of $\LAR(\Xmat,\yvec)$. Then using the identity \eqref{eqn:Ckangles} as well as \eqref{eqn:Aupdate} and \eqref{eqn:aupdate}, we may write
$
    \sum_{k=m+1}^p T_{nk}^2 / \sigma^2 = \frac{1}{\sigma^2}\sum_{k=m+1}^p \frac{(\hat \evec_k^T\berror_n)^2}{\|\hat \evec_k\|^2}$,
where $\hat \evec_k = (\Imat - \hat \Pmat_{k-1})\hat \hvec_k$, where $\hat \hvec_k$ is the column of $\Xmat$ entering the active set on step $k$ of $\LAR(\Xmat,\yvec)$ for steps $k=m+1,\dots,p$. Now, on the event $\calL_m$ the vectors $\hat \evec_{m+1},\dots,\hat\evec_p$ form an orthogonal basis for the column space of the matrix $(\Imat - \Pmat_m)\Xmat$, so that $
\frac{1}{\sigma^2} \sum_{k=m+1}^p T_{nk}^2 = \frac{1}{\sigma^2} \berror_n^T\Pmat\berror_n$,
where $\Pmat$ denotes the orthogonal projection onto the column space of $(\Imat - \Pmat_m) \Xmat$. We see that this converges in distribution to a random variable having the $\chi^2_{p-m}$ distribution.
\end{proof}

\begin{proof}{\bf of Theorem \ref{thm:mest}:}
 Note first that  for $\hat \sigma_n^2 = (n-p)^{-1} \yvec_n^T(\Imat - \Pmat_\Xmat)\yvec_n \equiv (n-p)^{-1} \berror_n^T(\Imat - \Pmat_\Xmat)\berror_n$ from 
(\ref{eqn:varest}), it holds that
$ n^{-1}\berror_n^T \berror_n \inprob \sigma^2>0$ as $n\to \infty$ by (R3), while $|(n-p) n^{-1}\hat \sigma_n^2 -  n^{-1}\berror_n^T \berror_n| = n^{-1}\berror_n^T \Pmat_{\Xmat} \berror_n$ where $n^{-1}\E(\berror_n^T \Pmat_{\Xmat} \berror_n) = p \sigma^2/n \rightarrow 0$ by (R1).  Consequently, we have $\hat \sigma_n^2 \inprob \sigma^2$ as $n \to \infty$. 
Consider the case $m = 0$.  Introducing the random variable $S_1 \sim \chi^2_p$, we have $\PP(S_1 > \chi^2_{p,n^{-1}}) = n^{-1}$ and $\hat S_{n,1}\indist S_1$ as $n \to \infty$  (i.e., the latter by Theorem~\ref{thm:Tasymp} and $\hat \sigma_n^2 \inprob \sigma^2$ with Slutsky's theorem), so we may write
\begin{align*}
\PP(\bar m \neq 0)  = \PP( \hat S_{n,1} > \chi^2_{p,n^{-1}}) 
&\leq |\PP( \hat S_{n,1} > \chi^2_{p,n^{-1}}) - \PP(S_1 > \chi^2_{p,n^{-1}})| + n^{-1} \\
&\leq \sup_{x \in \mathbb{R}}|\PP( \hat S_{n,1} \leq x  ) - \PP(S_1 \leq x)| + n^{-1}  \to 0  
\end{align*}
as $ n \to \infty$, using Polya's theorem on the convergence of distribution functions.  This establishes consistency, $\PP(\bar m = 0) \to 1$ as $n \to \infty$, in this case.  
Now consider the case $m \in \{1,\dots,p\}$.  We have
\begin{align} \nonumber
\PP(\bar m = m) &= \PP( \cap_{k=1}^m \{\hat S_{n,k} > \chi^2_{p-k+1,n^{-1}}\} \cap \{\hat S_{n,m+1} \leq \chi^2_{p-m,n^{-1}}\}) \\\nonumber
        &\geq \PP( \{\hat S_{n,m} > \chi^2_{p,n^{-1}} \} \cap \{\hat S_{n,m+1} \leq \chi^2_{p-m,n^{-1}}\}) \\ \label{eqn:need2bound}
        & \geq 1 - \PP(\hat S_{n,m} \leq \chi^2_{p,n^{-1}}) - \PP(\hat S_{n,m+1} > \chi^2_{p-m,n^{-1}}),
\end{align}    
where the second inequality comes from noting $\hat S_{n,m} \leq \dots \leq \hat S_{n,1}$ and that the quantile $\chi^2_{k,n^{-1}}$ is strictly increasing in $k$.  It is sufficient to show that the two probabilities in \eqref{eqn:need2bound} converge to zero as $n \to \infty$. To bound the second probability, we introduce $S_{m+1} \sim \chi^2_{p-m}$ and write
\begin{eqnarray*}
\PP(\hat S_{n,m+1} > \chi^2_{p-m,n^{-1}}) &\leq &|\PP(\hat S_{n,m+1} > \chi^2_{p-m,n^{-1}})  - \PP(S_{m+1} > \chi^2_{p-m,n^{-1}}) | + n^{-1}\\
& \leq &  \sup_{x\in\mathbb{R}}|\PP(\hat S_{n,m+1} \leq x)  - \PP(S_{m+1} \leq x) | + n^{-1}
\to 0
\end{eqnarray*}
as $n \to \infty$, using $\hat S_{n,m+1} \indist S_{m+1}$  (i.e., by Theorem~\ref{thm:Tasymp} and $\hat{\sigma}^2_n \inprob \sigma^2$)  with Polya's theorem. To bound the first probability in \eqref{eqn:need2bound}, we introduce $\phi_{n,m}  \equiv  n (A_m^{-2} - A_{m-1}^{-2}) C_m^2$   and  note also that $\hat S_{n,m} \geq \hat W_{n,m}$ and, on the event $\calL_m$ in \eqref{eqn:L_m}, $\hat W_{n,m} \geq (\phi_{n,m}^{1/2} -  |T_{nm} | )^2 /\hat\sigma_n^2 $ holds for $T_{nm}$ in (\ref{eqn:T}), in order to write
\begin{align*}
\PP(\hat S_{n,m} \leq \chi^2_{p,n^{-1}})  \;\leq \;\PP\left( (\phi_{n,m}^{1/2} -  |T_{nm}| )^2  /\hat\sigma_n^2\leq \chi^2_{p,n^{-1}}\right)   \;\leq\; 
\PP\left( (\phi_{n,m}^{1/2} -  |T_{nm} |)^2/\hat\sigma_n^2\leq q_{n,p}  \right)
\end{align*}
using above 
\begin{equation}
    \chi^2_{p,n^{-1}} \leq q_{n,p} \equiv  p + 2 \log n + 2 \sqrt{ p \log n}
\end{equation}
for all $n \geq 1$ (cf.~\cite{inglot2010inequalities}).  Since $  \hat\sigma_n^2 \inprob  \sigma^2>0$ as $n \to \infty$, it suffices now to show that 
$(\phi_{n,m}^{1/2} -  |T_{nm}|)^2/q_{n,p} \stackrel{p}{\rightarrow} \infty$ diverges as $n\to \infty$.  This follows using that $\{|T_{nm}|\}_{n=1}^\infty$ is stochastically bounded by $T_{nm} \indist \mathcal{N}(0,\sigma^2)$ as $n\to\infty$  along with the fact $q_{n,p}$ grows at the rate $\log n$  while, for large enough $n$, we have $\phi_{n,m}/n$ is bounded away from zero by
\begin{align*}
    \phi_{n,m} \equiv n(A_m^{-2} - A_{m-1}^{-2})C_m^2  
    &= n(1 - s_m \xvec_{j_m}^T\avec_{m-1}/A_{m-1})^2C_m^2/ \evec_m^T\evec_m \\
    & \geq n\tau^2C_m^2/  \evec_m^T\evec_m  \\
    & \geq n (\tau\delta)^2/(2 \Rmat_{\{j_m\}|\calA_{m-1}}),
\end{align*}
 where the second equality comes from \eqref{eqn:Aupdate}, the first inequality comes from invoking Lemma \ref{lem:denombound} together with the condition (M1), and the third inequality from $C_m > \delta$ by (M2) (see Figure \ref{fig:margincond2}), as well as $\evec_m^T\evec_m = \xvec_{j_m}(\Imat - \Pmat_{m-1})\xvec_{j_m} \to \Rmat_{\{j_k\}|\calA_{m-1}}>0$ as $n \to \infty$.   
\end{proof}

\section{Proofs of Section 5 results}
\label{sec:bootstrapproof}

Here we present a result giving the asymptotic joint distribution of centered and scaled step coefficients $\sqrt{n}((\hat \bvec_k)_j - (\bvec_k)_j)$, $j \in \hat \calA_k$, $k = 1,\dots,\bar m$ as $n \to \infty$, where $\hat \bvec_{\bar m}$ is re-defined as $\hat \bvec_{\bar m} = \bar \bvec_{\bar m}$.  While the step coefficients have a complicated covariance structure, they are found to converge jointly to a multivariate normal distribution. To express the result, introduce, in the context of the prototypical LAR path in Definition~\ref{def:1}, the collection of random vectors $\{\vvec_k \in \mathbb{R}^k,k=1,\dots,m\}$ such that the vector $\vvec = (\vvec_1^T,\dots,\vvec_m^T)^T$ is a $m(m+1)/2\times 1$ multivariate normal random vector having mean zero and covariance matrix such that 
\begin{equation}
\label{eqn:covvk}
\Var(\vvec_k) = \sigma ^2 \Rmat^{-1}_{\calA_k,\calA_k}\Big[\Rmat_{\calA_k,\calA_k} + \lambda_k^2\Rmat_{\{j_{k+1}\}|\calA_k}\svec_k\svec_k^T\Big]\Rmat^{-1}_{\calA_k,\calA_k}
\end{equation}
for $k=1,\dots,m-1$, $\Var(\vvec_m) = \sigma ^2 \Rmat^{-1}_{\calA_m,\calA_m}$, and 
\begin{equation}
\label{eqn:covvkvk}
\Cov(\vvec_k,\vvec_{k'}) = \sigma^2\Rmat^{-1}_{\calA_k,\calA_k}\Big[\Rmat_{\calA_k,\calA_k} - \lambda_k\svec_k\Rmat_{j_{k+1},\calA_{k'}|\calA_k}^T\Big]\Rmat^{-1}_{\calA_{k'},\calA_{k'}}
\end{equation}
for $1 \leq k < k' \leq m$, where $\svec_k$ is the vector of signs $\svec_k\equiv(s_1,\dots,s_k)^T$ and
\[
\lambda_k \equiv \frac{s_{k+1}}{1 - s_{k+1}\Rmat_{j_{k+1},\calA_k}\Rmat_{\calA_k,\calA_k}^{-1}\svec_k}
\]
for $k=1,\dots,m-1$. Then we have the following result:

\begin{theorem}
\label{thm:coef}
Under the Theorem \ref{thm:larconsistency} conditions and with $\bar m$ as in \eqref{eqn:mest} and $\hat \bvec_{\bar m}$ re-defined as $\hat \bvec_{\bar m} \equiv \bar \bvec_{\bar m}$, we have
\[
\sqrt{n}[((\hat \bvec_1)_{\hat\calA_1} - (\bvec_{1})_{\hat\calA_1})^T,\dots,((\hat \bvec_m)_{\hat\calA_m} - (\bvec_{m})_{\hat\calA_m})^T]^T \indist  \vvec \quad \mbox{as $n\to\infty$.}
\]
\end{theorem}

\begin{proof}{\bf of Theorem \ref{thm:coef}:}
First, note that the vector $\bvec_k$ satisfying $\Xmat \bvec_k = \bmu_k$ and having entries $(\bvec_k)_j = 0$ for $j \notin \calA_k$ has non-zero entries given by
$
(\bvec_k)_{\calA_k} = (\Xmat_{\calA_k}^T\Xmat_{\calA_k})^{-1}\Xmat^T_{\calA_k}\bmu_k$.
Now, as $n\to \infty$, we will have $\hat \calA_k = \calA_k$ and $\hat s_k = s_k$ for $k = 1,\dots,m$ as well as $\bar m = m$ with probability approaching one.  Assuming this event holds, for any step $k < m$ we have
\begin{equation}
\label{eqn:rootnbk}
\sqrt{n}((\hat \bvec_k)_{\hat\calA_k}  - (\bvec_k)_{\hat \calA_k}) = (\Xmat_{\calA_k}^T\Xmat_{\calA_k})^{-1}\Xmat_{\calA_k}^T \sqrt{n}(\hat \yvec_k - \bmu_k),
\end{equation}
where Lemmas \ref{lem:Ck} and \ref{lem:muk} give
\begin{align*}
    \sqrt{n}(\hat \yvec_k - \bmu_k)  = \sqrt{n}\Pmat_k (\yvec - \bmu) - \frac{1}{A_k}\avec_k\sqrt{n}(\hat C_{k+1} - C_{k+1}) 
     = \Pmat_k\berror_n - \frac{1}{A_k}\avec_k \frac{s_{k+1}\evec_{k+1}^T\berror_n}{1 - s_{k+1}\xvec_{j_{k+1}}^T\avec_k/A_k},
\end{align*}
where we have used $\yvec = \bmu + n^{-1/2}\berror_n$. From here we may re-write \eqref{eqn:rootnbk} as
\[
\sqrt{n}((\hat \bvec_k)_{\hat\calA_k}  - (\bvec_k)_{\hat \calA_k}) = (\Xmat_{\calA_k}^T\Xmat_{\calA_k})^{-1}\Big[\Xmat_{\calA_k} - \Big(\frac{s_{k+1}}{1-s_{k+1}\xvec_{j_{k+1}}^T\avec_k/A_k}\Big)\evec_{k+1}\svec_k^T\Big]^T\berror_n,
\]
using $A_k^{-1}\avec_k = \Xmat_k(\Xmat_k^T\Xmat_k)^{-1}\onevec = \Xmat_{\calA_k}(\Xmat_{\calA_k}^T\Xmat_{\calA_k})^{-1}\svec_k$. Then
\begin{align} \nonumber
\Var(\sqrt{n}&((\hat \bvec_k)_{\hat\calA_k}  - (\bvec_k)_{\hat \calA_k})) \\ \label{eqn:varbk}
&= \sigma^2(\Xmat_{\calA_k}^T\Xmat_{\calA_k})^{-1}\Big[\Xmat_{\calA_k}^T\Xmat_{\calA_k} + \frac{\evec_{k+1}^T\evec_{k+1}}{(1-s_{k+1}\xvec_{j_{k+1}}^T\avec_k/A_k)^2} \svec_k\svec_k^T\Big](\Xmat_{\calA_k}^T\Xmat_{\calA_k})^{-1},
\end{align}
where we have used $\Xmat_{\calA_k}^T\evec_{k+1} = \zerovec$ and $\Var(\berror_n) = \sigma^2\Imat_n$. Now under Assumption \ref{assum:regassumptions}, as $n \to \infty$, $\Xmat_{\calA_k}^T\Xmat_{\calA_k} \to \Rmat_{\calA_k,\calA_k}$, $\evec_{k+1}^T\evec_{k+1} \to \Rmat_{j_{k+1}|\calA_k}$, and $\xvec_{j_{k+1}}^T\avec_k/A_k \to \Rmat_{j_{k+1},\calA_k}\Rmat_{\calA_k,\calA_k}^{-1}\svec_k$. Substituting these limits in \eqref{eqn:varbk} yields the covariance matrix in \eqref{eqn:covvk} for $k<m$.
For $k = m$ we have 
\[
\sqrt{n}((\hat \bvec_m)_{\hat\calA_m}  - (\bvec_m)_{\hat \calA_m}) = (\Xmat_{\calA_m}^T\Xmat_{\calA_m})^{-1}\Xmat_{\calA_m}^T \sqrt{n}(\yvec - \bmu_m) = (\Xmat_{\calA_m}^T\Xmat_{\calA_m})^{-1}\Xmat_{\calA_m}^T\berror_n,
\]
since we have re-defined $\hat \bvec_{\bar m}$ as $\hat \bvec_{\bar m} \equiv\bar \bvec_{\bar m} = (\Xmat_{\calA_{\bar m}}^T\Xmat_{\calA_{\bar m}})^{-1}\Xmat_{\calA_{\bar m}}^T\yvec$ and we assume $\bar m = m$ holds and since $\bmu_m = \bmu$ by the assumption (R4) that $\bmu$ lies in the column space of $\Xmat$. Now we have
$\Var(\sqrt{n}((\hat \bvec_m)_{\hat\calA_m}  - (\bvec_m)_{\hat \calA_m})) = \sigma^2 (\Xmat_{\calA_m}^T\Xmat_{\calA_m})^{-1} \to \sigma^2 \Rmat_{\calA_m,\calA_m}^{-1}$,
as $n \to \infty$, which matches the expression in \eqref{eqn:covvk} for $k=m$.
Lastly, for any steps $k,k'$ such that $1 \leq k < k'\leq m$ we may write
\begin{align*}
    \Cov(&\sqrt{n}((\hat \bvec_k)_{\hat\calA_k}  - (\bvec_k)_{\hat \calA_k}),\sqrt{n}((\hat \bvec_{k'})_{\hat\calA_{k'}}  - (\bvec_{k'})_{\hat \calA_{k'}})) \\
    &= \sigma^2(\Xmat_{\calA_k}^T\Xmat_{\calA_k})^{-1}\Big[\Xmat_{\calA_k}^T\Xmat_{\calA_k} - \frac{s_{k+1}}{(1-s_{k+1}\xvec_{j_{k+1}}^T\avec_k/A_k)} \svec_k\evec_{k+1}^T\Xmat_{\calA_{k'}}\Big](\Xmat_{\calA_{k'}}^T\Xmat_{\calA_{k'}})^{-1},
\end{align*}
using $\Xmat_{\calA_k}^T\evec_{k'+1} = \zerovec$ and $\evec_{k+1}^T\evec_{k'+1} = 0$. With $\evec_{k+1}^T\Xmat_{\calA_{k'}} \to \Rmat_{j_{k+1},\calA_{k'}|\calA_k}$ as $n \to \infty$, we see that the covariance approaches the expression in \eqref{eqn:covvkvk} as $n \to \infty$. 
\end{proof}

\begin{proof}{\bf of Theorem \ref{thm:bootstrap}:}
For context, recall that the original data responses $\yvec\equiv \bmu + \berror$ and design matrix $\Xmat$ are used as inputs to the LAR sample path (i.e., $\LAR(\Xmat,\yvec)$ from Algorithm~\ref{alg:lar_efron}) which estimates the LAR population path (i.e., $\LAR(\Xmat,\bmu)$ from Algorithm~\ref{alg:lar}).  
Recall also that bootstrap responses are defined as $\yvec^* \equiv \bar \bmu + \berror^*$, where the  bootstrap population mean  $\bar \bmu  \equiv \hat \Pmat_{\bar m} \yvec$   is the orthogonal projection matrix given by the columns $\{\xvec_j, j \in \hat \calA_{\bar m}\}$ active in the sample path at an estimated step $\bar m$.  With this background, the critical step in establishing the validity of bootstrap
is showing that the bootstrap analog  
$\LAR(\Xmat,\bar \bmu)$  is consistent for the structure of the original $\LAR(\Xmat,\bmu)$ population path, where the latter is   defined fundamentally  by a sequence of $m$ variable indices $j_1,\dots,j_m$ that  enter the active set $\calA_k$
of $\LAR(\Xmat,\bmu)$ in a one-by-one consecutive fashion on  steps $k=1,\dots,m$, respectively, 
with corresponding signs $s_1,\ldots,s_m$
and where   $m \leq p$ denotes the number of positive population step correlations  (i.e., $C_1>C_2>\ldots>C_m>0$) so that $C_k=0$ for any $m < k \leq p$.  Note that (i) the number $m$ of non-zero step correlations, (ii) the variable order over the first $m$ steps, and (iii) the variable signs are the attributes of 
 the population $\LAR(\Xmat,\bmu)$ path  that serve to completely determine the intricate limit distribution of the  statistical quantities $(T_{n 1},\ldots, T_{n p})$ 
 in (\ref{eqn:T}) by Theorem~\ref{thm:Tasymp};
the latter distribution essentially gives  the studentized sampling distribution $(\hat T_{n1},\ldots, \hat T_{np}) \equiv (  T_{n1},\ldots,  T_{np})/\hat \sigma_n $ that we seek to approximate by bootstrap (using that $\hat \sigma_n$ is a constant in large samples by $\hat \sigma_n \inprob \sigma>0$ as in the proof of Theorem~\ref{thm:mest}). Note that the exact values of the first $m$ population step correlations $C_1>\ldots>C_m>0$ are
not important to this limit distribution in Theorem~\ref{thm:Tasymp}, though these values are used in defining the centering 
of the  quantities $\{T_{nk}\}_{k=1}^m$ 
 in (\ref{eqn:T}).  Hence, in order to formally establish the validity of the bootstrap approximation, it then becomes necessary to show that the bootstrap rendition $\LAR(\Xmat,\bar\bmu)$ 
of the original LAR population path likewise shares the  structure of  consisting of (i) the same number $m$ of non-zero step correlations, (ii) the   same variable order $j_1,\dots,j_m$ over the first $m$ steps, and (iii) the same variable signs 
$s_1,\ldots,s_m$, where these features must technically hold with arbitrarily high probability for large $n$; furthermore, when this structure holds for the bootstrap and if $\bar C_1 >\ldots >\bar C_m>0=\bar C_{m+1}$ denote the corresponding bootstrap step correlations  from  $\LAR(\Xmat,\bar\bmu)$, it must also be shown
that $\bar C_k = \hat{C}_k$ over steps $k=1,\ldots,m$, where $\hat{C}_k$ denote the sample step  correlations from $\LAR(\Xmat,\yvec)$, due to the fact that $\bar C_k = \hat{C}_k$ is the centering used in defining the bootstrap analog quantities  $\{\hat T_{nk}^*\}_{k=1}^m$
over  non-zero 
step correlations.   
To show $\LAR(\Xmat,\bar\bmu)$  has the above properties, 
suppose first that the event 
in (\ref{eqn:L_m}) and the event that $\bar m =m$ both hold, which occurs with arbitrarily high probability for large $n$. These combined events entail that the original sample path of $\LAR(\Xmat,\yvec)$  admits variables $j_1,\dots,j_m$  one-by-one consecutively over the first $m$ steps with variable signs 
$s_1,\ldots,s_m$, respectively.  Consequently, it further holds that 
$\bar\bmu = \Pmat_m \yvec$ by $\hat \Pmat_{\bar m}= \Pmat_m$ and that the first $m$ sample step correlations from $\LAR(\Xmat,\yvec)$ are given by
\begin{equation}
\label{eqn:djn1}
\hat{C}_k = \frac{s_k \evec_{k}^T \yvec}{ (1 - s_k \xvec_{j_k}^T\avec_{k-1} / A_{k-1})}, \quad k=1,\ldots,m,
\end{equation}
where  $\evec_k\equiv  (\Imat - \Pmat_{k-1})\xvec_{j_k} $  by using (\ref{eqn:Ckangles}) of Lemma~\ref{lem:Ck}.
From the bootstrap path $\LAR(\Xmat,\bar\bmu = \Pmat_m \yvec)$, next let $\bar \bmu_k$ 
denote the approximation of $\bar \bmu$ at step $k=1,\ldots,m$ from Algorithm~\ref{alg:lar}  (i.e., applied to $\bar \bmu$ in place of $\bmu$)  where $\bar \bmu_0=\zerovec$ and, likewise,  let  $\bar C_k \equiv \max\{|(\bar \cvec_k)_j|\}$ and $\bar \calA \equiv \{j :|(\bar \cvec_k)_j| = \bar C_k \}$ , where $\bar \cvec_k \equiv \Xmat ^T(\bar \bmu - \bar \yvec_{k-1})$; further, let $\bar s_k$ denote the sign of the column of $\Xmat$ entering on step $k=1,\ldots,m$.
Now repeating the proof of Theorem~\ref{thm:larconsistency}  with $\bar \bmu = \Pmat_m \yvec = \bmu + \Pmat_m \berror$ (i.e.,  $\bmu =  n^{-1/2} \Xmat_n \bbeta_n = \Pmat_m \bmu$ here) in place of $\yvec = \bmu + \berror$, one finds that the same conclusions of Theorem~\ref{thm:larconsistency}(i)-(iii) hold upon replacing instances of ``$\hat{C}_k, \hat \cvec_k, \hat \calA_{k}, \hat{s}_k$" with analogs  
 ``$\bar{C}_k, \bar \cvec_k,\bar \calA_{k}, \bar s_k$" (that is, the proof does not essentially change upon replacing errors $\berror$ with $\Pmat_m \berror$, where the latter errors have a variance matrix of smaller spectral norm compared to the former errors).  From this, we  
can conclude that, with arbitrarily high probability for increasing $n$, it holds that  the path $\LAR(\Xmat,\bar\bmu)$ 
admits the   same variable order $j_1,\dots,j_m$ over the first $m$ steps with the same signs $s_1,\ldots,s_k$  as the original population path given by $\LAR(\Xmat,\bmu)$ or as the sample path 
given by $\LAR(\Xmat,\yvec)$;  when this feature holds,  we can determine that 
\begin{equation}
\label{eqn:djn2}
\bar{C}_k = \frac{s_k \evec_{k}^T  \bar\bmu}{ (1 - s_k \xvec_{j_k}^T\avec_{k-1} / A_{k-1})}=\frac{s_k \evec_{k}^T \yvec}{ (1 - s_k \xvec_{j_k}^T\avec_{k-1} / A_{k-1})} = \hat{C}_k, \quad k=1,\ldots,m,
\end{equation}
where  $\evec_k\equiv  (\Imat - \Pmat_{k-1})\xvec_{j_k} $,  by using (\ref{eqn:Ckangles}) of Lemma~\ref{lem:Ck}
for $\bar \bmu= \Pmat_m \yvec$ along with $\evec_{k}^T  \bar\bmu =\evec_{k}^T  \yvec $ for $k=1,\ldots,m$ and (\ref{eqn:djn1}).  
Likewise,  it must similarly follow that $\bar C_{m+1}=0$ or that  path $\LAR(\Xmat,\bar\bmu)$ has  $m$ non-zero step correlations; to see the latter, if $\bar C_{m+1}>0$ were true then  $\Xmat^T(\Imat - \Pmat_m) \bar \bmu $ would need to have a non-zero entry by  Lemma~\ref{lem:Ck} 
or Lemma~\ref{lem:whichtoenter}(ii), but 
 $\Xmat^T(\Imat - \Pmat_m) \bar \bmu =\zerovec$ by $\bar \bmu= \Pmat_m \yvec$.
We have now shown that the bootstrap rendition $\LAR(\Xmat,\bar\bmu)$, with arbitrarily high probability for increasing $n$,   has the required properties of (i) $m$ of non-zero step correlations, (ii) the  appropriate variable order $j_1,\dots,j_m$ over the first $m$ steps, and (iii) the proper variable signs 
$s_1,\ldots,s_m$, where additionally it holds that $\bar C_k = \hat C_k$ for each $k=1,\ldots,m$.  
 To complete proof of bootstrap distributional consistency, we use an argument to characterize convergence in probability through pointwise almost sure convergence along subsequences.  
Let $(\Omega, \mathcal{F}, \PP)$ denote a probability space for the original random vectors $(\Xmat_n, \berror_n)$, $n \geq 1$, in Assumption~\ref{assum:regassumptions}.  Let $\{n_j\}_{j \geq 1}$ denote an arbitrary subsequence of $\{n\}_{n \geq 1}$.
Now using the assumption that the bootstrap errors $\berror_n^*$ satisfy (R3) in probability, along with the convergence in probability given in Theorem~\ref{thm:larconsistency} and Theorem~\ref{thm:mest} and developed above for the bootstrap path $\LAR(\Xmat,\bar \bmu )$, we may extract a further subsequence $\{n_\ell\}_{\ell \geq 1}$ of $\{n_j\}_{j \geq 1}$  and an event $A \in \mathcal{F}$ of probability one (i.e., $\PP(A)=1$) such that, for any given point $\omega \in A $, it holds pointwise along the subsequence  $(\Xmat_{n_\ell}(\omega), \berror_{n_\ell}(\omega)) \equiv (\Xmat_{n_\ell}, \berror_{n_\ell}) $, $n_\ell \geq 1$,  (suppressing dependence of notation on $\omega$ and sometimes $n_\ell$ in the following without loss of clarity) that:   (a) eventually for large $n_\ell$,  it holds that  
$\bar m = m$; the variables $j_1,\ldots,j_m$ with signs $s_1,\ldots,s_m$ enter stepwise in both $\LAR(\Xmat, \yvec)$
and $\LAR(\Xmat,\bar \bmu)$ for steps $k=1,\ldots,m$; and $\hat{C}_k=\bar C_k$ for each $k =1,\ldots,m$ while $\bar C_{m+1}=0$; (b) also, $\max_{1 \leq k \leq m}|\bar C_k - C_k| \rightarrow 0$,  $\max_{1 \leq k \leq m}\|\bar \cvec_k - \cvec_k\|\rightarrow 0$, and $\hat \sigma_{n_\ell}\rightarrow \sigma>0$   as $n_\ell \to \infty$; and (c) in the bootstrap world (or in terms of probability induced by resampling),     
$n_\ell^{-1/2}\Xmat_{n_\ell}^T\berror^*_{n_\ell} \indist \calN(\zerovec,\bSigma \sigma^2)$ and $ n_\ell^{-1}\berror_{n_\ell}^{*T}\berror_{n_\ell}^* \inprob \sigma^2$ as $n_\ell \to \infty$.  Part~(b)  ensures that, for a given $\omega \in A$, Assumptions (M1)-(M2) hold also.  Hence,    pointwise for $\omega \in A$, we have all the assumptions in place to apply the same proof of Theorem~\ref{thm:Tasymp}  to show that the bootstrap 
analog $(T_{n_\ell 1}^*,\ldots, T_{n_\ell p}^*)$, where 
\[
\hat T_{n_\ell k}^* \equiv \hat s_k^*\bigg(\frac{1}{\hat A_{k}^{*2}}- \frac{1}{\hat A_{k - 1}^{*2}}\bigg)^{1/2}\sqrt{n_\ell}\big( \hat C^*_k - \bar C_k \big), \quad k = 1,\dots,p,
\]
has  the same distributional limit as $n_\ell \to \infty$
as  the  limit 
given in   Theorem~\ref{thm:Tasymp} for the   quantities 
$(T_{n 1},\ldots, T_{n p})$ 
 from (\ref{eqn:T}).   
Additionally, along a given $\omega \in A$, we also have 
 $\hat \sigma_{n_\ell}^{*2} \equiv ({n_\ell}-p)^{-1} \yvec_{n_\ell}^{*T}(\Imat - \Pmat_\Xmat)\yvec_{n_\ell}^* \equiv ({n_\ell}-p)^{-1} \berror_{n_\ell}^{*T}(\Imat - \Pmat_\Xmat)\berror_{n_\ell}^*
\inprob \sigma^2$ as $n_\ell \to \infty$  due to the facts that $ n_\ell^{-1}\berror_{n_\ell}^{*T}\berror_{n_\ell}^* \inprob \sigma^2$ as $n_\ell \to \infty$  in addition to    $|({n_\ell}-p) {n_\ell}^{-1}\hat \sigma_{n_\ell}^{*2} -  {n_\ell}^{-1}\berror_{n_\ell}^{*T} \berror_{n_\ell}^*| = {n_\ell}^{-1}\berror_{n_\ell}^{*T} \Pmat_{\Xmat} \berror_{n_\ell}^* \inprob 0$  because ${n_\ell}^{-1}\E_*(\berror_{n_\ell}^T \Pmat_{\Xmat} \berror_{n_\ell}) = p \hat{\sigma}_{n_\ell}^2/{n_\ell} \rightarrow 0$  as  $n_\ell \to \infty$ in bootstrap expectation $\E_*$
using $\hat{\sigma}_{n_\ell}^2
\rightarrow \sigma^2$.
Consequently,  pointwise for $\omega \in A$,  the bootstrap quantities $(\hat T^*_{n_\ell 1},\ldots, \hat T^*_{ n_\ell p}) \equiv (\hat T^*_{{n_\ell}1},\ldots, \hat T^*_{{n_\ell}p})/\sigma_{n_\ell}^*$ 
will converge in distribution as $n_\ell \to \infty$ to a (continuous) distributional limit that matches that of 
  $(\hat T_{n 1},\ldots, \hat T_{n p}) \equiv (  T_{n 1},\ldots,  T_{n p})/\hat \sigma_n$ at the original data level; in other words,
  pointwise for $\omega \in A$,  we have
  \[
\sup_{\calB \in \calB(\mathbb{R}^p)} \Big|\PP_*\big((\hat T^*_{n_\ell 1},\dots,\hat T^*_{n_\ell p}) \in \calB\big) -  \PP\big((\hat T_{n_\ell 1},\dots,\hat T_{n_\ell p})\in \calB\big ) \Big| \rightarrow 0 \quad \mbox{as $n_\ell\to \infty$}.
\] 
Because the last convergence holds with probability 1 along a subsequence $\{n_\ell\}_{\ell \geq 1}$  of $\{n_j\}_{j \geq 1}$ and because $\{n_j\}_{j \geq 1}$
was an arbitrarily chosen subsequence of $\{n\}_{n \geq 1}$, the probabilistic convergence in Theorem~\ref{thm:bootstrap}
now follows.
\end{proof}
% \begin{theorem}
% Under the conditions of Theorem \ref{thm:larconsistency} and with $\bar m$ estimated as in \eqref{eqn:mest}, for $k = 1,\dots,m$ we have
% \[
% \sqrt{n}((\hat \bvec_k)_{\hat \calA_k}  - ( \bvec_k)_{\hat \calA_k}) \indist \calN(\zerovec,\Vmat_{k,k})
% \]
% as $n \to \infty$, where $\Vmat_{m,m} = \Rmat^{-1}_{\calA_m,\calA_m}$ and
% \[
% \Vmat_{k,k} = \Rmat^{-1}_{\calA_k,\calA_k}\Big[\Rmat_{\calA_k,\calA_k} + \frac{\Rmat_{\{j_{k+1}\}|\calA_k}}{(1 - s_{k+1}\Rmat_{j_{k+1},\calA_k}\Rmat_{\calA_k,\calA_k}^{-1}\svec_k)^2}\svec_k\svec_k^T\Big]\Rmat^{-1}_{\calA_k,\calA_k}
% \]
% for $k = 1,\dots,m-1$, where $\svec_k$ is the vector of signs $\svec_k=(s_1,\dots,s_k)^T$.
% \end{theorem}

% \begin{theorem}
% Under the conditions of Theorem \ref{thm:larconsistency} for $k = 1,\dots,m$ we have
% \[
% \sqrt{n}((\hat \bvec_k)_{\hat \calA_k}  - ( \bvec_k)_{\hat \calA_k}) \indist \calN\Big(\zerovec,\sigma^2\Rmat^{-1}_{\calA_k,\calA_k}\Big[\Rmat_{\calA_k,\calA_k} + \lambda_k\svec_k\svec_k^T\Big]\Rmat^{-1}_{\calA_k,\calA_k}\Big)
% \]
% as $n \to \infty$, where 
% \[
% \lambda_k \equiv \frac{\Rmat_{\{j_{k+1}\}|\calA_k}}{(1 - s_{k+1}\Rmat_{\{j_{k+1}\},\calA_k}\Rmat_{\calA_k,\calA_k}^{-1}\svec_k)^2} > 0
% \]
% for $k = 1,\dots,m-1$, $\lambda_m \equiv 0$, and  $\svec_k$ is the vector of signs $\svec_k=(s_1,\dots,s_k)^T$.
% \end{theorem}

\section{Additional data analysis plots}
\label{app:illu}

Here we present some additional plots related to the analysis of the diabetes data analyzed in Section \ref{sec:diabetes}.  Figure \ref{fig:diabetes_larinf_lar} shows the sample LAR path on the data; this figure presents the same information as Figure 3 of \cite{efron2004least}, though our step correlations and coefficients are scaled differently, and we plot the step coefficients against the step number rather than against the sum of the absolute values of the coefficients.
Figure \ref{fig:diabetes_larinf_larinf2} depicts in the left panel the estimation of $m$ via \eqref{eqn:mest}, which results in $\bar m = 5$ and in the right panel the bootstrap probability of active set membership of each variable index over the LAR steps. In the sample path the variable bmi enters the active set on step 1 in front of the variable ltg; however, in bootstrap sampling, bmi enters on step 1 about 75\% of the time with ltg entering first the other 25\% of the time. Note also that the tch variable (gray dashed line), though entering on step 7 in the sample path, is seen to enter as early as steps 3 or 4 about 5\% and 25\% of the time, respectively.  This suggests that the separation condition (M2) may not hold for $\delta$ large enough to guarantee that the sample LAR path will perfectly recover the sequence of active sets in the underlying population path. 

\begin{figure}[ht]
\begin{center}
\includegraphics[width = 0.8\textwidth]{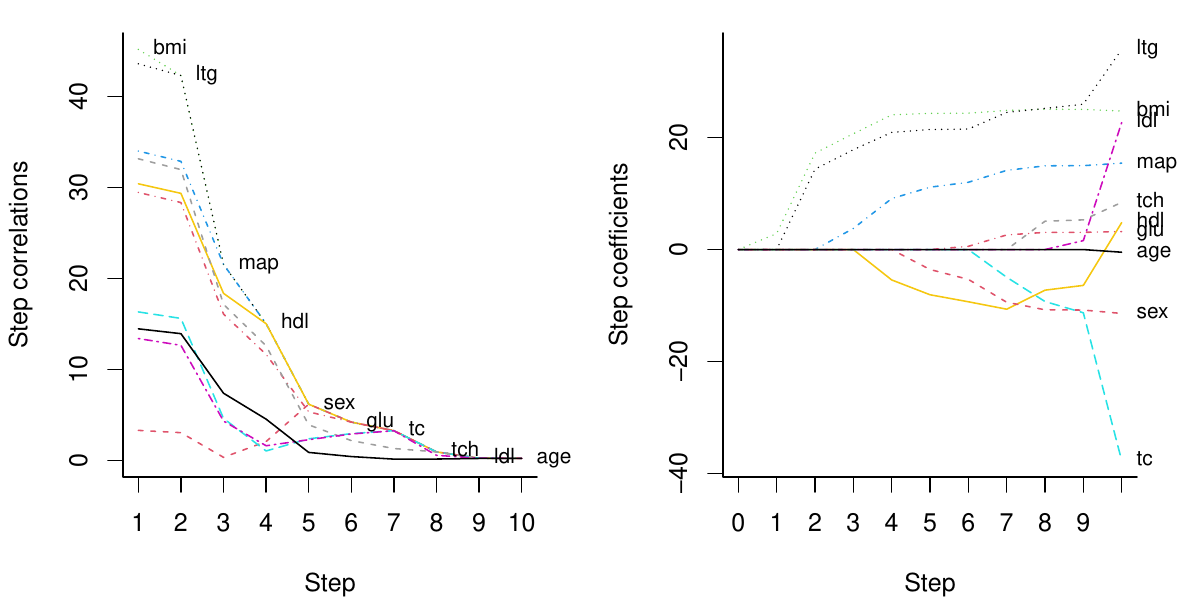}
\caption{Sample path $\LAR(\Xmat,\yvec)$ on the diabetes data described in Section \ref{sec:diabetes}.}
\label{fig:diabetes_larinf_lar}
\end{center}
\end{figure}

\begin{figure}[ht]
\begin{center}
\includegraphics[width = 0.8\textwidth]{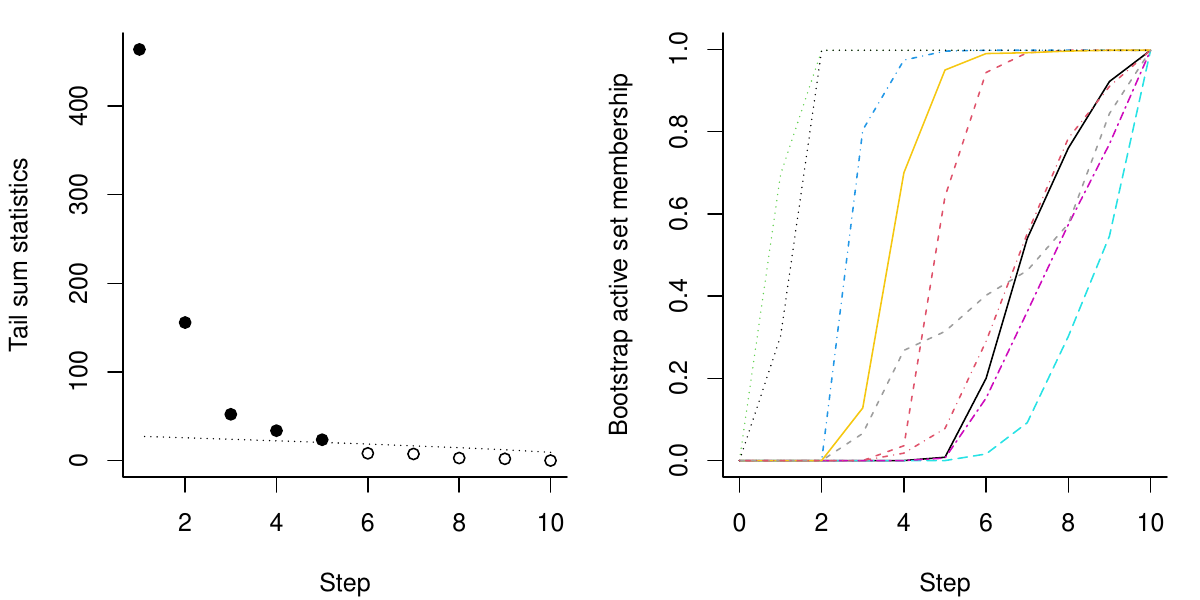}
\caption{Diagnostic plots for the diabetes data described in Section \ref{sec:diabetes}.}
\label{fig:diabetes_larinf_larinf2}
\end{center}
\end{figure}

% Note: in this sample, the section number is hard-coded in. Following
% proper LaTeX conventions, it should properly be coded as a reference:

%In this appendix we prove the following theorem from
%Section~\ref{sec:textree-generalization}:

\vskip 0.2in
\bibliography{lit.bib}

@book{RV2025,
  title={Modern multivariate
statistical learning},
  author={Ryan, Kenneth J. and Vardeman, Stephen, B.},
  year={2025},
  publisher={Textbook draft}
}

@article{lee2016exact,
  title={Exact post-selection inference, with application to the lasso},
  author={Lee, Jason D and Sun, Dennis L and Sun, Yuekai and Taylor, Jonathan E and others},
  journal={The Annals of Statistics},
  volume={44},
  number={3},
  pages={907--927},
  year={2016},
  publisher={Institute of Mathematical Statistics}
}

@book{mammen2012does,
  title={When does bootstrap work?: asymptotic results and simulations},
  author={Mammen, Enno},
  volume={77},
  year={2012},
  publisher={Springer Science \& Business Media}
}

@article{freedman1981bootstrapping,
  title={Bootstrapping regression models},
  author={Freedman, David A},
  journal={The Annals of Statistics},
  pages={1218--1228},
  year={1981},
  publisher={JSTOR}
}

@article{giurcanu2019bootstrapping,
  title={Bootstrapping LASSO-type estimators in regression models},
  author={Giurcanu, Mihai and Presnell, Brett},
  journal={Journal of Statistical Planning and Inference},
  volume={199},
  pages={114--125},
  year={2019},
  publisher={Elsevier}
}

@article{bickel1983bootstrapping,
  title={Bootstrapping regression models with many parameters},
  author={Bickel, Peter J and Freedman, David A},
  journal={Festschrift for Erich L. Lehmann},
  pages={28--48},
  year={1983}
}

@article{lai1979strong,
  title={Strong consistency of least squares estimates in multiple regression II},
  author={Lai, T L and Robbins, Herbert and Wei, C Zi},
  journal={Journal of multivariate analysis},
  volume={9},
  number={3},
  pages={343--361},
  year={1979},
  publisher={Academic Press}
}

@article{efron2004least,
  title={Least angle regression},
  author={Efron, Bradley and Hastie, Trevor and Johnstone, Iain and Tibshirani, Robert and others},
  journal={The Annals of Statistics},
  volume={32},
  number={2},
  pages={407--499},
  year={2004},
  publisher={Institute of Mathematical Statistics}
}

@article{tibshirani1996regression,
  title={Regression shrinkage and selection via the lasso},
  author={Tibshirani, Robert},
  journal={Journal of the Royal Statistical Society: Series B (Methodological)},
  volume={58},
  number={1},
  pages={267--288},
  year={1996},
  publisher={Wiley Online Library}
}

@article{chatterjee2011bootstrapping,
  title={Bootstrapping lasso estimators},
  author={Chatterjee, Arindam and Lahiri, Soumendra Nath},
  journal={Journal of the American Statistical Association},
  volume={106},
  number={494},
  pages={608--625},
  year={2011},
  publisher={Taylor \& Francis}
}

@article{lockhart2014significance,
  title={A significance test for the lasso},
  author={Lockhart, Richard and Taylor, Jonathan and Tibshirani, Ryan J and Tibshirani, Robert},
  journal={Annals of Statistics},
  volume={42},
  number={2},
  pages={413},
  year={2014},
  publisher={NIH Public Access}
}

@article{wang2021infrared,
  title={Infrared thermography for measuring elevated body temperature: clinical accuracy, calibration, and evaluation},
  author={Wang, Quanzeng and Zhou, Yangling and Ghassemi, Pejman and McBride, David and Casamento, Jon P and Pfefer, T Joshua},
  journal={Sensors},
  volume={22},
  number={1},
  pages={215},
  year={2021},
  publisher={MDPI}
}

@misc{wang2023facial,
  title={Facial and oral temperature data from a large set of human subject volunteers},
  author={Wang, Quanzeng and Zhou, Yangling and Ghassemi, Pejman and Chenna, Dwith and Chen, Michelle and Casamento, Jon and Pfefer, Joshua and Mcbride, David},
  year={2023},
  publisher={PhysioNet}
}

@article{PhysioNet,
 author = "Goldberger, A. L. and Amaral, L. A. N. and Glass, L. and
	   Hausdorff, J. M. and Ivanov, P. Ch. and Mark, R. G. and
	   Mietus, J. E. and Moody, G. B. and Peng, C.-K. and Stanley, H. E.",
 title = "{PhysioBank, PhysioToolkit, and PhysioNet}: Components of a New
	  Research Resource for Complex Physiologic Signals",
 journal = "Circulation [Online]",
 year = "2000",
 volume = "101",
 number = "23",
 pages = "e215--e220"}

@article{tibshirani2015general,
  title={A general framework for fast stagewise algorithms.},
  author={Tibshirani, Ryan J},
  journal={J. Mach. Learn. Res.},
  volume={16},
  number={1},
  pages={2543--2588},
  year={2015}
}

@article{zhao2006model,
  title={On model selection consistency of Lasso},
  author={Zhao, Peng and Yu, Bin},
  journal={The Journal of Machine Learning Research},
  volume={7},
  pages={2541--2563},
  year={2006},
  publisher={JMLR. org}
}

@article{hesterberg2008least,
author = {Tim Hesterberg and Nam Hee Choi and Lukas Meier and Chris Fraley},
title = {{Least angle and $\ell_1$ penalized regression: A review}},
volume = {2},
journal = {Statistics Surveys},
number = {none},
publisher = {Amer. Statist. Assoc., the Bernoulli Soc., the Inst. Math. Statist., and the Statist. Soc. Canada},
pages = {61 -- 93},
year = {2008}
}

@article{taylor2014post,
  title={Post-selection adaptive inference for least angle regression and the lasso},
  author={Taylor, Jonathan and Lockhart, Richard and Tibshirani, Ryan J and Tibshirani, Robert},
  journal={arXiv preprint arXiv:1401.3889},
  volume={354},
  year={2014},
  publisher={Citeseer}
}

@article{khan2007robust,
  title={Robust linear model selection based on least angle regression},
  author={Khan, Jafar A and Van Aelst, Stefan and Zamar, Ruben H},
  journal={Journal of the American Statistical Association},
  volume={102},
  number={480},
  pages={1289--1299},
  year={2007},
  publisher={Taylor \& Francis}
}

@article{hastie2007forward,
author = {Trevor Hastie and Jonathan Taylor and Robert Tibshirani and Guenther Walther},
title = {{Forward stagewise regression and the monotone lasso}},
volume = {1},
journal = {Electronic Journal of Statistics},
number = {none},
publisher = {Institute of Mathematical Statistics and Bernoulli Society},
pages = {1 -- 29},
year = {2007}
}

@article{tibshirani2013lasso,
author = {Ryan J. Tibshirani},
title = {{The lasso problem and uniqueness}},
volume = {7},
journal = {Electronic Journal of Statistics},
number = {none},
publisher = {Institute of Mathematical Statistics and Bernoulli Society},
pages = {1456 -- 1490},
year = {2013}
}

@article{azais2018power,
author = {Jean-Marc Aza{\"i}s and Yohann De Castro and St{\'e}phane Mourareau},
title = {{Power of the spacing test for least-angle regression}},
volume = {24},
journal = {Bernoulli},
number = {1},
publisher = {Bernoulli Society for Mathematical Statistics and Probability},
pages = {465 -- 492},
year = {2018}
}

@article{su2018first,
  title={When is the first spurious variable selected by sequential regression procedures?},
  author={Su, Weijie J},
  journal={Biometrika},
  volume={105},
  number={3},
  pages={517--527},
  year={2018},
  publisher={Oxford University Press}
}

@article{g2016sequential,
  title={Sequential selection procedures and false discovery rate control},
  author={G'Sell, Max Grazier and Wager, Stefan and Chouldechova, Alexandra and Tibshirani, Robert},
  journal={Journal of the Royal Statistical Society Series B: Statistical Methodology},
  volume={78},
  number={2},
  pages={423--444},
  year={2016},
  publisher={Oxford University Press}
}

@Manual{larspackage,
    title = {lars: Least Angle Regression, Lasso and Forward Stagewise},
    author = {Trevor Hastie and Brad Efron},
    year = {2022},
    note = {R package version 1.3},
    url = {https://CRAN.R-project.org/package=lars},
  }

@article{inglot2010inequalities,
  title={Inequalities for quantiles of the chi-square distribution},
  author={Inglot, Tadeusz},
  journal={Probability and Mathematical Statistics},
  volume={30},
  number={2},
  pages={339--351},
  year={2010},
  publisher={Wroclaw University Press}
}

@article{azais2022multiple,
  title={Multiple testing and variable selection along the path of the least angle regression},
  author={Aza{\"\i}s, Jean-Marc and De Castro, Yohann},
  journal={Information and Inference: A Journal of the IMA},
  volume={11},
  number={4},
  pages={1329--1388},
  year={2022},
  publisher={Oxford University Press}
}

\end{document}